\pdfoutput=1
\RequirePackage{ifpdf}
\ifpdf 
\documentclass[pdftex]{sigma}
\else
\documentclass{sigma}
\fi

\usepackage{tikz-cd}
\usepackage{mathtools}
\usepackage{dsfont}
\newtheorem{thm}{Theorem}[section]
\newtheorem{lem}[thm]{Lemma}

{\theoremstyle{definition}
\newtheorem{defn}[thm]{Definition}
\newtheorem{Remark}[thm]{Remark}}

\numberwithin{equation}{section}
\DeclareMathOperator{\Res}{Res}

\newcommand{\inv}{^{-1}}
\newcommand{\abs}[1]{\lvert#1\rvert}
\newcommand{\fermion}[3]{{\prescript{}{#1}\psi_{(#2)}^{#3}}}

\begin{document}

\allowdisplaybreaks

\newcommand{\arXivNumber}{1605.00192}

\renewcommand{\PaperNumber}{023}

\FirstPageHeading

\ShortArticleName{$\tau$-Functions, Birkhoff Factorizations and Difference Equations}

\ArticleName{$\boldsymbol{\tau}$-Functions, Birkhoff Factorizations\\ and Difference Equations}

\Author{Darlayne ADDABBO~$^\dag$ and Maarten BERGVELT~$^\ddag$}

\AuthorNameForHeading{D.~Addabbo and M.~Bergvelt}

\Address{$^\dag$~Department of Mathematics, University of Notre Dame, Notre Dame, IN 46556, USA}
\EmailD{\href{mailto:daddabbo@nd.edu }{daddabbo@nd.edu}}

\Address{$^\ddag$~Department of Mathematics, University of Illinois, Urbana-Champaign, IL 61801, USA}
\EmailD{\href{mailto: bergv@illinois.edu }{bergv@illinois.edu}}

\ArticleDates{Received July 24, 2018, in final form March 05, 2019; Published online March 27, 2019}

\Abstract{$Q$-systems and $T$-systems are systems of integrable difference equations that have recently attracted much attention, and have wide applications in representation theory and statistical mechanics. We show that certain $\tau$-functions, given as matrix elements of the action of the loop group of ${\rm GL}_{2}$ on two-component fermionic Fock space, give solutions of a~$Q$-system. An obvious generalization using the loop group of ${\rm GL}_3$ acting on three-component fermionic Fock space leads to a new system of 4 difference equations.}


\Classification{17B80}

\section{Introduction}\label{sec:introduction}

Many integrable differential equations can be transformed {to} simpler, bilinear form by introducing new dependent variables called $\tau$-functions. In practice, these $\tau$-functions are given as matrix elements of infinite-dimensional groups or Lie algebras, etc.

For instance, the famous Korteweg--de Vries (KdV) equation
\begin{gather}
u_t+u_{xxx}+6 uu_x=0\label{eq:85}
\end{gather}
is transformed by the substitution
\begin{gather*}
u=2\ln(\tau)_{xx}
\end{gather*}
into Hirota bilinear form
\begin{gather*}
\big(D_{x}D_{t}+D_{x}^{4}\big)\tau\cdot\tau=0,
\end{gather*}
where $D_{u}$ is the Hirota operator so that $D_{u}\sigma\cdot\tau=\sigma_u\tau-\sigma\tau_u$. See~\cite{MR2085332} for details and many more examples. In the case of the KdV equation, the $\tau$-function is a matrix element for the action of the loop group of ${\rm GL}_{2}$ on one-component fermionic Fock space, see for instance~\cite{MR86a:58093, MR1021978,MR1736222}.

To produce the integrable equations from $\tau$-functions, one introduces an intermediate object, the Baker function. It satisfies linear equations, and the compatibility of these equations gives the integrable hierarchy.

For instance, for the KdV case, the $\tau$-function is a scalar function $\tau(t_{1},t_{3},t_{5},\dots)$ of odd times ($t_{1}=x,t_{3}=t$), and the Baker function is also a scalar function $\Psi$ of the times $t_{2k+1}$ and of an extra variable~$z$, the spectral parameter. It is defined by
\begin{gather*}
\Psi(z;t_{1},t_{3},\dots)=(\Gamma(z,t)\circ\tau(t))/\tau(t),
\end{gather*}
where $\Gamma(z,t)={\rm e}^{\sum z^{k}t_{k}}{\rm e}^{-\sum \frac1{kz^{k}}\partial_{t_{k}}}$; this is essentially the vertex operator for the free fermion vertex algebra, see, e.g.,~\cite{MR1651389}. The Baker function satisfies linear equations
\begin{gather}
\partial_{t_{k}}\Psi(z;t)=B_{k}(\partial_{x})\Psi(z;t),\label{eq:84}
\end{gather}
where $B_{k}(\partial_{x})$ is a degree $k$ differential operator in $\partial_{x}$. The compatibility of \eqref{eq:84} for $k=1$ and $k=3$ turns out to give precisely the Korteweg--de Vries equation~\eqref{eq:85}.

In this paper we are interested in integrable \emph{difference} (as opposed to differential) equations. Still, we follow very much the setup sketched above for the KdV hierarchy.

In the first part of this paper, we introduce a collection of $\tau$-functions as matrix elements of the action of loop group elements for ${\rm GL}_{2}$, depending on discrete variables\footnote{These variables can be thought of as coordinates on the lower triangular subgroup of the loop group of ${\rm GL}_{2}$.} $c_{i}$, which play a~similar role as the higher KdV times $t_{2k+1}$, $k>1$. These $\tau$-functions are of the form $\tau_{k}^{(\alpha)}(c_{i})$, where $k$, $\alpha$ are discrete variables. In fact, these $\tau$-functions turn out to be (see Theorem~\ref{Thm:tauHankel}) Hankel determinants, well known since the 19th century in the theory of orthogonal polynomials, see, e.g.,~\cite{MR2191786}.

We then define Baker functions. In this case, they are $2\times 2$ matrices depending on a spectral parameter $z$, on the discrete variables $k$, $\alpha$ (and on the~$c_{i}$):
\begin{gather*}
\Psi^{[k](\alpha)}(z)=\frac{(-1)^k}{\tau_{k}^{(\alpha)}}
\begin{bmatrix}
 z^{k}&0\\0&z^{-k}
\end{bmatrix}
\begin{bmatrix}
 S^{+}(z)&0\\0&S^{-}(z)
\end{bmatrix}
\begin{bmatrix}
 \tau_{k}^{(\alpha)}& \tau_{k-1}^{(\alpha)}/z\\ \tau_{k+1}^{(\alpha)}/z& \tau_{k}^{(\alpha)}
\end{bmatrix},
\end{gather*}
where $S^{\pm}(z)=(1-S/z)^{\pm1}$ are the shift fields, constructed from the elementary shift $S\colon \mathbb{C}[c_{k}]\to\mathbb{C}[c_{k}]$, defined as the multiplicative map such that $S(1)=0$, $S(c_{k})=c_{k+1}$. The shift fields $S^{\pm}(z)$ play a similar role here that the vertex operator $\Gamma(z)$ does in the theory of the KdV hierarchy.

Next, we introduce linear equations for the $2\times 2$ Baker functions:
\begin{gather*}
\Psi^{[k](\alpha+1)}=\Psi^{[k](\alpha)}V^{(\alpha)}_{k},\qquad \Psi^{[k-1](\alpha+1)}=\Psi^{[k](\alpha)}W^{(\alpha)}_{k}.
\end{gather*}
Here, the \emph{connection matrices} $V^{(\alpha)}_{k}$, $W^{(\alpha)}_{k}$ are $2\times 2$ matrices depending on the spectral parameter. These connection matrices can be explicitly expressed in terms of the $\tau$-functions.

We then show that compatibility of these equations leads to the discrete zero-curvature equations
\begin{gather*}
V_{k}^{(\alpha)}\big(W_{k+1}^{(\alpha)}\big)^{-1}=\big(W_{k+1}^{(\alpha-1)}\big)^{-1}V_{k+1}^{(\alpha-1)}.
\end{gather*}
Since we can give explicit expressions for the connection matrices $V_{k}^{(\alpha)}$, $W_{k}^{(\alpha)}$ in terms of the $\tau$-functions, we obtain the following basic system:
\begin{gather*}
\big(\tau_k^{(\alpha)}\big)^2=\tau_{k}^{(\alpha-1)}\tau_{k}^{(\alpha+1)}-\tau_{k+1}^{(\alpha-1)}\tau_{k-1}^{(\alpha+1)}, \qquad \alpha\in \mathbb{Z}, \qquad k=0,1,\dots.
\end{gather*}
After applying a change of variables, one can see that this is precisely the $A_{\infty/2}$ $Q$-system, see, e.g.,~\cite{MR2566162}. We refer to this system as the $2Q$-system, as it is obtained from the representation theory of the central extension of the loop group of~${\rm GL}_{2}$.

In the second part of the paper we generalize our derivation of the $2Q$-system by using the loop group of ${\rm GL}_{3}$, obtaining $\tau$-functions $\tau_{k,\ell}^{(\alpha,\beta)}(c_{i},d_{i},e_{i})$, where $k,\ell,\alpha,\beta\in\mathbb{Z}$ and the $c_{i}$, $d_{i}$, $e_{i}$ are coordinates on the lower triangular subgroup of the loop group of~${\rm GL}_{3}$. We can explicitly calculate these $\tau$-functions, see Theorem~\ref{thm:n=3taufunctions}, but their formula is much more complicated than the simple Hankel determinants in the $2\times 2$ case. Next we introduce Baker functions, now $3\times 3$ matrices depending on a spectral parameter, and the linear equations for the Baker functions. Again, we can explicitly calculate the connection matrices in terms of $\tau$-functions, see Lemma~\ref{lem:9}. Compatibility of the equations satisfied by the connection matrices in this case gives us a system of \emph{four} equations (Theorem~\ref{thm:3by3Qsystem}), which we will refer to as the $3Q$-system. For all $\alpha,\beta\in {\mathbb Z}$ and $k,\ell \geq 0$
\begin{gather}
 \tau_{k + 1, \ell}^{(\alpha,\beta)}\tau_{k - 1, \ell - 1}^{(\alpha + 1,\beta)}+ \tau_{k, \ell - 1}^{ (\alpha + 1,\beta)} \tau_{k, \ell}^{ (\alpha,\beta)}=\tau_{k, \ell - 1}^{( \alpha,\beta)}\tau_{k, \ell}^{( \alpha + 1,\beta)},\nonumber\\
 \tau_{k + 1, \ell +1}^{(\alpha,\beta)}\tau_{k, \ell}^{(\alpha,\beta + 1)} + \tau_{k +1, \ell}^{(\alpha,\beta + 1)}\tau_{k, \ell + 1}^{(\alpha,\beta)}= \tau_{k + 1, \ell + 1}^{(\alpha,\beta + 1)} \tau_{k,\ell}^{(\alpha,\beta)},\nonumber\\
\big(\tau_{k,\ell}^{(\alpha,\beta)}\big)^{2}= \tau_{k, \ell}^{(\alpha + 1,\beta)} \tau_{k, \ell}^{(\alpha - 1,\beta)}+ \tau_{k+1, \ell+1}^{ (\alpha-1,\beta)} \tau_{k - 1, \ell-1}^{(\alpha + 1,\beta)}-\tau_{k+1,\ell}^{(\alpha-1,\beta)}\tau_{k - 1, \ell}^{(\alpha + 1,\beta)} ,\nonumber\\
\big(\tau_{k, \ell}^{ (\alpha, \beta)}\big)^{2}=\tau_{k, \ell}^{ (\alpha, \beta + 1)} \tau_{k, \ell}^{ (\alpha, \beta -1)}-\tau_{k, \ell + 1}^{(\alpha, \beta - 1)}\tau_{k, \ell - 1}^{ (\alpha, \beta +1)}
-\tau_{k-1, \ell}^{(\alpha,\beta-1)}\tau_{k + 1, \ell}^{(\alpha, \beta + 1)}.\label{eq:48}
\end{gather}
The first two of these new equations are generalizations of $T$-system equations. More precisely, for fixed $\beta$, after a change of variables, the first equation is a $T$-system equation. Similarly, in the second equation, after applying a change of variables, we obtain a $T$-system equation for fixed~$\alpha$. It is known that $Q$- and $T$-systems are related to many areas in mathematics and physics. See for instance \cite{MR2577308,MR2566162} for relations to cluster algebras, and~\cite{MR2773889} for applications in integrable systems.

In particular it is known, see \cite{MR2254805}, that some particular solutions of $T$-systems are $q$-characters of Kirillov--Reshetikhin modules~\cite{MR906858, MR1255302}. It is therefore natural to ask if a similar representation theoretic meaning of particular solutions of the new $3Q$-system, \eqref{eq:48}, exists.

In another direction, the $\tau$-functions for the $Q$-system are, as we mentioned above, determinants of Hankel matrices, which appear in the theory of orthogonal polynomials~\cite{MR2191786}, and in the Toda lattice~\cite{MR1424652}. Again one can wonder what the meaning of the $3Q$-system is from the point of view of orthogonal polynomials and Toda lattices. We will see in Section~\ref{sec:wideh-tau-funct} that the $\tau$-functions of the $3Q$-system depend on the choice of a~lower triangular matrix
\begin{gather*}
 \begin{bmatrix}
 1&0&0\\ C(z)&1&0\\ D(z)&E(z)&1
 \end{bmatrix}.
\end{gather*}
Here $C(z)$, $D(z)$, $E(z)$ are series in $z$. When we look at the special case where $E(z)=0,$ we obtain $\tau$-functions that are determinants of block Hankel matrices, related to bi-orthogonal polynomials, and $4$-band
Toda lattices (see, e.g.,~\cite{MR2754822}). See~\cite{MR3709701} for a~preliminary report on this.

We hope that it is clear from this paper that the theory of $Q$-systems and $T$-systems, with their many applications, is just the tip of an iceberg. For any $n>2$, there are $nQ$-systems and $nT$-systems, which are generalizations of the $2Q$- and $2T$-systems. In this paper, we discuss the construction of the $nQ$-systems for $n=2,3$. See~\cite{doi:10.1142/S0129167X18500908} for more general hierarchies.

\section[The $2\times 2$ case]{The $\boldsymbol{2\times 2}$ case}\label{sec:2x2-case}

\subsection[$2\times 2$ $\tau$-functions and $Q$-system]{$\boldsymbol{2\times 2}$ $\boldsymbol{\tau}$-functions and $\boldsymbol{Q}$-system}
\label{sec:wideh-tau-funct-1}

We have an action of the central extension, $\widehat{\rm GL}_2$, of the loop group $\widetilde{\rm GL}_2 ={\rm GL}_{2} (\mathbb{C}((z)) )$ on two-component fermionic Fock space, $F^{(2)}$, the semi-infinite wedge spaced based in $\mathbb{C}^2\otimes \mathbb{C}\big[z,z^{-1}\big]$, see, e.g., \cite{tKvdL:BosFer} for $n$-component fermions, and~\cite{KaPe:LectinfwedgMKP} for the construction of central extensions of Lie algebras and corresponding groups. Some of this material is reviewed in Appendix~\ref{sec:ferm-semi-infin}. Let $\pi\colon \widehat{\rm GL}_2 \to \widetilde{\rm GL}_2 $ be the projection onto the non-centrally extended loop group. {We will consider the action of a~group element, $g_{C}\in \widehat{\rm GL}_2$, where
\begin{gather}\label{eq:7}
 \pi(g_{C})= \begin{bmatrix}
 1 & 0\\
 C(z) & 1
 \end{bmatrix},
\end{gather} on the vacuum vector $v_{0}$} of $F^{(2)}$:
\begin{gather*} v_{0}=
\begin{bmatrix}
 1\\0\end{bmatrix}\wedge
\begin{bmatrix}
 0\\1\end{bmatrix}
\wedge
\begin{bmatrix}
 z\\0\end{bmatrix}\wedge
\begin{bmatrix}
 0\\z
\end{bmatrix}\wedge
\begin{bmatrix}
 z^{2}\\0
\end{bmatrix}\wedge
\begin{bmatrix}
 0\\z^{2}\\
\end{bmatrix}
\wedge\cdots,\end{gather*}
see~\eqref{eq:4}.
Here
\begin{gather*}
C(z)= \sum_{i\in \mathbb{Z}} c_{i}z^{-i-1},
\end{gather*}
where $c_{i}\in\mathbb{C}$. In order for $\pi(g_{C})$ to belong to ${\rm GL}_{2}(\mathbb{C}((z)))$ we would need to impose the condition that $c_{i}=0$ for $i\ll 0$. However, sometimes it is useful to think of the $c_{i}$s as formal variables. In that case, $\pi(g_{C})$ belongs to the invertible elements in $\mathfrak{gl}_{2}(\mathbb{C}((z)))[[c_{i}]]$, and will be a $2\times 2$ matrix with coefficients given by series that are infinite in both directions in~$z$.

In order to define our $\tau$-functions, we need to define fermionic translation operators which we denote $Q_i$, $i=0,1$. In~\eqref{eq:3}, we will carefully define these operators and their action on~$F^{(2)}$, in terms of wedging and contracting operators,
\begin{gather*} e\big(e_{a}^{k}\big)\alpha=e_{a}^{k}\wedge\alpha,\qquad
 i\big(e_{a}^{k}\big)\alpha=\beta, \qquad \text{if} \quad \alpha=e_{a}^{k}\wedge\beta,\end{gather*}
for $\alpha,\beta\in F^{(2)}$. Here, $e_a^{k}=e_{a}z^{k}$, where $e_{a}$ denotes the standard basis vectors of $\mathbb{C}^{2}$, $e_{0}=
\begin{bmatrix}
 1\\0
\end{bmatrix}$, $e_1=
\begin{bmatrix}
 0\\1
\end{bmatrix}$. The projections of the $Q_i$s onto the loop group $\widetilde{\rm GL}_2 $ are given by
\begin{gather*} \pi(Q_{0})=
\begin{bmatrix}
 z\inv&0\\0&-1
\end{bmatrix}, \qquad
\pi(Q_{1})=
\begin{bmatrix}
-1&0\\0& z\inv
\end{bmatrix}.
\end{gather*}
We also need the translation element
\begin{gather*}
T=Q_{1}Q_{0}^{-1}\overset{\pi}\mapsto
\begin{bmatrix}
 -z&0\\0&-z\inv
\end{bmatrix}.
\end{gather*}
We define shifts on the series $C(z)$ by
\begin{gather}
C^{(\alpha)}(z)=(-1)^\alpha z^{\alpha}C(z)=\sum_{i\in\mathbb{Z}}(-1)^\alpha c_{i+\alpha}z^{-i-1}.\label{eq:31}
\end{gather}
We similarly define the shifted group element $g_C^{(\alpha)}=Q_0^{\alpha}g_{C}Q_0^{-\alpha}$ so that
\begin{gather}
\pi\big(g_{C}^{(\alpha)}\big)=\begin{bmatrix}
 1 & 0\\
 C^{(\alpha)}(z) & 1
 \end{bmatrix}=
\begin{bmatrix}
 1 & 0\\
 (-1)^\alpha z^{\alpha}C(z) & 1
\end{bmatrix}.\label{eq:28}
\end{gather}
We then have
\begin{gather}
 Q_{0}\inv g_{C}^{(\alpha+1)}=g_{C}^{(\alpha)}Q_{0}\inv,\label{eq:44}
\end{gather}
and the same relation with $\pi$ applied,
\begin{gather}
\pi(Q_{0})\inv \pi\big(g_{C}^{(\alpha+1)}\big)=\pi\big(g_{C}^{(\alpha)}\big)\pi(Q_{0})\inv.\label{eq:45}
\end{gather}

The fundamental objects in the theory of the Toda lattice (see, e.g., \cite{MR1424652}) and $Q$-systems are the $\tau$-functions defined by
\begin{gather}
\tau_{k}^{(\alpha)}=\big\langle T^{k}v_{0},g_{C}^{(\alpha)}v_{0}\big\rangle.\label{eq:32}
\end{gather}
Here $\langle\,,\rangle$ is the bilinear form on semi-infinite wedges. (For more details, see Appendix~\ref{sec:semi-infinite-wedge}, where we will define a basis for $F^{(2)}$. $\langle\,,\,\rangle$ is the bilinear product with respect to which these basis vectors are orthonormal.)

The $\tau$-functions in the $2\times 2$ case are determinants of Hankel matrices.

\begin{thm}\label{Thm:tauHankel}For $k<0$ and all $\alpha\in\mathbb{Z}$ we have $\tau_{k}^{(\alpha)}=0$, and $\tau_{0}^{(\alpha)}=1$. For $k>0$ and all $\alpha\in\mathbb{Z}$
\begin{gather*}
 \tau_k^{(\alpha)} =\frac1{k!}\Res_{\bf{w}}\left(\prod_{i=1}^{k}C^{(\alpha)}(w_{i})\prod_{1\le i<j\le k}(w_i-w_j)^2\right) \\
\hphantom{\tau_k^{(\alpha)}}{} =(-1)^{k\alpha}\det
 \begin{bmatrix}
 c_{\alpha} & c_{\alpha+1} & \cdots & c_{\alpha+k-1}\\
 c_{\alpha+1} & c_{\alpha+2} & \cdots & c_{\alpha+k}\\
 \vdots & \vdots & \vdots & \vdots\\
 c_{\alpha+k-1} & c_{\alpha+k} & \cdots & c_{\alpha+2k-2}
 \end{bmatrix}.
\end{gather*}
Here $\Res_{\mathbf{w}}=\Res_{w_{1}}\Res_{w_{2}}\cdots\Res_{w_{k}}$, and the residue $\Res_{z}\left(a(z)\right)$ is the coefficient $a_{-1}$ of $z\inv$ in a series $a(z)=\sum\limits_{i\in\mathbb{Z}}a_{i}z^{i}$.
\end{thm}

For the proof of this, see Appendix~\ref{sec:expr-tau-funct-1}.

The simple form of the $\tau$-functions allows us to apply the Desnanot--Jacobi identity (cf.~\cite{MR1718370}) to obtain~\cite{MR2566162} the following difference equations, referred to as the $2Q$-system, satisfied by our $\tau$-functions: For all $k\ge0$ and for all~$\alpha\in\mathbb{Z}$,
\begin{gather*}
\tau_{k}^{(\alpha)}\tau_{k-2}^{(\alpha+2)}=\tau_{k-1}^{(\alpha+2)}\tau_{k-1}^{(\alpha)}-\big(\tau_{k-1}^{(\alpha+1)}\big)^2.\label{eq:50}
\end{gather*}
The disadvantage of obtaining the difference equations in this way is that it is not at all apparent how to generalize this to the $3\times 3$ situation, in which the formulas for the $\tau$-functions are much more complicated. We thus present another way of obtaining our $2\times 2$ difference equations.

\subsection{Birkhoff factorization} \label{sec:birkh-fact}

Define an element of the central extension of the loop group of ${\rm GL}_{2}$:
\begin{gather*}
g^{[k](\alpha)}=T^{-k}g^{(\alpha)}_{C},
\end{gather*}
and assume that it has a Birkhoff factorization \cite{MR699439, PrSe:LpGrps} (see also Appendix~\ref{sec:birkh-fact-n}):
\begin{gather*}
g^{[k](\alpha)}=g^{[k](\alpha)}_{-}g^{[k](\alpha)}_{0+},
\end{gather*}
where $\pi\big(g^{[k](\alpha)}_{-}\big)=1+{O}\big(z\inv\big)$ and $\pi\big(g^{[k](\alpha)}_{0+}\big)=A_{k}^{(\alpha)}+{O}(z)$, for $A_{k}^{(\alpha)}$ an invertible $z$ independent matrix. This assumption is justified precisely when the matrix element $\tau^{(\alpha)}_{k}(g)=\big\langle v_{0},g^{[k](\alpha)}v_{0}\big\rangle$ is not zero, see for instance,~\cite{MR87b:58039}. In their paper, Segal--Wilson treat essentially the case of $n=1$ of the theory of $n$-component fermionic Fock space used in our current paper, although they emphasize the connection to the geometry of infinite Grassmannians, whereas we put the theory of fermion operators in the forefront. Segal--Wilson explain that the vanishing of the $\tau$-function detects that the corresponding element $W=gH_{+}$ of the infinite Grassmannian is not in the big cell. Being in the big cell for $W=gH_{+}$ is equivalent to~$g$ having a Birkhoff factorization. We leave it to the reader to check that this picture still holds for arbitrary~$n$.

Now we want to display the negative component of $\pi\big(g^{[k](\alpha)}\big)$. To calculate this we make some extra structure explicit.

Let $\mathcal{N}$ be the subgroup of elements of $\widetilde{\rm GL}_2 $ of the form~\eqref{eq:7}. We can think of the coeffi\-cients~$c_{k}$ as coordinates on~$\mathcal{N}$, so{\samepage
\begin{gather*}
B=\mathbb{C}[c_{k}]_{k\in\mathbb{Z}}
\end{gather*}
is the coordinate ring of $\mathcal{N}$.}

First define shifts acting on $B$. These are multiplicative maps given on generators by
\begin{gather*}
 S^{\alpha}\colon \ B\to B,\qquad S^{\alpha}(1)=0, \qquad S^{\alpha}(c_{k})=c_{k+\alpha}, \qquad \alpha\in\mathbb{Z}.
\end{gather*}
We will often write $S^{\pm}$ for $S^{\pm1}$.

We also define \emph{shift fields}. These are multiplicative maps
\begin{gather}
S^{\pm}(z)\colon \ B\to B\big[\big[z\inv\big]\big],\label{eq:10}
\end{gather}
given by
\begin{gather*}
S^{\pm}(z)=\left(1-\frac{S^{+}}{z}\right)^{\pm1} ,
\end{gather*}
\begin{thm}\label{Thm:birkh-fact-tau} For $k\ge0$ and all $\alpha\in\mathbb{Z}$
 \begin{gather*}
\pi(g^{[k](\alpha)}_{-})=
\begin{bmatrix}
 S^{+}(z)\tau^{(\alpha)}_{k}& S^{+}(z)\tau^{(\alpha)}_{k-1}/z\\
S^{-}(z)\tau^{(\alpha)}_{k+1}/z&S^{-}(z)\tau^{(\alpha)}_{k}
\end{bmatrix}/\tau^{(\alpha)}_{k}.
\end{gather*}
\end{thm}

We sketch the proof in Appendix \ref{sec:2times2-case-proof}.

Now that we have expressed the negative component of the Birkhoff factorization in terms of matrix elements of the centrally extended loop group, we no longer need the central extension and we will simplify notation: for the remainder of Section~\ref{sec:2x2-case}, we will write $g^{[k](\alpha)}_{-}$ for $\pi\big(g^{[k](\alpha)}_{-}\big)$, $T$ for $\pi(T)$ and~$Q_{a}$ for $\pi(Q_{a})$, $a=0,1$. In particular, in the rest of this section we write
\begin{gather*} T=
\begin{bmatrix}
 -z&0\\0&-z\inv
\end{bmatrix}=Q_{1}Q_{0}\inv,\qquad Q_{0}=
\begin{bmatrix}
 z\inv&0\\0&-1
\end{bmatrix},\qquad Q_{1}=
\begin{bmatrix}
 -1&0\\0&z\inv
\end{bmatrix}.
\end{gather*}

\begin{Remark} Theorem~\ref{Thm:birkh-fact-tau} implies that for $k\ge 0$ we have a Birkhoff factorization for $g^{[k](\alpha)}$, as long as $\tau_{k}^{(\alpha)}$ is not zero. Here we are dropping the $\pi$ as discussed above. It is easy to see that for $k<0$ such a~factorization is not possible. Indeed, assume for simplicity that $\alpha=0$, and consider
\begin{gather*}
T^{k}g_{C}=(-1)^{k}
\begin{bmatrix}
  z^{k}&0\\
z^{-k}C(z)&z^{-k}
\end{bmatrix}=
\begin{bmatrix}
z^{k}&0\\
\big(z^{-k}C(z)\big)_{-}&z^{-k}
\end{bmatrix}
\left(
(-1)^{k}
  \begin{bmatrix}
    1&0\\
    \big(z^{-k}C(z)\big)_{0+}z^{k}&1
  \end{bmatrix}
\right).
\end{gather*}
Here the subscripts $-$, $0+$ on a two sided infinite series in $z$ denote the terms containing negative, respectively non-negative powers of~$z$.

The existence a Birkhoff factorization for $T^{k}g_{C}$ for $k<0$ reduces then to the existence of a~Birkhoff factorization of the left factor
\begin{gather*}
  \Gamma=    \begin{bmatrix}
z^{k}&0\\
(z^{-k}C(z))_{-}&z^{-k}
\end{bmatrix}=
\begin{bmatrix}
z^{k}&0\\
\gamma_{0}z^{-1}+\gamma_{1}z^{-2}+\cdots&z^{-k}
  \end{bmatrix},
\end{gather*}
since the right hand side factor of $T^{k}g_{C}$ already belongs to the non-negative loop group.

If we could write this factor  as $\Gamma_-\Gamma_{0+}$, then we would  have $\Gamma (\Gamma_{0+} )\inv=  \begin{bmatrix}    1&0\\0&1\end{bmatrix}+\mathcal{O}\big(z\inv\big)$. In particular, looking at the   second column of this matrix equality, we see that this would mean  (since the entries of $\Gamma_{0+}$ and its inverse would contain only non-negative powers of $z$) that there are power series $f(z),g(z)\in\mathbb{C}[[z]]$ so that
  \begin{gather*}
f(z)
\begin{bmatrix}
  z^{k}\\ \gamma_{0}z^{-1}+\gamma_{1}z^{-2}+\cdots
\end{bmatrix}+g(z) \begin{bmatrix}   0\\ z^{-k}\end{bmatrix}=\begin{bmatrix}
  0\\ 1 \end{bmatrix}+\mathcal{O}\big(z\inv\big)
.
\end{gather*}
It is clear that for $k<0$ such series $f(z)$, $g(z)$ do not exist (we would need $f(z)$ to be zero, but there is no  $g(z)$ in $\mathbb{C}[[z]]$ such that $g(z)z^{-k}=1 +\mathcal{O}\big(z\inv\big)$).

The argument for $\alpha\ne 0$ is similar.
\end{Remark}

\subsection{Matrix Baker functions and connection matrices}\label{sec:matr-baker-funct}

Next, we define the Baker functions. These are elements of the loop group defined by
\begin{gather}
\Psi^{[k](\alpha)}=T^{k}Q_0^{-\alpha}g_{-}^{[k](\alpha)}.\label{eq:42}
\end{gather}
Since the Baker functions are all invertible, they are related by \emph{connection matrices} belonging to the loop group. Define $\Gamma_{[k](\alpha)}^{[\ell](\beta)}\in \widetilde{\rm GL}_2 $ by
\begin{gather*}
\Psi^{[\ell](\beta)}=\Psi^{[k](\alpha)}\Gamma_{[k](\alpha)}^{[\ell](\beta)}.
\end{gather*}
We are interested in connection matrices that implement nearest neighbor steps on the lattice of Baker functions. So define \emph{elementary connection matrices}
\begin{gather*}
 U_{k}^{(\alpha)}=\Gamma_{[k](\alpha)}^{[k+1](\alpha)},\qquad V_{k}^{(\alpha)}=\Gamma_{[k](\alpha)}^{[k](\alpha+1)},\qquad W_{k}^{(\alpha)}=\Gamma_{[k](\alpha)}^{[k-1](\alpha+1)},
\end{gather*}
so that we have
\begin{gather}
 \Psi^{[k+1](\alpha)}=\Psi^{[k](\alpha)} U_{k}^{(\alpha)},\qquad
 \Psi^{[k](\alpha+1)}=\Psi^{[k](\alpha)} V_{k}^{(\alpha)},\qquad
 \Psi^{[k-1](\alpha+1)}=\Psi^{[k](\alpha)}W_{k}^{(\alpha)}.
\label{eq:43}
\end{gather}
Pictorially:
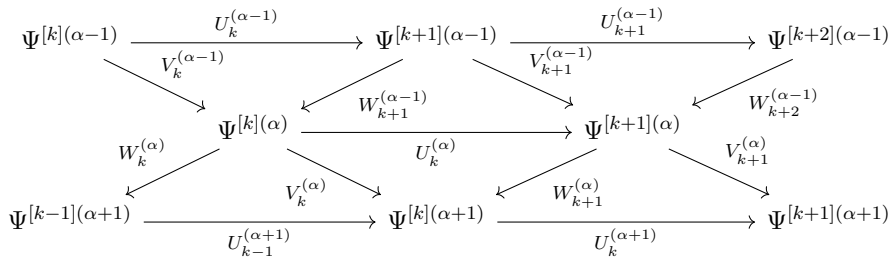
\begin{equation*}
\begin{tikzpicture}[baseline= (a).base]
\node[scale=.9] (a) at (0,0){
 \begin{tikzcd}
\Psi^{[k](\alpha-1)}
\arrow{rr}{U_{k}^{(\alpha-1)}}
\ar{rd}{V_{k}^{(\alpha-1)}}
&{}&
\Psi^{[k+1](\alpha-1)}
\ar{rr}{U_{k+1}^{(\alpha-1)}}
\ar{rd}{V_{k+1}^{(\alpha-1)}}
\ar{ld}{W_{k+1}^{(\alpha-1)}}&{}&\Psi^{[k+2](\alpha-1)}
\ar{dl}{W_{k+2}^{(\alpha-1)}}\\
{}& \Psi^{[k](\alpha)}
 \ar{rr}[swap]{U_{k}^{(\alpha)}}
 \ar{dr}[swap]{V_{k}^{(\alpha)}}
\ar{dl}[swap]{W_{k}^{(\alpha)}}&{} &
\Psi^{[k+1](\alpha)}
\ar{ld}{W_{k+1}^{(\alpha)}}
\ar{rd}{V_{k+1}^{(\alpha)}}&{}&
 \\
\Psi^{[k-1](\alpha+1)}
\arrow{rr}[swap]{U_{k-1}^{(\alpha+1)}}
&{}&
\Psi^{[k](\alpha+1)}
\ar{rr}[swap]{U_{k}^{(\alpha+1)}}&{}&\Psi^{[k+1](\alpha+1)}
\end{tikzcd}
};
\end{tikzpicture}
\end{equation*}
Walking around the triangles in this diagram, we see that we get
factorizations of all elementary connection matrices. In particular,
\begin{gather}
U_{k}^{(\alpha)}=V_{k}^{(\alpha)}\big(W_{k+1}^{(\alpha)}\big)^{-1}=\big(W_{k+1}^{(\alpha-1)}\big)^{-1}V_{k+1}^{(\alpha-1)}.\label{eq:47}
\end{gather}
Such factorizations are well known in the theory of integrable systems, see for instance Adler~\cite{Ad:BackGelfDick}, Sklyanin~\cite{MR1791892}. They go back to the work of Darboux in the 19th century, see for example~\cite{MR1146435}.

We study the elementary connection matrices more explicitly.
\begin{lem}\label{lem:5}
 \begin{gather*}
 U_{k}^{(\alpha)}=\big(g^{[k](\alpha)}_{-}\big)\inv T g_{-}^{[k+1](\alpha)}= g_{0+}^{[k](\alpha)}\big(g_{0+}^{[k+1](\alpha)}\big)\inv,\\ 
 V_{k}^{(\alpha)}=\big(g_{-}^{[k](\alpha)}\big)\inv Q_{0}\inv g_{-}^{[k](\alpha+1)}= g_{0+}^{[k](\alpha)}Q_{0}\inv \big(g_{0+}^{[k](\alpha+1)}\big)\inv,\\ 
 W_{k}^{(\alpha)}=\big(g_{-}^{[k](\alpha)}\big)\inv Q_{1}\inv g_{-}^{[k-1](\alpha+1)}= g_{0+}^{[k](\alpha)}Q_{0}\inv \big(g_{0+}^{[k-1](\alpha+1)}\big)\inv.
 \end{gather*}
\end{lem}
\begin{proof}The first expression for the elementary connection matrices (in terms of negative components $g^{[k](\alpha)}_{-}$) follows from the definition~\eqref{eq:43} of the connection matrices and the definition~\eqref{eq:42} of the Baker functions. It also uses $Q_{0}T=Q_{1}$.

To derive the second expression for the elementary connection matrices in terms of positive components, $g^{[k](\alpha)}_{0+}$ we use (see also~\eqref{eq:44}, or rather~\eqref{eq:45})
\begin{gather*}
 Tg^{[k+1](\alpha)}_{-}g^{[k+1](\alpha)}_{0+}=g^{[k](\alpha)}_{-}g^{[k](\alpha)}_{0+},\\
 Q_{0}\inv g^{[k](\alpha+1)}_{-}g^{[k](\alpha+1)}_{0+}=g^{[k](\alpha)}_{-}g^{[k](\alpha)}_{0+}Q_{0}\inv,\\
 Q_{1}\inv g^{[k-1](\alpha+1)}_{-}g^{[k-1](\alpha+1)}_{0+}=g^{[k](\alpha)}_{-}g^{[k](\alpha)}_{0+}Q_{0}\inv.
\end{gather*}
Rearranging factors then proves the second form for the connection matrices.
\end{proof}

\begin{Remark}\label{remark1}Note that the second equality in the above lemma tells us that the elementary connection matrices contain only~$z^{k}$ for $k\ge0$. This allows us to easily calculate these connection matrices in terms of $\tau$-functions, as we can ignore any (often complicated) terms that would contribute only negative powers of~$z$.
\end{Remark}

First note that Theorem \ref{Thm:birkh-fact-tau} allows us to expand $g_-^{[k](\alpha)}$ and its inverse up to order $z\inv$ as
\begin{gather}
 g_{-}^{[k](\alpha)}=1_{2\times2}+\frac1{z}
 \begin{bmatrix}
 S^{+}[1]\tau_{k}^{(\alpha)}/\tau_{k}^{(\alpha)}&\frac1{h_{k-1}^{(\alpha)}}\vspace{1mm}\\
 h_{k}^{(\alpha)}& S^{-}[1]\tau^{(\alpha)}_{k}/\tau^{(\alpha)}_{k}
 \end{bmatrix}+ O\big(z^{-2}\big),\nonumber\\
 \big(g_{-}^{[k](\alpha)}\big)\inv =1_{2\times 2}+\frac1{z} \begin{bmatrix}
 -S^{+}[1]\tau_{k}^{(\alpha)}/\tau_{k}^{(\alpha)}&-\frac{1}{h_{k-1}^{(\alpha)}}\vspace{1mm}\\
 -h_{k}^{(\alpha)}&
 -S^{-}[1]\tau^{(\alpha)}_{k}/\tau^{(\alpha)}_{k}
 \end{bmatrix}+ O\big(z^{-2}\big).\label{eq:46}
\end{gather}
Here
\begin{gather*}
h_{k}^{(\alpha)}=\frac{\tau_{k+1}^{(\alpha)}}{\tau_{k}^{(\alpha)}},
\end{gather*}
and we expand the shift fields in partial shifts
\begin{gather*}
S^{\pm}(z)f=\sum_{n=0}^{\infty}S^{\pm}[n]f z^{-n}.
\end{gather*}
\begin{lem}\label{lem:4}
\begin{gather*}
V_{k}^{(\alpha)}=
\begin{bmatrix}
 z+v^{(\alpha)}_{k-1}&\dfrac1{h_{k-1}^{(\alpha+1)}}\vspace{1mm}\\
-h_{k}^{(\alpha)}&-1
\end{bmatrix}, \qquad
W^{(\alpha)}_{k}= \begin{bmatrix}
 -1&-\dfrac1{h^{(\alpha)}_{k-1}}\vspace{1mm}\\
h_{k-1}^{(\alpha+1)}& z+w^{(\alpha)}_{k-1}
\end{bmatrix}, \qquad k=0,1,\dots.
\end{gather*}
Here
\begin{gather*}
v_{k}^{(\alpha)}=\frac{h_{k+1}^{(\alpha)}}{h_{k}^{(\alpha+1)}},\qquad w_{k}^{(\alpha)}=\frac{h_{k}^{(\alpha+1)}}{h_{k}^{(\alpha)}}, \qquad v_{-1}^{(\alpha)}=\frac1{h_{-1}^{(\alpha)}}=0.
\end{gather*}
Note that $\det\big(V^{(\alpha)}_{k}\big)=-z=\det\big(W^{(\alpha)}_k\big)$.
\end{lem}

\begin{proof} As an example, we calculate $V_{k}^{(\alpha)}=\big(g^{[k](\alpha)}_{-}\big)\inv Q_{0}\inv g_-^{[k](\alpha+1)}$, using~\eqref{eq:46}:
\begin{gather*}
 V_{k}^{(\alpha)}=
 \begin{bmatrix}
 1+\dfrac{x}{z}+O\big(z^{-2}\big)&O\big(z\inv\big)\vspace{1mm}\\
 -\dfrac{h_{k}^{(\alpha)}}{z}+O\big(z^{-2}\big)&1+{O}\big(z\inv\big)
 \end{bmatrix}
 \begin{bmatrix}
 z&0\\0&-1
 \end{bmatrix}
 \begin{bmatrix}
 1+\dfrac{y}{z}+O\big(z^{-2}\big)&\dfrac1{zh_{k-1}^{(\alpha+1)}}+O\big(z^{-2}\big)\vspace{1mm}\\
 O\big(z\inv\big)&1+O\big(z\inv\big)
 \end{bmatrix}\\
\hphantom{V_{k}^{(\alpha)}}{} =\begin{bmatrix}
z+x+y&\dfrac1{h_{k-1}^{(\alpha+1)}}\vspace{1mm}\\
-h_{k}^{(\alpha)}&-1
 \end{bmatrix},
 \end{gather*}
dropping all terms containing $z\inv$ or lower, see Remark~\ref{remark1} for why this is justified. Here~$x$,~$y$ are some expressions in the $\tau$-functions which we will determine by noting that $\det\big(V_{k}^{(\alpha)}\big)=-z$. We see that $x+y=\frac{h_{k}^{(\alpha)}}{h_{k-1}^{(\alpha+1)}}=v_{k-1}^{(\alpha)}$. This proves the lemma for $V_{k}^{(\alpha)}$, $k=1,2,\dots$. The proof for~$W_{k}^{(\alpha)}$ and~$V_{0}^{(\alpha)}$ is similar.
\end{proof}

We now return to the two factorizations~\eqref{eq:47} of the connection matrix. Using the expressions for $V_{k}^{(\alpha)}$, $W_{k}^{(\alpha)}$ from Lemma~\ref{lem:4}, we find two expressions for $U_k^{(\alpha)}$:
\begin{gather*}
 U_k^{(\alpha)}=V_k^{(\alpha)}\big(W_{k+1}^{(\alpha)}\big)^{-1} =
 \begin{bmatrix}-z-\dfrac{h_k^{(\alpha)}}{h_{k-1}^{(\alpha+1)}}
 -\dfrac{h_{k}^{(\alpha+1)}}{h_{k}^{(\alpha)}}&
 -\dfrac{1}{h_{k}^{(\alpha)}}\vspace{1mm}\\
h_{k}^{(\alpha)}
 & 0
\end{bmatrix}\nonumber\\
\hphantom{U_k^{(\alpha)}}{}
 =\big(W_{k+1}^{(\alpha-1)}\big)^{-1}V_{k+1}^{(\alpha-1)} =
\begin{bmatrix}-z-\dfrac{h_{k+1}^{(\alpha-1)}}{h_{k}^{(\alpha)}}-\dfrac{h_k^{(\alpha)}}{h_{k}^{(\alpha-1)}} & -\dfrac{1}{h_{k}^{(\alpha)}}\vspace{1mm}\\h_{k}^{(\alpha)} & 0\end{bmatrix},
\end{gather*}
giving equations for the $h_{k}^{(\alpha)}$ variables
\begin{gather}
\frac{h^{(\alpha)}_{k}}{h_{k-1}^{(\alpha+1)}}+\frac{h_{k}^{(\alpha+1)}}{h^{(\alpha)}_{k}}=\frac{h_{k+1}^{(\alpha-1)}}{h_{k}^{(\alpha)}}+\frac{h_{k}^{(\alpha)}}{h_{k}^{(\alpha-1)}}.\label{eq:63}
\end{gather}

\begin{thm}\label{Thm:Qsystem}The equations \eqref{eq:63} are equivalent to the $2Q$-system
\begin{gather*}
\big(\tau_k^{(\alpha)}\big)^2=\tau_{k}^{(\alpha-1)}\tau_{k}^{(\alpha+1)}-\tau_{k+1}^{(\alpha-1)}\tau_{k-1}^{(\alpha+1)},\qquad k=0,1,\dots.
\end{gather*}
\end{thm}
\begin{proof}Write \eqref{eq:63} out in terms of $\tau$-functions
\begin{gather*}
\frac{\tau_{k+1}^{(\alpha)}\tau_{k-1}^{(\alpha+1)}}{\tau_{k}^{(\alpha)}\tau_{k}^{(\alpha+1)}}+
\frac{\tau_{k}^{(\alpha)}\tau_{k+1}^{(\alpha+1)}}{\tau_{k+1}^{(\alpha)}\tau_{k}^{(\alpha+1)}}=
\frac{\tau_{k+2}^{(\alpha-1)}\tau_{k}^{(\alpha)}}{\tau_{k+1}^{(\alpha-1)}\tau_{k+1}^{(\alpha)}}+
\frac{\tau_{k+1}^{(\alpha)}\tau_{k}^{(\alpha-1)}}{\tau_{k}^{(\alpha)}\tau_{k+1}^{(\alpha-1)}}.
\end{gather*}
Bringing all terms under the same denominator and then rearranging terms, we see that this is equivalent to
\begin{gather}\label{eq:1}
\big(\tau_{k}^{(\alpha)}\big)^2\big(\tau_{k+2}^{(\alpha-1)}\tau_{k}^{(\alpha+1)}-\tau_{k+1}^{(\alpha-1)}\tau_{k+1}^{(\alpha+1)}\big)=
\big(\tau_{k+1}^{(\alpha)}\big)^2 \big(\tau_{k+1}^{(\alpha-1)}\tau_{k-1}^{(\alpha+1)}-\tau_{k}^{(\alpha-1)}\tau_{k}^{(\alpha+1)}\big).
\end{gather}
Notice that if
\begin{gather*}
\big(\tau_k^{(\alpha)}\big)^2=\tau_{k}^{(\alpha-1)}\tau_{k}^{(\alpha+1)}-\tau_{k+1}^{(\alpha-1)}\tau_{k-1}^{(\alpha+1)},
\end{gather*}
then \eqref{eq:1} implies
\begin{gather*}
\big(\tau_{k+1}^{(\alpha)}\big)^2=\tau_{k+1}^{(\alpha-1)}\tau_{k+1}^{(\alpha+1)}-\tau_{k+2}^{(\alpha-1)}\tau_k^{(\alpha+1)}.
\end{gather*}
We thus need only prove that the equality holds for $k=0$. But this is just
\begin{gather*}
\big(\tau_{0}^{(\alpha)}\big)^2=\tau_{0}^{(\alpha-1)}\tau_{0}^{(\alpha+1)}-\tau_{1}^{(\alpha-1)}\tau_{-1}^{(\alpha+1)},
\end{gather*}
which is true since $\tau_{-1}^{(\alpha)}=0$ and $\tau_{0}^{(\alpha)}=1$ for all $\alpha$. So the theorem follows.
\end{proof}

So we have rederived {the $2Q$-system}, see the equations \eqref{eq:50}, using the Birkhoff factorization.

\section[$3\times 3$ case]{$\boldsymbol{3\times 3}$ case}\label{sec:wideh-tau-funct}

\subsection[$\tau$-functions]{$\boldsymbol{\tau}$-functions}\label{sec:tau-functions}

We now discuss the generalization to the $3\times 3$ case, proceeding very similarly to the $2\times2$ case.

We have an action of the central extension $\widehat{\rm GL}_3$ of the loop group $\widetilde{\rm GL}_3 ={\rm GL}_{3}\big(\mathbb{C}\big(\big(z\inv\big)\big)\big)$ on three-component fermionic Fock space $F^{(3)}$. See, e.g.,~\cite{tKvdL:BosFer} for $n$-component fermions, and~\cite{KaPe:LectinfwedgMKP} for the construction of central extensions of Lie algebras and corresponding groups. Some of this material is reviewed in Appendix~\ref{sec:ferm-semi-infin}. Let $\pi\colon \widehat{\rm GL}_3 \to \widetilde{\rm GL}_3 $ be the projection onto the non-centrally extended loop group and consider the action of the group element, $g_{C,D,E}\in\widehat{\rm GL}_3$, where
\begin{gather*}
 \pi(g_{C,D,E})= \begin{bmatrix}
 1 & 0&0\\
 C(z) & 1&0\\
 D(z)&E(z)&1
 \end{bmatrix},
\end{gather*} on the vacuum vector of $F^{(3)}$. Here
\begin{gather*}
X(z)= \sum_{i\in \mathbb{Z}} x_{i}z^{-i-1},\qquad X=C,D,E, \qquad x=c,d,e,
\end{gather*}
where the $x_{i}$ are formal variables and the vacuum vector, $v_0$ is, analogous to the $2\times 2$ case,
\begin{gather*} v_{0}=
\begin{bmatrix}
1\\0\\0\end{bmatrix}\wedge
\begin{bmatrix}
0\\1\\0\end{bmatrix}\wedge
\begin{bmatrix}
0\\0\\1\end{bmatrix}
\wedge
\begin{bmatrix}
z\\0\\0\end{bmatrix}\wedge
\begin{bmatrix}
0\\z\\0\end{bmatrix}\wedge
\begin{bmatrix}
0\\0\\z\end{bmatrix}
\wedge
\begin{bmatrix}
z^2\\0\\0\end{bmatrix}\wedge
\begin{bmatrix}
0\\z^2\\0\end{bmatrix}\wedge
\begin{bmatrix}
0\\0\\z^2\end{bmatrix}
\wedge\cdots,
\end{gather*}
see~\eqref{eq:4}.

As in the $2\times 2$ case, we have fermionic translation operators $Q_i$, $0\le i\le 2$. The action of these $Q_i$s on $F^{(3)}$ is defined carefully in the appendix (see \eqref{eq:3}). Their projections onto the loop group $\widetilde{\rm GL}_3 $ are given by the following (commuting) matrices
\begin{gather*}
\pi(Q_{0})=
\begin{bmatrix}
 z\inv&0&0\\0&-1&0\\0&0&-1
\end{bmatrix}, \qquad\!
\pi(Q_{1})=
\begin{bmatrix}
-1&0&0\\0& z\inv&0\\0&0&-1
\end{bmatrix},\qquad\! \pi(Q_{2})=
\begin{bmatrix}
 -1&0&0\\0&-1&0\\0&0&z\inv
\end{bmatrix}.
\end{gather*}
We also have the translation elements
\begin{gather*}
T_{1}=Q_{1}Q_{0}^{-1}\overset{\pi}\mapsto
\begin{bmatrix}
 -z&0&0\\0&-z\inv&0\\0&0&1
\end{bmatrix}, \qquad
T_{2}=Q_{2}Q_{1}^{-1}\overset{\pi}\mapsto
\begin{bmatrix}
 1&0&0\\0&-z&0\\0&0&-z\inv
\end{bmatrix}.\end{gather*}
We define shifts on the series $X(z)$ by
\begin{gather}
X^{(\alpha)}(z)=(-1)^\alpha z^{\alpha}X(z)=\sum_{i\in\mathbb{Z}}(-1)^\alpha x_{i+\alpha}z^{-i-1},\qquad X=C,D,E,\qquad x=c,d,e.\label{eq:61}
\end{gather}
\begin{Remark}It is convenient to allow our series to be infinite in both directions. This causes no issues of convergence, if we think of the coefficients of these series to be formal variables. For example, if
\begin{gather*}
a(z)=\sum_{i\in\mathbb{Z}}a_{i}z^{-i-1},\qquad b(z)=\sum_{i\in\mathbb{Z}}b_{i}z^{-i-1}
\end{gather*}
then
\begin{gather*}
a(z)b(z)=\sum_{k\in\mathbb{Z}}\left(\sum_{i\in\mathbb{Z}}a_{i}b_{k-i-1}\right)z^{-k-1}
\end{gather*}
has as coefficient of $z^{-k-1}$ the well defined element
\begin{gather*}
\sum_{i\in\mathbb{Z}}a_{i}b_{k-i-1}\in\mathbb{C}[[a_{i},b_{i}]]_{i\in\mathbb{Z}}.
\end{gather*}
\end{Remark}

Analogous to the shifted group elements of the $2\times 2$ case, see \eqref{eq:7}, are the shifted group elements $g^{(\alpha,\beta)}=Q_0^{\alpha}Q_{1}^{\beta}g_{C,D,E}Q_{1}^{-\beta}Q_0^{-\alpha}$,
\begin{gather}
\pi\big(g_{C,D,E}^{(\alpha,\beta)}\big)=\begin{bmatrix}
 1 & 0&0\\
 C^{(\alpha-\beta)}(z) & 1&0\\
 D^{(\alpha)}(z)&E^{(\beta)}(z)&1
 \end{bmatrix}.
\label{eq:55}
\end{gather}
We then have (using $Q_{0}Q_{1}=-Q_{1}Q_{0}$)
\begin{gather*}
 Q_{0}\inv g^{(\alpha+1,\beta)}=g^{(\alpha,\beta)}Q_{0}\inv, \qquad Q_{1}\inv g^{(\alpha,\beta+1)}=g^{(\alpha,\beta)}Q_{1}\inv,
\end{gather*}
and the same relations with $\pi$ applied,
\begin{gather}
\pi (Q_{0})\inv \pi\big(g^{(\alpha+1,\beta)}\big)=\pi\big(g^{(\alpha,\beta)}\big)\pi(Q_{0})\inv,\nonumber\\
\pi(Q_{1})\inv \pi\big(g^{(\alpha,\beta+1)}\big) =\pi\big(g^{(\alpha,\beta)}\big)\pi(Q_{1})\inv.\label{eq:57}
\end{gather}

Similarly to the $2\times2$ case, the fundamental objects in the $3\times 3$ theory are the $\tau$-functions defined by
\begin{gather}
\tau_{k,\ell}^{(\alpha,\beta)}=\big\langle T_{1}^{k}T_{2}^{\ell}v_{0},g^{(\alpha,\beta)}v_{0}\big\rangle.\label{eq:62}
\end{gather}
Here $v_{0}$ is the vacuum vector in the three-component fermionic Fock space $F^{(3)}$, and $\langle \,,\, \rangle$ is the bilinear form, see Appendix~\ref{sec:semi-infinite-wedge}. (As in the $2\times 2$ case, in the appendix we define a basis for~$F^{(3)}$, and $\langle \,,\,\rangle$ is the bilinear form with respect to which these basis vectors are orthonormal.) Note that if we introduce another translation group element $T_{3}=Q_{2}Q_{0}\inv$ then we can write nonuniquely
\begin{gather*}
T_{1}^{k}T_{2}^{\ell}=(\pm1)T_{1}^{n_{c}}T_{2}^{n_{e}}T_{3}^{n_{d}},
\end{gather*}
where $k=n_{c}+n_{d}$, $\ell=n_{d}+n_{e}$, and we take $n_{c},n_{d},n_{e}\ge0$.

\begin{thm}\label{thm:n=3taufunctions} For all $\alpha,\beta\in\mathbb{Z}$ and $k,\ell\ge0$
\begin{gather*}
\tau^{\alpha,\beta}_{k,\ell}=\sum_{n_{c}+n_{d}=k,n_{d}+n_{e}=\ell \atop n_{c},n_{d},n_{e}\ge0}c^{(\alpha,\beta)}_{n_{c},n_{d},n_{e}},
\end{gather*}
where
\begin{gather*}
c^{(\alpha,\beta)}_{n_{c},n_{d},n_{e}}= \frac{(-1)^\frac{{n_{d}(n_{d}+1)}}{2}}{n_{c}!n_{d}!n_{e}!}\Res_{\mathbf{x},\mathbf{y},\mathbf{z}}\left(\prod_{i=1}^{n_{c}}C^{(\alpha-\beta)}(x_{i})
\prod_{i=1}^{n_{d}}D^{(\alpha)}(y_{i})\prod_{i=1}^{n_{e}}E^{(\beta)}(z_{i}) p_{n_{c},n_{d},n_{e}}\right),
\end{gather*}
and
\begin{gather*}
 p_{n_{c},n_{c},n_{e}} = \prod\limits_{1\le i<j\le n_{c}}(x_{i}-x_{j})^{2} \prod\limits_{1\le i<j\le n_{d}}(y_{i}-y_{j})^{2} \prod\limits_{1\le i<j\le n_{e}}(z_{i}-z_{j})^{2}\\
\hphantom{p_{n_{c},n_{c},n_{e}} =}{} \times\frac{ \prod\limits_{i=1}^{n_{c}} \prod\limits_{j=1}^{n_{d}}(x_{i}-y_{j})
 \prod\limits_{i=1}^{n_{d}} \prod\limits_{j=1}^{n_{e}}(y_{i}-z_{j})} { \prod\limits_{i=1}^{n_{c}} \prod\limits_{j=1}^{n_{e}}(x_{i}-z_{j})}.
\end{gather*}
\end{thm}

Here and from now on, we use the convention that we expand $\frac{1}{x-z}$ in positive powers of the second variable, so $\frac{1}{x-z}= \sum\limits_{i=0}^\infty\frac{z^i}{x^{i+1}}$.

We discuss the proof of the above theorem in Appendix~\ref{sec:proof-n=3}.

\subsection[Examples of $3\times 3$ $\tau$-functions]{Examples of $\boldsymbol{3\times 3}$ $\boldsymbol{\tau}$-functions}\label{sec:examples-3times-3}
\begin{enumerate}\itemsep=0pt
\item \label{item:8} $\tau_{k,\ell}^{(\alpha,\beta)}=0$ if $k<0$ or $\ell<0$.
\item $\tau_{0,0}^{(\alpha,\beta)}=1$.
\item \label{item:9} $\tau_{k,0}^{(\alpha,\beta)}=(-1)^{k(\alpha+\beta)}
\det\begin{bmatrix}
c_{\alpha-\beta} & c_{\alpha-\beta+1} & \cdots & c_{\alpha-\beta+k-1}\\
c_{\alpha-\beta+1} & c_{\alpha-\beta+2} & \cdots & c_{\alpha-\beta+k}\\
\vdots & \vdots & \cdots & \vdots\\
c_{\alpha-\beta+k-1} & c_{\alpha-\beta+k} & \cdots & c_{\alpha-\beta+2k-2}.
\end{bmatrix}$.
\item \label{item:7} $\tau_{0,\ell}^{(\alpha,\beta)}=(-1)^{\beta\ell}\det\begin{bmatrix}
e_{\beta} & e_{\beta+1} & \cdots & e_{\beta+\ell-1}\\
e_{\beta+1} & e_{\beta+2} & \cdots & e_{\beta+\ell}\\
\vdots & \vdots & \cdots & \vdots\\
e_{\beta+\ell-1} & e_{\beta+\ell} & \cdots & e_{\beta+2\ell-2}\\
\end{bmatrix}$.
\item $\tau_{1,1}^{(\alpha,\beta)}=(-1)^{\alpha}\left(-d_{\alpha}+\displaystyle \sum_{i=0}^{\infty}e_{\beta+i}c_{\alpha-\beta-i-1}\right)$.
\item
 $\tau_{1,2}^{(\alpha,\beta)}=(-1)^{\alpha+\beta}\left(e_{\beta}\displaystyle\sum_{i=1}^{\infty}e_{\beta+i+1}c_{\alpha-\beta-i-1}-e_{\beta+1}\right. $\\
 $\left.\hphantom{\tau_{1,2}^{(\alpha,\beta)}=}{} \times
 \displaystyle\sum_{i=0}^{\infty}e_{\beta+i+1}c_{\alpha-\beta-i-2}+e_{\beta+1}d_{\alpha}-e_{\beta}d_{\alpha+1}\right)$.
\item $\tau_{2,1}^{(\alpha,\beta)}=(-1)^\beta\left(\!c_{\alpha-\beta+1}\displaystyle \sum_{i=0}^{\infty}e_{\beta+i}c_{\alpha-\beta-i-1}-c_{\alpha-\beta}\displaystyle \sum_{i=0}^{\infty}e_{\beta+i}c_{\alpha-\beta-i}+c_{\alpha-\beta}d_{\alpha+1}-c_{\alpha-\beta+1}d_{\alpha}\!\right)$.
\end{enumerate}

\begin{Remark} \label{rem:degrees}Note that the summands $c^{(\alpha,\beta)}_{n_{c},n_{d},n_{e}}$ of the $\tau$-functions are of degree $n_{x}$ in the coefficients $x_{k}$ of the series $X(z)=\sum x_{k}z^{-k-1}$, for $x=c,d,e$, $X=C,D,E$.
\end{Remark}

\subsection[Birkhoff factorization for the $3\times 3$ case]{Birkhoff factorization for the $\boldsymbol{3\times 3}$ case} \label{sec:3x3birkh-fact}

Define centrally extended loop group elements
\begin{gather*}
g^{[k,\ell](\alpha,\beta)}=T^{-\ell}_{2}T_{1}^{-k}g^{(\alpha,\beta)},
\end{gather*}
and assume that they have a Birkhoff factorization \cite{PrSe:LpGrps} (see Appendix~\ref{sec:birkh-fact-n}):
\begin{gather*}
g^{[k,\ell](\alpha,\beta)}=g^{[k,\ell](\alpha,\beta)}_{-}g^{[k,\ell](\alpha,\beta)}_{0+},
\end{gather*}
where $\pi\big(g^{[k,\ell](\alpha,\beta)}_{-}\big)=1+O\big(z\inv\big)$ and $\pi\big(g^{[k,\ell](\alpha,\beta)}_{0+}\big)=A_{k,\ell}^{(\alpha,\beta)}+O(z)$, for $A_{k,\ell}^{(\alpha,\beta)}$ an invertible $z$ independent matrix. As in the $2\times2$ case, this assumption is justified precisely when $\tau^{(\alpha,\beta)}_{k,\ell}(g)=\big \langle v_{0},g^{[k,\ell](\alpha,\beta)}v_{0}\big\rangle$ is not zero, see the discussion at the beginning of Section~\ref{sec:birkh-fact}.

Now we want to display the negative component of $\pi\big(g^{[k,\ell](\alpha,\beta)}\big)$. As we did in the $2\times 2$ case, to calculate this we make some extra structure explicit, see Section~\ref{sec:birkh-fact} for the simpler situation.

Let $\mathcal{N}$ be the subgroup of elements of $\widetilde{\rm GL}_3 $ of the form~\eqref{eq:55}. We can think of the coefficients~$x_{k}$, $x=c,d,e$ as coordinates on $\mathcal{N}$, so{\samepage
\begin{gather*}
B=\mathbb{C}[[c_{k},d_{k},e_{k}]]_{k\in\mathbb{Z}}
\end{gather*}
is the coordinate ring of $\mathcal{N}$.}

We first define shifts acting on $B$: these are multiplicative maps given on generators by ($x,y\in\{c,d,e\}$)
\begin{gather*}
 S_{x}^{\alpha}\colon \ B\to B,\qquad S_{x}^{\alpha}(1)=0, \qquad S_{x}^{\alpha}(x_{k})=x_{k+\alpha},\nonumber\\
 \hphantom{S_{x}^{\alpha}\colon} \ S_{x}^{\alpha} (y_{k} )=y_{k},\qquad y\ne x , \qquad \alpha\in\mathbb{Z}.
\end{gather*}
We will often write $S_{x}^{\pm}$ for $S_{x}^{\pm1}$.

We also define \emph{shift fields}. These are multiplicative maps
\begin{gather*}
S_{x}^{\pm}(z)\colon \ B\to B\big[\big[z\inv\big]\big],
\end{gather*}
given by
\begin{gather*}
S_{x}^{\pm}(z)=\left(1-\frac{S_{x}^{+}}{z}\right)^{\pm1}.
\end{gather*}
\begin{thm}\label{Thm:birkh-3x3fact-tau} For $k,\ell\ge0$ and all $\alpha,\beta\in\mathbb{Z}$
 \begin{gather*}
\pi\big(g^{[k,\ell](\alpha,\beta)}_{-}\big)=\big(\Sigma\mathcal{T}_{k,\ell}^{(\alpha,\beta)}\big)/\tau_{k,\ell}^{(\alpha,\beta)},
\end{gather*}
where
\begin{gather*}
\Sigma=\begin{bmatrix}
 S_{c}^{+}(z)S_{d}^{+}(z)&0&0\\
 0&S^{-}_{c}(z)S^{+}_{e}(z)&0\\
 0&0&S^{-}_{d}(z)S_{e}^{-}(z)
\end{bmatrix},\\
\mathcal{T}_{k,\ell}^{(\alpha,\beta)}=
\begin{bmatrix}
\tau_{k,\ell}^{(\alpha,\beta)}&
(-1)^\ell \dfrac{\tau_{k-1,\ell}^{(\alpha,\beta)}}{z}&(-1)^{\ell+1}\dfrac{\tau_{k-1,\ell-1}^{(\alpha,\beta)}}{z}\vspace{1mm}\\
(-1)^\ell\dfrac{\tau_{k+1,\ell}^{(\alpha,\beta)}}z&
\tau_{k,\ell}^{(\alpha,\beta)}&\dfrac{\tau_{k,\ell-1}^{(\alpha,\beta)}}{z}\vspace{1mm}\\
(-1)^{\ell+1}\dfrac{\tau_{k+1,\ell+1}^{(\alpha,\beta)}}z&
\dfrac{\tau_{k,\ell+1}^{(\alpha,\beta)}}{z}&\tau_{k,\ell}^{(\alpha,\beta)}
\end{bmatrix}.
\end{gather*}
\end{thm}

We sketch the proof in Appendix~\ref{sec:birkh-fact-3tim}.

As in the $2\times 2$ case, we have now expressed the negative component of the Birkhoff factorization in terms of matrix elements of the centrally extended loop group, so we no longer need the central extension and we will simplify notation by writing $g^{[k,\ell](\alpha,\beta)}_{-}$ for $\pi\big(g^{[k,\ell](\alpha,\beta)}_{-}\big)$, $T_{i}$~for~$\pi(T_{i})$, $i=1,2$ and similarly $Q_{a}$ for $\pi(Q_{a})$, $a=0,1,2$. In particular, in the rest of this section we write
\begin{gather*}
T_{1}=\begin{bmatrix}
 -z&0&0\\0&-z\inv&0\\
0&0&1
\end{bmatrix}=Q_{1}Q_{0}\inv,\qquad
T_{2}=
\begin{bmatrix}
 1&0&0\\0&-z&0\\
0&0&-z\inv
\end{bmatrix}=Q_{2}Q_{1}\inv,
\end{gather*}
and
\begin{gather*} Q_{0}=
\begin{bmatrix}
 z\inv&0&0\\0&-1&0\\
0&0&-1
\end{bmatrix},\qquad
Q_{1}=
\begin{bmatrix}
 -1&0&0
\\0&z\inv&0\\
0&0&-1
\end{bmatrix},\qquad
Q_{2}=
\begin{bmatrix}
 -1&0&0
\\0&-1&0\\
0&0&z\inv
\end{bmatrix}.
\end{gather*}

\subsection[Matrix Baker functions and connection matrices, $3\times 3$ case]{Matrix Baker functions and connection matrices, $\boldsymbol{3\times 3}$ case}\label{sec:3x3matr-baker-funct}

Next, we define the Baker functions. These are now elements of the loop group of ${\rm GL}_{3},$ defined by
\begin{gather*}
 \Psi^{[k,\ell](\alpha,\beta)} =T_{1}^{k}T_{2}^{\ell}Q_0^{-\alpha}Q_1^{-\beta}g_{-}^{[k,\ell](\alpha,\beta)}.
\end{gather*}
Since the Baker functions are all invertible, they are related by (right) multiplication by connection matrices belonging to $\widetilde{\rm GL}_3 $. In particular, define
\begin{gather*}
 \Psi^{[k^{\prime},\ell^{\prime}](\alpha^{\prime},\beta^{\prime})}=\Psi^{[k,\ell](\alpha,\beta)}
 \Gamma^{[k^{\prime},\ell^{\prime}](\alpha^{\prime},\beta^{\prime})}_{[k,\ell](\alpha,\beta)},
\end{gather*}
so that
\begin{gather*}
 \Gamma^{[k^{\prime},\ell^{\prime}](\alpha^{\prime},\beta^{\prime})}_{[k,\ell](\alpha,\beta)}
=\big(g_{-}^{[k,\ell](\alpha,\beta)}\big)\inv Q_{0}^{x_{0}}Q_{1}^{x_{1}}Q_{2}^{x_{2}} g_{-}^{[k^{\prime},\ell^{\prime}](\alpha^{\prime},\beta^{\prime})},
\end{gather*}
where
\begin{gather*}
 x_{0}=k-k^{\prime}+\alpha-\alpha^{\prime}, \qquad
 x_{1}=k^{\prime}-k+\ell-\ell^{\prime}+\beta-\beta^{\prime}, \qquad
 x_{2}=\ell^{\prime}-\ell.
\end{gather*}
The simplest connection matrices are those where $(x_{0},x_{1},x_{2})$ has two zero components and the other absolute value 1. We therefore define the elementary connection matrices
\begin{gather}
 V_{k,\ell}^{(\alpha_{+},\beta)}=\Gamma^{[k,\ell](\alpha+1,\beta)}_{[k,\ell](\alpha,\beta)}= \big(g_{-}^{[k,\ell](\alpha,\beta)}\big)\inv Q_{0}\inv g_{-}^{[k,\ell](\alpha+1,\beta)},\nonumber\\
 V_{k,\ell}^{(\alpha,\beta_{+})}=\Gamma^{[k,\ell](\alpha,\beta+1)}_{[k,\ell](\alpha,\beta)}= \big(g_{-}^{[k,\ell](\alpha,\beta)}\big)\inv Q_{1}\inv g_{-}^{[k,\ell](\alpha,\beta+1)},\nonumber\\
 W_{k,\ell}^{(\alpha_{+},\beta)}=\Gamma^{[k-1,\ell](\alpha+1,\beta)}_{[k,\ell](\alpha,\beta)}= \big(g_{-}^{[k,\ell](\alpha,\beta)}\big)\inv Q_{1}\inv g_{-}^{[k-1,\ell](\alpha+1,\beta)},\nonumber\\
 W_{k,\ell}^{(\alpha,\beta_{+})}=\Gamma^{[k,\ell-1](\alpha,\beta+1)}_{[k,\ell](\alpha,\beta)}= \big(g_{-}^{[k,\ell](\alpha,\beta)}\big)\inv Q_{2}\inv g_{-}^{[k,\ell-1](\alpha,\beta+1)}.\label{eq:67}
\end{gather}
Also define translation matrices
\begin{gather}
U^{(\alpha,\beta)}_{k_{+},\ell}=\Gamma^{[k+1,\ell](\alpha,\beta)}_{[k,\ell](\alpha,\beta)}=\big(g_{-}^{[k,\ell](\alpha,\beta)}\big)\inv T_{1} \big(g_{-}^{[k+1,\ell](\alpha,\beta)}\big),\nonumber\\
U^{(\alpha,\beta)}_{k,\ell_{+}}=\Gamma^{[k,\ell+1](\alpha,\beta)}_{[k,\ell](\alpha,\beta)}=\big(g_{-}^{[k,\ell](\alpha,\beta)}\big)\inv T_{2} \big(g_{-}^{[k,\ell+1](\alpha,\beta)}\big).\label{eq:68}
\end{gather}
Pictorially, fixing $k$ and $\alpha$, we have
\begin{equation*}
\begin{tikzpicture}[baseline= (a).base]
\node[scale=.77] (a) at (0,0){
 \begin{tikzcd}
\Psi^{[k,\ell](\alpha,\beta-1)}
\arrow{rr}{U_{k,\ell_{+}}^{(\alpha,\beta-1)}}
\ar{rd}{V_{k,\ell}^{(\alpha,\beta-1_{+})}}
&{}&
\Psi^{[k,\ell+1](\alpha,\beta-1)}
\ar{rr}{U_{k,\ell+1_+}^{(\alpha,\beta-1)}}
\ar{rd}{V_{k,\ell+1}^{(\alpha,\beta-1_{+})}}
\ar{ld}{W_{k,\ell+1}^{(\alpha,\beta-1_{+})}}&{}&\Psi^{[k,\ell+2](\alpha,\beta-1)}
\ar{dl}{W_{k,\ell+2}^{(\alpha,\beta-1_{+})}}\\
{}& \Psi^{[k,\ell](\alpha,\beta)}
 \ar{rr}[swap]{U_{k,\ell_{+}}^{(\alpha,\beta)}}
 \ar{dr}[swap]{V_{k,\ell}^{(\alpha,\beta_{+})}}
\ar{dl}[swap]{W_{k,\ell}^{(\alpha,\beta_{+})}}&{} &
\Psi^{[k,\ell+1](\alpha,\beta)}
\ar{ld}{W_{k,\ell+1}^{(\alpha,\beta_{+})}}
\ar{rd}{V_{k,\ell+1}^{(\alpha,\beta_{+})}}&{}&
 \\
\Psi^{[k,\ell-1](\alpha,\beta+1)}
\arrow{rr}[swap]{U_{k,\ell-1_{+}}^{(\alpha,\beta+1)}}
&{}&
\Psi^{[k,\ell](\alpha,\beta+1)}
\ar{rr}[swap]{U_{k,\ell_{+}}^{(\alpha,\beta+1)}}&{}&\Psi^{[k,\ell+1](\alpha,\beta+1)}
\end{tikzcd}
};
\end{tikzpicture}
\end{equation*}
Walking around the triangles in this and the similar diagram where $\ell$, $\beta$ are fixed we see that we get factorizations of the elementary translation matrices $U_{k_{+},\ell}^{(\alpha,\beta)}$, $U_{k,\ell_{+}}^{(\alpha,\beta)}$. For instance,
\begin{gather}
 U_{k,\ell_{+}}^{(\alpha,\beta)}=V_{k,\ell}^{(\alpha,\beta_{+})}\big(W_{k,\ell+1}^{(\alpha,\beta_{+})}\big)^{-1} =\big(W_{k,\ell+1}^{(\alpha,\beta-1_{+})}\big)^{-1}V_{k,\ell+1}^{(\alpha,\beta-1_{+})},\nonumber\\
 U_{k_{+},\ell}^{(\alpha,\beta)}=V_{k,\ell}^{(\alpha_{+},\beta)}\big(W_{k+1,\ell}^{(\alpha_{+},\beta)}\big)^{-1} =\big(W_{k+1,\ell}^{(\alpha-1_{+},\beta)}\big)^{-1}V_{k+1,\ell}^{(\alpha-1_{+},\beta)}.\label{eq:69}
\end{gather}
We will argue that all identities for the connection matrices are the result of those in \eqref{eq:69}.

\begin{lem}\label{connectmatrixprod} Any connection matrix $\Gamma^{[k^{\prime},\ell^{\prime}](\alpha^{\prime},\beta^{\prime})} _{[k,\ell](\alpha,\beta)}$ is a product of the four types of elementary connection matrices~\eqref{eq:67}.
\end{lem}

\begin{proof} We need to show that we can move from $\Psi^{[k,\ell](\alpha,\beta)}$ to $\Psi^{[k^{\prime},\ell^{\prime}](\alpha^{\prime},\beta^{\prime})}$ just using the elementary connection matrices. First of all, the translation matrices \eqref{eq:68} are products of elementary connection matrices, see~\eqref{eq:69}. We can move from $\Psi^{[k,\ell](\alpha,\beta)}$ to $\Psi^{[k^{\prime},\ell^{\prime}](\alpha,\beta)}$ using just $U_{{k_{i}}_{+},\ell_{i}}$ and/or $U_{k_{i},{\ell_{i}}_{+}}$, keeping $(\alpha,\beta)$ fixed. Then we use the $V^{(\alpha_{+},\beta)}$ and/or $V^{(\alpha,\beta_{+})}$ to adjust the $(\alpha,\beta)$ to $(\alpha^{\prime},\beta^{\prime})$ to reach $\Psi^{[k^{\prime},\ell^{\prime}](\alpha^{\prime},\beta^{\prime})}$.
\end{proof}

The expression in Lemma \ref{connectmatrixprod}, $\Gamma=\Gamma^{[k^{\prime},\ell^{\prime}](\alpha^{\prime},\beta^{\prime})}_{[k,\ell](\alpha,\beta)}$ as a product of elementary connection matrices is of course not unique. Each path from $\Psi^{[k,\ell](\alpha,\beta)}$ to $\Psi^{[k^{\prime},\ell^{\prime}](\alpha^{\prime},\beta^{\prime})}$ in the lattice of Baker functions with diagonal or anti-diagonal steps gives a product expression for $\Gamma$: the diagonal steps give~$V$ factors and the anti-diagonal steps give~$W$ factors. Now it should be clear that any two paths from $\Psi^{[k,\ell](\alpha,\beta)}$ to $\Psi^{[k^{\prime},\ell^{\prime}](\alpha^{\prime},\beta^{\prime})}$ can be deformed into each other by moves
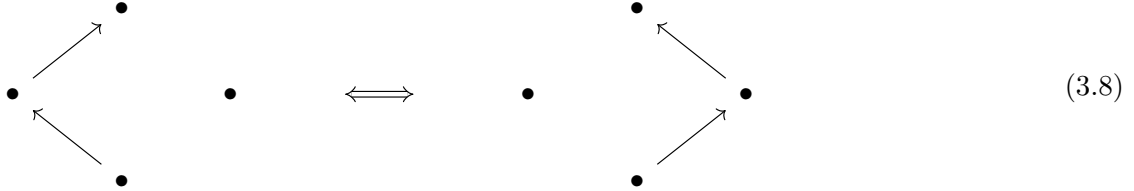
\begin{equation}
\begin{tikzcd}
 {}&\bullet&{}&{}&{}&{}&\bullet&{}\\
\bullet \ar{ru}{}&{}&\bullet&{}\ar[Leftrightarrow]{r}&{}&\bullet&{}&\bullet\ar{lu}
\\
 {}&\bullet\ar{lu}&{}&{}&{}&{}&\bullet\ar{ru}&{}
\end{tikzcd}\label{eq:70}
\end{equation}
or
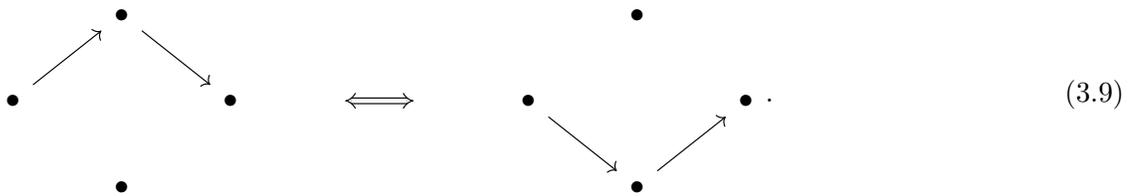
\begin{equation}
\begin{tikzcd}
 {}&\bullet
\ar{rd}&{}&{}&{}&{}&\bullet&{}\\
 \bullet\ar{ru}{}&{}&\bullet&\ar[Leftrightarrow]{r}&{}&\bullet\ar{rd}{}&{}&\bullet\\
{}&\bullet&{}&{}&{}&{}&\bullet\ar{ru}&{}
\end{tikzcd}.\label{eq:71}
\end{equation}
The moves \eqref{eq:71} correspond to identities \eqref{eq:69} and moves \eqref{eq:70} correspond to similar equations of the form $V W=W V$ (without inverses on~$W$). These last equations will be equivalent to those in~\eqref{eq:69}. So the upshot is that all equations obtained by writing an arbitrary connection matrix $\Gamma$ as a~product of elementary connection matrices follow from~\eqref{eq:69}.

We will therefore concentrate on \eqref{eq:69}. In particular, we will see that these equations will imply the equations for our $\tau$-functions~(\ref{eq:48}).

We first check that our elementary connection matrices and translation matrices do not contain any negative powers of~$z$, which is not obvious from the definitions~\eqref{eq:67} and~\eqref{eq:68}.

The following lemma tells us that the elementary connection matrices and translation matrices contain only $z^{k}$ for $k\ge0$.
\begin{lem}\label{lem:8}For the elementary connection matrices, we have the positive expressions
 \begin{gather}
 V_{k,\ell}^{(\alpha_{+},\beta)}= g^{[k,\ell](\alpha,\beta)}_{0+}Q_{0}\inv\big(g^{[k,\ell](\alpha+1,\beta)}_{0+}\big)\inv,\nonumber\\
 V_{k,\ell}^{(\alpha,\beta_{+})}= g^{[k,\ell](\alpha,\beta)}_{0+}Q_{1}\inv\big(g^{[k,\ell](\alpha,\beta+1)}_{0+}\big)\inv,\nonumber\\
 W_{k,\ell}^{(\alpha_{+},\beta)}=g^{[k,\ell](\alpha,\beta)}_{0+}Q_{0}\inv\big(g^{[k-1,\ell](\alpha+1,\beta)}_{0+}\big)\inv ,\nonumber\\
 W_{k,\ell}^{(\alpha,\beta_{+})}=g^{[k,\ell](\alpha,\beta)}_{0+}Q_{1}\inv\big(g^{[k,\ell-1](\alpha,\beta+1)}_{0+}\big)\inv.\label{eq:72}
\end{gather}
Similarly, for the translation matrices
\begin{gather*}
 U^{(\alpha,\beta)}_{k_{+},\ell}=g_{0+}^{[k+1,\ell](\alpha,\beta)}\big(g_{0+}^{[k,\ell](\alpha,\beta)}\big)\inv,\qquad
 U^{(\alpha,\beta)}_{k,\ell_{+}}=g_{0+}^{[k,\ell+1](\alpha,\beta)}\big(g_{0+}^{[k,\ell](\alpha,\beta)}\big)\inv.
\end{gather*}
\end{lem}

\begin{proof} From \eqref{eq:57} it follows that
\begin{gather*}
 Q_{0}\inv g^{[k,\ell](\alpha+1,\beta)}_{-}g^{[k,\ell](\alpha+1,\beta)}_{0+}=g^{[k,\ell](\alpha,\beta)}_{-}g^{[k,\ell](\alpha,\beta)}_{0+}Q_{0}\inv,\\
 Q_{1}\inv g^{[k,\ell](\alpha,\beta+1)}_{-}g^{[k,\ell](\alpha,\beta+1)}_{0+}=g^{[k,\ell](\alpha,\beta)}_{-}g^{[k,\ell](\alpha,\beta)}_{0+}Q_{1}\inv,
\end{gather*}
from which the result for the $V$ matrices follows by rearranging factors.

Similarly, using $Q_{1}\inv T_{1}=Q_{0}\inv, Q_{2}\inv T_{2}=Q_{1}\inv$ and again \eqref{eq:57} we find
\begin{gather*}
 Q_{1}\inv g^{[k-1,\ell](\alpha+1,\beta)}_{-}g^{[k-1,\ell](\alpha+1,\beta)}_{0+} =g^{[k,\ell](\alpha,\beta)}_{-}g^{[k,\ell](\alpha,\beta)}_{0+}Q_{0}\inv,\\
 Q_{2}\inv g^{[k,\ell-1](\alpha,\beta+1)}_{-}g^{[k,\ell-1](\alpha,\beta+1)}_{0+}=g^{[k,\ell](\alpha,\beta)}_{-}g^{[k,\ell](\alpha,\beta)}_{0+}Q_{1}\inv,
\end{gather*}
giving the result for the $W$ matrices.

Finally the positive expression for $U$ matrices follows from
 \begin{gather*}
 g^{[k+1,\ell](\alpha,\beta)}_{-}g^{[k+1,\ell](\alpha,\beta)}_{0+}=T_1\inv g^{[k,\ell](\alpha,\beta)}_{-}g^{[k,\ell](\alpha,\beta)}_{0+},\\
 g^{[k,\ell+1](\alpha,\beta)}_{-}g^{[k,\ell+1](\alpha,\beta)}_{0+}=T_2\inv g^{[k,\ell](\alpha,\beta)}_{-}g^{[k,\ell](\alpha,\beta)}_{0+},
\end{gather*}
and rearranging factors.
\end{proof}

As in the simpler, $2\times2$ case, this lemma allows us to calculate the connection matrices easily in terms of the $\tau$-functions.

First note that, similarly to the $2\times 2$ case, Theorem \ref{Thm:birkh-3x3fact-tau} allows us to expand $g_{-}^{[k,\ell](\alpha,\beta)}$ and its inverse up to order $z\inv$ as (we suppress the shift $(\alpha,\beta)$)
\begin{gather}\label{eq:65}
g_{-}^{[k,\ell]}=
\begin{bmatrix}
 1+O\big(z^{-1}\big) & \dfrac{(-1)^\ell}{zh_{\underline{k-1},\ell}}+O\big(z^{-2}\big) &
 \dfrac{(-1)^{\ell}}{zh_{\underline{k-1},\underline{\ell-1}}}+O\big(z^{-2}\big)\vspace{1mm}\\
 \dfrac{(-1)^\ell h_{\underline{k},\ell}}{z}+O\big(z^{-2}\big) &
 1+O\big(z^{-1}\big) &
\dfrac{1}{zh_{k,\underline{\ell-1}}}+O\big(z^{-2}\big)\vspace{1mm}\\
 \dfrac{ (-1)^{\ell+1}h_{\underline{k},\underline {\ell}}}{z}+O\big(z^{-2}\big) &
 \dfrac{h_{k,\underline{\ell}}}{z}+O\big(z^{-2}\big) & 1+O\big(z^{-1}\big)
\end{bmatrix},
\end{gather}
where $O(z^i)$ are terms with power of $z$ equal to $i$ or lower, and we define quotients of $\tau$-functions as
\begin{gather*}
h_{\underline{k},\ell}=\frac{\tau_{k+1,\ell}}{\tau_{k,\ell}},\qquad
h_{{k},\underline{\ell}}=\frac{\tau_{k,\ell+1}}{\tau_{k,\ell}},\qquad
h_{\underline{k},\underline{\ell}}=\frac{\tau_{k+1,\ell+1}}{\tau_{k,\ell}}.
\end{gather*}
This formula then gives the following formula for $\big(g_{-}^{[k,\ell]}\big)^{-1}$
\begin{gather}\label{eq:75}
\big(g_{-}^{[k,\ell]}\big)^{-1}=
\begin{bmatrix}
 1+O\big(z^{-1}\big) &
 \dfrac{(-1)^{\ell+1}}{zh_{\underline{k-1},\ell}}+O\big(z^{-2}\big)
& \dfrac{(-1)^{\ell+1} }{zh_{\underline{k-1},\underline{\ell-1}}}+O\big(z^{-2}\big)\vspace{1mm}\\
\dfrac{(-1)^{\ell+1}h_{\underline{k},\ell}}{z}+O\big(z^{-2}\big) &
 1+O\big(z^{-1}\big) & \dfrac{-1}{zh_{k,\underline{\ell-1}}}+O\big(z^{-2}\big)\vspace{1mm}\\
\dfrac{ (-1)^{\ell}h_{\underline{k},\underline{\ell}}}{z}+O\big(z^{-2}\big) &
\dfrac{ -h_{k,\underline{\ell}}}{z}+O\big(z^{-2}\big) & 1+O\big(z^{-1}\big)
\end{bmatrix}.
\end{gather}

\subsection{Explicit formulae for connection matrices}\label{sec:expl-form-conn}

\begin{lem}\label{lem:9}
 \begin{gather*}V^{(\alpha_{+},\beta)}_{k,\ell}=
 \begin{bmatrix}
 z+A &
 \dfrac{(-1)^\ell}{h_{\underline{k-1},{\ell}}^{(\alpha+1,\beta)}}
& \dfrac{(-1)^{\ell}}{h_{\underline{k-1},\underline{\ell-1}}^{(\alpha+1,\beta)}}\vspace{2mm}\\
 (-1)^{\ell+1} h_{\underline{k},\ell}^{(\alpha,\beta)} & -1 & 0\vspace{2mm}\\
 (-1)^{\ell}h_{\underline{k},\underline{\ell}}^{(\alpha,\beta)} & 0 & -1
 \end{bmatrix},\qquad A=\frac{h_{\underline{k},\ell}^{(\alpha,\beta)}}
 {h_{\underline{k-1},\ell}^{(\alpha+1,\beta)}}
 -\frac{h_{\underline{k},\underline{\ell}}^{(\alpha,\beta)}}
 {h_{\underline{k-1},\underline{\ell-1}}^{(\alpha+1,\beta)}},\\
V^{(\alpha,\beta_{+})}_{k,\ell}=
\begin{bmatrix}
 -1& \dfrac{(-1)^{\ell+1}}{h^{(\alpha,\beta)}_{\underline{k-1},\ell}}&0\vspace{2mm}\\
 (-1)^\ell h^{(\alpha,\beta+1)}_{\underline{k},\ell}&z+B&\dfrac{1}{h^{(\alpha,\beta+1)}_{k,\underline{\ell-1}}}\vspace{2mm}\\
 0&-h_{k,\underline{\ell}}^{(\alpha,\beta)}&-1
\end{bmatrix},\qquad B=\frac{h_{{k},\underline{\ell}}^{(\alpha,\beta)}}{h_{k,\underline{\ell-1}}^{(\alpha,\beta+1)}}+
\frac{h_{\underline{k},\ell}^{(\alpha,\beta+1)}}{h_{\underline{k-1},\ell}^{(\alpha,\beta)}},\\
W^{(\alpha_{+},\beta)}_{k,\ell}=
\begin{bmatrix}
 -1&\dfrac{(-1)^{\ell+1}}{h_{\underline{k-1},\ell}^{(\alpha,\beta)}}&0\vspace{2mm}\\
 (-1)^\ell h_{\underline{k-1},\ell}^{(\alpha+1,\beta)}&z+C&\dfrac{1} {h_{k-1,\underline{\ell-1}}^{(\alpha+1,\beta)}}\vspace{2mm}\\
0&-h_{k,\underline{\ell}}^{(\alpha,\beta)}&-1
\end{bmatrix},\qquad C=\frac{h_{k,\underline{\ell}}^{(\alpha,\beta)}}{h_{k-1,\underline{\ell-1}}^{(\alpha+1,\beta)}}
+\frac{h_{\underline{k-1},\ell}^{(\alpha+1,\beta)}}{h_{\underline{k-1},\ell}^{(\alpha,\beta)}},\\
W^{(\alpha,\beta_{+})}_{k,\ell}=
\begin{bmatrix}
 -1&0&\dfrac{(-1)^{\ell+1}}{h_{\underline{k-1},\underline{\ell-1}}^{(\alpha,\beta)}}\vspace{2mm}\\
 0&-1&\dfrac{-1}{h_{k,\underline{\ell-1}}^{(\alpha,\beta)}}\vspace{2mm}\\
 (-1)^{\ell}h_{\underline{k},\underline{\ell-1}}^{(\alpha,\beta+1)}& h_{k,\underline{\ell-1}}^{(\alpha,\beta+1)}&
 z+D
\end{bmatrix},\qquad D=\frac{h_{\underline{k},\underline{\ell-1}}^{(\alpha,\beta+1)}}
 {h_{\underline{k-1},\underline{\ell-1}}^{(\alpha,\beta)}}
 +\frac{h_{k,\underline{\ell-1}}^{(\alpha,\beta+1)}}
 {h_{k,\underline{\ell-1}}^{(\alpha,\beta)}}.
 \end{gather*}
\end{lem}

\begin{proof}For the off-diagonal entries, use the expressions \eqref{eq:65} and~\eqref{eq:75} for $g_{-}$ and $(g_{-})\inv$ in the definitions~\eqref{eq:72} of the elementary connection matrices, and then use the positivity, Lemma~\ref{lem:8}, to discard all terms with $z\inv$ or lower. For the diagonal terms, use that the determinant of all elementary connection matrices is $z$. We comment that the inverses of $W_{k,\ell}^{(\alpha_+,\beta)}$ and $W_{k,\ell}^{(\alpha,\beta_+)}$ can now easily be calculated.
\end{proof}

For the translation matrices, we obtain the following expressions
\begin{gather*}
U_{k_{+},\ell}^{(\alpha,\beta)}=V_{k,\ell}^{(\alpha_+,\beta)}\big(W_{k+1,\ell}^{(\alpha_+,\beta)}\big)^{-1}\\
\hphantom{U_{k_{+},\ell}^{(\alpha,\beta)}}{} =\begin{bmatrix}
 -z-\dfrac{h_{\underline{k},\ell}^{(\alpha+1,\beta)}}{h_{\underline{k},\ell}^{(\alpha,\beta)}}-\dfrac{h_{\underline{k},\ell}^{(\alpha,\beta)}}{h_{\underline{k-1},\ell}^{(\alpha+1,\beta)}} +\dfrac{h_{\underline{k},\underline{\ell}}^{(\alpha,\beta)}}{h_{\underline{k-1},\underline{\ell-1}}^{(\alpha+1,\beta)}}& \dfrac{(-1)^{\ell+1}}{h_{\underline{k},\ell}^{(\alpha,\beta)}} & \dfrac{(-1)^{\ell+1}}{h_{\underline{k},{\ell}}^{(\alpha,\beta)}h_{k,\underline{\ell-1}}^{(\alpha+1,\beta)}}+\frac{(-1)^{\ell+1}}{h_{\underline{k-1},\underline{\ell-1}}^{(\alpha+1,\beta)}}\vspace{2mm}\\
 (-1)^{\ell}h_{\underline{k},\ell}^{(\alpha,\beta)}& 0 & 0\vspace{2mm}\\
(-1)^{\ell+1}h_{\underline{k},\underline{\ell}}^{(\alpha,\beta)}& 0 & 1
\end{bmatrix}. \end{gather*}
Equivalently,
\begin{gather*}
 U_{k_{+},\ell}^{(\alpha,\beta)}=\big(W_{k+1,\ell}^{(\alpha-1_+,\beta)}\big)^{-1}V_{k+1,\ell}^{(\alpha-1_+,\beta)} \nonumber\\
 \hphantom{U_{k_{+},\ell}^{(\alpha,\beta)}}{} =
\begin{bmatrix}
 -z-\dfrac{h_{\underline{k+1},\ell}^{(\alpha-1,\beta)}}{h_{\underline{k},\ell}^{(\alpha,\beta)}}-\dfrac{h_{\underline{k},\ell}^{(\alpha,\beta)}}{h_{\underline{k},\ell}^{(\alpha-1,\beta)}} +\dfrac{h_{\underline{k+1},\underline{\ell}}^{(\alpha-1,\beta)}}{h_{\underline{k},\underline{\ell-1}}^{(\alpha,\beta)}}& \dfrac{(-1)^{\ell+1}}{h_{\underline{k},\ell}^{(\alpha,\beta)}}& \dfrac{(-1)^{\ell+1}}{h_{\underline{k},\underline{\ell-1}}^{(\alpha,\beta)}}\vspace{2mm}\\
 (-1)^{\ell}h_{\underline{k},\ell}^{(\alpha,\beta)}& 0 &0\vspace{2mm}\\
 (-1)^{\ell+1}h_{\underline{k},{\ell}}^{(\alpha,\beta)}h_{{k+1},{\underline{\ell}}}^{(\alpha-1,\beta)}+(-1)^{\ell+1}h_{\underline{k+1},{\underline{\ell}}}^{(\alpha-1,\beta)}& 0 & 1
\end{bmatrix},
\\
U_{k,\ell_{+}}^{(\alpha,\beta)}=V_{k,\ell}^{(\alpha,\beta_+)}\big(W_{k,\ell+1}^{(\alpha,\beta_+)}\big)^{-1} \nonumber\\
 {}= \begin{bmatrix}
 1& \dfrac{(-1)^{\ell}}{h_{\underline{k-1},\ell}^{(\alpha,\beta)}}& 0\vspace{2mm}\\
 (-1)^{\ell+1}h_{\underline{k},\ell}^{(\alpha,\beta+1)}+(-1)^{\ell}\dfrac{h_{\underline{k},\underline{\ell}}^{(\alpha,\beta+1)}}{h_{k,\underline{\ell}}^{(\alpha,\beta)}}& -z-\dfrac{h_{k,\underline{\ell}}^{(\alpha,\beta)}}{h_{k,\underline{\ell-1}}^{(\alpha,\beta+1)}}-\dfrac{h_{\underline{k},\ell}^{(\alpha,\beta+1)}}{h_{\underline{k-1},\ell}^{(\alpha,\beta)}} -\dfrac{h_{k,\underline{\ell}}^{(\alpha,\beta+1)}}{h_{k,\underline{\ell}}^{(\alpha,\beta)}} & -\dfrac{1}{h_{k,\underline{\ell}}^{(\alpha,\beta)}}\vspace{2mm}\\
 0& h_{k,\underline{\ell}}^{(\alpha,\beta)} & 0
\end{bmatrix}.\nonumber \end{gather*}
Equivalently,
 \begin{gather*}
 U_{k,\ell_{+}}^{(\alpha,\beta)}=\big(W_{k,\ell+1}^{(\alpha,\beta-1_+)}\big)^{-1}V_{k,\ell+1}^{(\alpha,\beta-1_+)} \\
 \hphantom{U_{k,\ell_{+}}^{(\alpha,\beta)}}{} =\begin{bmatrix}
 1& \dfrac{(-1)^{\ell+1}}{h_{\underline{k-1},\ell+1}^{(\alpha,\beta-1)}}+(-1)^\ell\dfrac{h_{k,\underline{\ell}}^{(\alpha,\beta)}}{h_{\underline{k-1},\underline{\ell}}^{(\alpha,\beta-1)}}& 0\vspace{2mm}\\
 (-1)^{\ell}h_{\underline{k},\ell+1}^{(\alpha,\beta)}& -z-\dfrac{h_{k,\underline{\ell+1}}^{(\alpha,\beta-1)}} {h_{k,\underline{\ell}}^{(\alpha,\beta)}}-\dfrac{h_{\underline{k},\ell+1}^{(\alpha,\beta)}}{h_{\underline{k-1},\ell+1}^{(\alpha,\beta-1)}}-\dfrac{h_{k,\underline{\ell}}^{(\alpha,\beta)}} {h_{{k},\underline{\ell}}^{(\alpha,\beta-1)}} & -\dfrac{1}{h_{k,\underline{\ell}}^{(\alpha,\beta)}}\vspace{2mm}\\
 0& h_{k,\underline{\ell}}^{(\alpha,\beta)} & 0
\end{bmatrix}. \end{gather*}

\subsection{Difference equations from factorizations}\label{sec:diff-equat-frpm}
\begin{thm}\label{thm:3by3Qsystem} The $\tau$-functions defined by \eqref{eq:62} satisfy the following system of four equations, referred to as the $3Q$-system. For all $\alpha,\beta \in{\mathbb Z}$ and $k,\ell \geq 0$
\begin{gather*}
 \tag{\texttt{[0]}} \big(\tau_{k,\ell}^{(\alpha+1,\beta)}\big)^2=\tau_{k,\ell}^{(\alpha,\beta)}\tau_{k,\ell}^{(\alpha+2,\beta)}+\tau_{k+1,\ell+1}^{(\alpha,\beta)} \tau_{k-1,\ell-1}^{(\alpha+2,\beta)}-\tau_{k+1,\ell}^{(\alpha,\beta)}\tau_{k-1,\ell}^{(\alpha+2,\beta)},\\
 \tag{\texttt{[2]}}\tau_{k-1,\ell-1}^{(\alpha+2,\beta)}\tau_{k+1,\ell}^{(\alpha+1,\beta)}+\tau_{k,\ell}^{(\alpha+1,\beta)}\tau_{k,\ell-1}^{(\alpha+2,\beta)}=\tau_{k,\ell-1}^{(\alpha+1,\beta)} \tau_{k,\ell}^{(\alpha+2,\beta)},\\
 \tag{\texttt{[1,1]}} \big(\tau_{k,\ell}^{(\alpha,\beta+1)}\big)^2=\tau_{k,\ell}^{(\alpha,\beta)}\tau_{k,\ell}^{(\alpha,\beta+2)}-\tau_{k,\ell-1}^{(\alpha,\beta+2)} \tau_{k,\ell+1}^{(\alpha,\beta)}-\tau_{k+1,\ell}^{(\alpha,\beta+2)}\tau_{k-1,\ell}^{(\alpha,\beta)},\\
 \tag{\texttt{[1]}} \tau_{k-1,\ell+1}^{(\alpha,\beta)}\tau_{k,\ell}^{(\alpha,\beta+1)}+\tau_{k-1,\ell}^{(\alpha,\beta+1)}\tau_{k,\ell+1}^{(\alpha,\beta)}=\tau_{k-1,\ell}^{(\alpha,\beta)}\tau_{k,\ell+1}^{(\alpha,\beta+1)}.
\end{gather*}
\end{thm}

\begin{proof}We substitute into the two factorizations \eqref{eq:69}, the results of Lemma~\ref{lem:9}. We find rather complicated rational expressions in the $\tau$-functions. From the two different expressions for $U_{k_{+},\ell}^{(\alpha,\beta)},$ we obtain
 \begin{gather} \label{hkeq1}
 -\frac{h_{\underline{k},\ell}^{(\alpha+1,\beta)}}{h_{\underline{k},\ell}^{(\alpha,\beta)}}-\frac{h_{\underline{k},\ell}^{(\alpha,\beta)}}{h_{\underline{k-1},\ell}^{(\alpha+1,\beta)}} +\frac{h_{\underline{k},\underline{\ell}}^{(\alpha,\beta)}}{h_{\underline{k-1},\underline{\ell-1}}^{(\alpha+1,\beta)}} =-\frac{h_{\underline{k+1},\ell}^{(\alpha-1,\beta)}}{h_{\underline{k},\ell}^{(\alpha,\beta)}}-\frac{h_{\underline{k},\ell}^{(\alpha,\beta)}}{h_{\underline{k},\ell}^{(\alpha-1,\beta)}} +\frac{h_{\underline{k+1},\underline{\ell}}^{(\alpha-1,\beta)}}{h_{\underline{k},\underline{\ell-1}}^{(\alpha,\beta)}},\\
 \label{hkeq2}\frac{(-1)^{\ell+1}}{h_{\underline{k},{\ell}}^{(\alpha,\beta)}h_{k,\underline{\ell-1}}^{(\alpha+1,\beta)}}+\frac{(-1)^{\ell+1}}{h_{\underline{k-1},\underline{\ell-1}}^{(\alpha+1,\beta)}} =\frac{(-1)^{\ell+1}}{h_{\underline{k},\underline{\ell-1}}^{(\alpha,\beta)}},
 \end{gather}
 and
 \begin{gather} \label{hkeq3}
 (-1)^{\ell+1}h_{\underline{k},\underline{\ell}}^{(\alpha,\beta)}=(-1)^{\ell+1}h_{\underline{k},{\ell}}^{(\alpha,\beta)}h_{{k+1}, {\underline{\ell}}}^{(\alpha-1,\beta)}+(-1)^{\ell+1}h_{\underline{k+1},{\underline{\ell}}}^{(\alpha-1,\beta)}.
 \end{gather}
Expressing equations (\ref{hkeq1}), (\ref{hkeq2}), (\ref{hkeq3}) in terms of $\tau$-functions, we get nontrivial equations in 3 components:
\begin{gather}
 -\frac{\tau_{k + 2, \ell + 1}^{ (\alpha - 1,\beta)} \tau_{k,\ell-1}^{(\alpha,\beta)}}
 {\tau_{k + 1, \ell}^{(\alpha - 1,\beta)} \tau_{k + 1, \ell}^{ (\alpha,\beta)}} +
 \frac{\tau_{k + 1, \ell + 1}^{(\alpha,\beta)} \tau_{k - 1, \ell- 1}^{(\alpha + 1,\beta)}}
 {\tau_{k, \ell}^{(\alpha + 1,\beta)} \tau_{k, \ell}^{(\alpha,\beta)}}-
 \frac{\tau_{k + 1, \ell}^{(\alpha,\beta)} \tau_{k - 1,\ell}^{(\alpha+ 1,\beta)}}
 {\tau_{k, \ell}^{(\alpha + 1,\beta)}\tau_{k, \ell}^{(\alpha,\beta)}} \notag\\
\tag{\texttt{[0,0]}}\qquad{} +
 \frac{\tau_{k + 1,\ell}^{(\alpha,\beta)} \tau_{k, \ell}^{(\alpha - 1,\beta)}}
 {\tau_{k + 1, \ell}^{(\alpha - 1,\beta)} \tau_{k, \ell}^{(\alpha,\beta)}} +
 \frac{\tau_{k + 2, \ell}^{(\alpha - 1,\beta)} \tau_{k, \ell}^{(\alpha,\beta)}}
 {\tau_{k + 1, \ell}^{(\alpha - 1,\beta)} \tau_{k + 1,\ell}^{(\alpha,\beta)}} -
 \frac{\tau_{k + 1, \ell}^{(\alpha + 1,\beta)} \tau_{k, \ell}^{(\alpha,\beta)}}
 {\tau_{k + 1, \ell}^{(\alpha,\beta)} \tau_{k, \ell}^{(\alpha + 1,\beta)}}=0,\\
 \tag{\texttt{[0,2]}}
 -\frac{ \tau_{k, \ell - 1}^{(\alpha,\beta)}} {\tau_{k + 1, \ell}^{(\alpha,\beta)}} +
 \frac{\tau_{k - 1, \ell - 1}^{(\alpha + 1,\beta)}}{\tau_{k, \ell}^{(\alpha + 1,\beta)}} +
 \frac{ \tau_{k, \ell - 1}^{ (\alpha + 1,\beta)} \tau_{k, \ell}^{ (\alpha,\beta)}}
 {\tau_{k + 1, \ell}^{ (\alpha,\beta)} \tau_{k, \ell}^{ (\alpha + 1,\beta)}}=0,\\
 \tag{\texttt{[2,0]}}
 \frac{ \tau_{k + 2, \ell + 1}^{ (\alpha - 1,\beta)}} {\tau_{k + 1, \ell}^{ (\alpha - 1,\beta)}} - \frac{\tau_{k + 1, \ell + 1}^{ (\alpha,\beta)}}
 {\tau_{k, \ell}^{ (\alpha,\beta)}} +\frac{\tau_{k + 1, \ell + 1}^{(\alpha - 1,\beta)}\tau_{k + 1, \ell}^{ (\alpha,\beta)}}{\tau_{k + 1, \ell}^{(\alpha - 1,\beta)} \tau_{k, \ell}^{(\alpha,\beta)}}=0.
 \end{gather}
 Following the same procedure, but instead using the two different factorizations of $U_{k,\ell_{+}}^{(\alpha,\beta)},$ we get the following relations satisfied by the $\tau$-functions:
\begin{gather}
 \tag{\texttt{[1,0]}}
 \frac{ \tau_{k + 1, \ell + 1}^{(\alpha,\beta + 1)} \tau_{k, \ell}^{(\alpha,\beta)}} {\tau_{k, \ell + 1}^{(\alpha,\beta)} \tau_{k, \ell}^{(\alpha,\beta + 1)}} -
\frac{\tau_{k + 1, \ell +1}^{(\alpha,\beta)}}{\tau_{k, \ell + 1}^{(\alpha,\beta)}} - \frac{\tau_{k +1, \ell}^{(\alpha,\beta + 1)}}{\tau_{k, \ell}^{(\alpha,\beta + 1)}}=0,\\
 \tag{\texttt{[0,1]}}
\frac{ \tau_{k - 1, \ell}^{(\alpha,\beta - 1)} \tau_{k, \ell + 1}^{(\alpha,\beta)}}{\tau_{k, \ell + 1}^{(\alpha,\beta - 1)} \tau_{k, \ell}^{(\alpha,\beta)}} -
\frac{\tau_{k - 1, \ell +1}^{(\alpha,\beta - 1)}}{\tau_{k, \ell + 1}^{(\alpha,\beta - 1)}} -\frac{\tau_{k - 1, \ell}^{(\alpha,\beta)}}{\tau_{k, \ell}^{(\alpha,\beta)}}=0,\\
\frac{\tau_{k + 1, \ell + 1}^{(\alpha,\beta)}\tau_{k - 1, \ell + 1}^{(\alpha,\beta -1)}}{\tau_{k, \ell + 1}^{(\alpha,\beta - 1)}
\tau_{k, \ell + 1}^{(\alpha,\beta)}} - \frac{\tau_{k + 1, \ell}^{(\alpha,\beta + 1)}\tau_{k - 1,\ell}^{(\alpha,\beta)}}{\tau_{k, \ell}^{(\alpha,\beta + 1)}
 \tau_{k, \ell}^{(\alpha,\beta)}} - \frac{\tau_{k, \ell + 1}^{(\alpha,\beta)} \tau_{k, \ell - 1}^{(\alpha,\beta + 1)}}{\tau_{k, \ell}^{(\alpha,\beta + 1)}\tau_{k, \ell}^{(\alpha,\beta)}}\notag \\
\qquad{} +\frac{\tau_{k, \ell + 1}^{(\alpha,\beta)}\tau_{k, \ell}^{(\alpha,\beta -1)}}{\tau_{k, \ell + 1}^{(\alpha,\beta - 1)}\tau_{k, \ell}^{(\alpha,\beta)}} +
\frac{\tau_{k, \ell + 2}^{(\alpha,\beta - 1)} \tau_{k, \ell}^{(\alpha,\beta)}}{\tau_{k, \ell + 1}^{(\alpha,\beta - 1)}
\tau_{k, \ell + 1}^{(\alpha,\beta)}} - \frac{\tau_{k, \ell + 1}^{(\alpha,\beta + 1)} \tau_{k, \ell}^{(\alpha,\beta)}}{\tau_{k, \ell + 1}^{(\alpha,\beta)} \tau_{k, \ell}^{(\alpha,\beta +1)}}=0. \tag{\texttt{[1,2]}}
\end{gather}
The two equations \texttt{[0,2]} and \texttt{[2,0]} can be brought under common denominator, and then the numerator is equal to $0$. (Here, we think of the coefficients $c_{i}$, $d_{i}$, $e_{i}$ as formal variables, so the denominators are not $0$.) The resulting equations are the same, up to a shift in the indices. This gives the equation \texttt{[2]} in the theorem. In the same way, \texttt{[1]} follows from \texttt{[0,1]}, \texttt{[1,0]}.

Next we bring \texttt{[0,0]} under common denominator. The vanishing of the numerator gives
\begin{gather}
 \tau_{k + 1, \ell + 1}^{(\alpha,\beta)} \tau_{k + 1, \ell}^{(\alpha - 1,\beta)} \tau_{k + 1, \ell}^{(\alpha,\beta)} \tau_{k - 1, \ell - 1}^{(\alpha + 1,\beta)} -
 \tau_{k + 2, \ell + 1}^{(\alpha - 1,\beta)} \tau_{k, \ell - 1}^{(\alpha,\beta)} \tau_{k, \ell}^{(\alpha + 1,\beta)} \tau_{k, \ell}^{(\alpha,\beta)} \nonumber\\
 \qquad{}+ \big( \tau_{k + 2, \ell}^{(\alpha - 1,\beta)} \tau_{k, \ell}^{(\alpha + 1,\beta)}-\tau_{k + 1, \ell}^{(\alpha + 1,\beta)} \tau_{k + 1, \ell}^{(\alpha - 1,\beta)}\big)
 (\tau_{k, \ell}^{(\alpha,\beta)})^{2} \nonumber\\
\qquad{}- \big(\tau_{k + 1, \ell}^{(\alpha - 1,\beta)} \tau_{k - 1, \ell}^{(\alpha + 1,\beta)} - \tau_{k, \ell}^{(\alpha + 1,\beta)} \tau_{k, \ell}^{(\alpha - 1,\beta)}\big)
 (\tau_{k + 1, \ell}^{(\alpha,\beta)})^{2}=0.\label{eq:78}
\end{gather}
In the first term we substitute
\begin{gather*}
\tau_{k+1, \ell+1}^{( \alpha ,\beta)}\tau_{k+1, \ell }^{( \alpha-1,\beta)}= \tau_{k+1, \ell}^{ (\alpha,\beta)} \tau_{k+1, \ell+1}^{ (\alpha-1,\beta)}+\tau_{k + 2, \ell+1}^{(\alpha-1,\beta)}\tau_{k, \ell}^{(\alpha,\beta)},
\end{gather*}
which follows from $\texttt{[2]}$ by the change of variables, $k\mapsto k+1$, $\ell\mapsto \ell+1$, $\alpha\mapsto \alpha-2$. The first term then becomes
\begin{gather*}
\tau_{k + 1, \ell}^{(\alpha,\beta)} \tau_{k - 1, \ell - 1}^{(\alpha + 1,\beta)}\big(\tau_{k+1, \ell}^{ (\alpha,\beta)} \tau_{k+1, \ell+1}^{ (\alpha-1,\beta)}+\tau_{k + 2, \ell+1}^{(\alpha-1,\beta)}\tau_{k, \ell}^{(\alpha,\beta)}\big).
\end{gather*}
In the second term of~\eqref{eq:78}, we use \texttt{[2]} in the form
\begin{gather*}
\tau_{k, \ell - 1}^{( \alpha,\beta)}\tau_{k, \ell}^{( \alpha + 1,\beta)}= \tau_{k, \ell - 1}^{ (\alpha + 1,\beta)} \tau_{k, \ell}^{ (\alpha,\beta)}+
\tau_{k + 1, \ell}^{(\alpha,\beta)}\tau_{k - 1, \ell - 1}^{(\alpha + 1,\beta)},
\end{gather*}
so that the second term becomes
\begin{gather*}
-\tau_{k + 2, \ell + 1}^{(\alpha - 1,\beta)}
 \tau_{k, \ell}^{(\alpha,\beta)}\big(\tau_{k, \ell - 1}^{ (\alpha + 1,\beta)} \tau_{k, \ell}^{ (\alpha,\beta)}+\tau_{k + 1, \ell}^{(\alpha,\beta)}\tau_{k - 1, \ell - 1}^{(\alpha + 1,\beta)}\big).
\end{gather*}
After the cancellation of two terms and collecting like terms, we have~\eqref{eq:78}
\begin{gather}
 \big(\tau_{k,\ell}^{(\alpha,\beta)}\big)^{2}\big( \tau_{k + 2, \ell}^{(\alpha - 1,\beta)} \tau_{k, \ell}^{(\alpha + 1,\beta)}
-\tau_{k + 1, \ell}^{(\alpha + 1,\beta)} \tau_{k + 1, \ell}^{(\alpha - 1,\beta)}-\tau_{k + 2, \ell + 1}^{(\alpha - 1,\beta)} \tau_{k, \ell - 1}^{ (\alpha + 1,\beta)}\big)\nonumber\\
\qquad{}= \big(\tau_{k + 1, \ell}^{(\alpha,\beta)}\big)^{2}\big(\tau_{k+1,\ell}^{(\alpha-1,\beta)}\tau_{k - 1, \ell}^{(\alpha + 1,\beta)} -
 \tau_{k, \ell}^{(\alpha + 1,\beta)} \tau_{k, \ell}^{(\alpha - 1,\beta)}- \tau_{k+1, \ell+1}^{ (\alpha-1,\beta)} \tau_{k - 1, \ell-1}^{(\alpha + 1,\beta)} \big).\label{eq:79}
\end{gather}
Now observe that the square factor on the right in~\eqref{eq:79} is obtained from the square factor on the left by a shift $k\mapsto k+1$, and similarly, mutatis mutandis, for the terms in the big parentheses. Therefore, if we know that
\begin{gather}
\big(\tau_{k,\ell}^{(\alpha,\beta)}\big)^{2}= \tau_{k, \ell}^{(\alpha + 1,\beta)} \tau_{k, \ell}^{(\alpha - 1,\beta)}
+\tau_{k+1, \ell+1}^{ (\alpha-1,\beta)} \tau_{k - 1, \ell-1}^{(\alpha + 1,\beta)}-\tau_{k+1,\ell}^{(\alpha-1,\beta)}\tau_{k - 1, \ell}^{(\alpha + 1,\beta)} ,\label{eq:80}
\end{gather}
then \eqref{eq:79} also tells us that
\begin{gather*}
\big(\tau_{k + 1, \ell}^{(\alpha,\beta)}\big)^{2}=\tau_{k + 1, \ell}^{(\alpha + 1,\beta)} \tau_{k + 1, \ell}^{(\alpha - 1,\beta)}+\tau_{k + 2, \ell + 1}^{(\alpha - 1,\beta)} \tau_{k, \ell - 1}^{ (\alpha + 1,\beta)}-
\tau_{k + 2, \ell}^{(\alpha - 1,\beta)} \tau_{k, \ell}^{(\alpha + 1,\beta)}.
\end{gather*}
So it suffices to check \eqref{eq:80} for $k=0$, which reads
\begin{gather*}
\big(\tau_{0,\ell}^{(\alpha,\beta)}\big)^{2}= \tau_{0, \ell}^{(\alpha + 1,\beta)} \tau_{0, \ell}^{(\alpha - 1,\beta)}.
\end{gather*}
But this last equation is true by Item~\ref{item:7} of the examples in Section~\ref{sec:examples-3times-3}: the $\tau$-function $\tau_{0,\ell}^{(\alpha,\beta)}$ is independent of $\alpha$, see also~\eqref{eq:55}. This proves part~\texttt{[0]} of the theorem.

Finally, we use the \texttt{[1,2]} component of the two expressions of $U^{(\alpha,\beta)}_{k,\ell_{+}}$. Bringing all terms under a common denominator, the vanishing of the numerator gives
\begin{gather}
 \tau_{k + 1, \ell}^{(\alpha, \beta + 1)}\tau_{k - 1, \ell}^{(\alpha,\beta)}\tau_{k, \ell + 1}^{(\alpha, \beta - 1)}\tau_{k, \ell + 1}^{(\alpha,\beta)} -
\tau_{k + 1, \ell + 1}^{(\alpha, \beta)}\tau_{k - 1, \ell + 1}^{(\alpha, \beta - 1)}\tau_{k, \ell}^{(\alpha, \beta + 1)}\tau_{k, \ell}^{(\alpha,\beta)} \nonumber \\
\qquad{}+ \big(\tau_{k, \ell + 1}^{ (\alpha, \beta + 1)}\tau_{k, \ell + 1}^{ (\alpha, \beta- 1)}-\tau_{k, \ell + 2}^{(\alpha, \beta -1)}\tau_{k, \ell}^{ (\alpha, \beta + 1)}\big)
 \big(\tau_{k, \ell}^{ (\alpha, \beta)}\big)^{2} \nonumber\\
\qquad{}+ \big(\tau_{k, \ell + 1}^{(\alpha, \beta - 1)}\tau_{k, \ell - 1}^{ (\alpha, \beta +1)} -\tau_{k, \ell}^{ (\alpha, \beta + 1)}\tau_{k, \ell}^{ (\alpha, \beta -1)}\big)
\big(\tau_{k, \ell + 1}^{(\alpha, \beta)}\big)^{2}=0.\label{eq:81}
\end{gather}
In the first term, substitute
\begin{gather*}
\tau_{k , \ell +1}^{(\alpha,\beta-1)}\tau_{k-1, \ell}^{(\alpha,\beta)} = \tau_{k , \ell + 1}^{(\alpha,\beta )} \tau_{k-1, \ell}^{(\alpha,\beta-1)}-
\tau_{k,\ell}^{(\alpha,\beta)}\tau_{k-1, \ell + 1}^{(\alpha,\beta-1)},
\end{gather*}
which is obtained from \texttt{[1]} by the shift $\beta\mapsto \beta-1$. The first term becomes
\begin{gather*}
 \tau_{k + 1, \ell}^{(\alpha, \beta + 1)}\tau_{k, \ell + 1}^{(\alpha,\beta)}\big(\tau_{k , \ell + 1}^{(\alpha,\beta )} \tau_{k-1, \ell}^{(\alpha,\beta-1)}-
\tau_{k,\ell}^{(\alpha,\beta)}\tau_{k-1, \ell + 1}^{(\alpha,\beta-1)}\big).
\end{gather*}
In the second term we also use \texttt{[1]}, in the form
\begin{gather*}
\tau_{k + 1, \ell +1}^{(\alpha,\beta)}\tau_{k, \ell}^{(\alpha,\beta + 1)}=\tau_{k + 1, \ell + 1}^{(\alpha,\beta + 1)} \tau_{k, \ell}^{(\alpha,\beta)}-\tau_{k +1, \ell}^{(\alpha,\beta + 1)}\tau_{k, \ell + 1}^{(\alpha,\beta)},
\end{gather*}
transforming the second term to
\begin{gather*}
-\tau_{k - 1, \ell + 1}^{(\alpha, \beta - 1)}\tau_{k, \ell}^{(\alpha,\beta)}\big(\tau_{k + 1, \ell + 1}^{(\alpha,\beta + 1)}\tau_{k, \ell}^{(\alpha,\beta)}-
\tau_{k +1, \ell}^{(\alpha,\beta + 1)}\tau_{k, \ell + 1}^{(\alpha,\beta)}\big).
\end{gather*}
After cancellation of two terms and collecting like terms, \eqref{eq:81} becomes
\begin{gather}
\big(\tau_{k, \ell + 1}^{ (\alpha, \beta + 1)}\tau_{k, \ell + 1}^{ (\alpha, \beta- 1)}-\tau_{k, \ell + 2}^{(\alpha, \beta -1)}\tau_{k, \ell}^{ (\alpha, \beta + 1)}-\tau_{k - 1, \ell + 1}^{(\alpha, \beta - 1)}
\tau_{k + 1, \ell + 1}^{(\alpha,\beta + 1)}\big) \big(\tau_{k, \ell}^{ (\alpha, \beta)}\big)^{2} \nonumber\\
\qquad{}= \big(\tau_{k, \ell}^{ (\alpha, \beta + 1)}\tau_{k, \ell}^{ (\alpha, \beta -1)}-\tau_{k, \ell + 1}^{(\alpha, \beta - 1)}\tau_{k, \ell - 1}^{ (\alpha, \beta +1)}-
\tau_{k-1, \ell}^{(\alpha,\beta-1)}\tau_{k + 1, \ell}^{(\alpha, \beta + 1)}\big)\big(\tau_{k, \ell + 1}^{(\alpha, \beta)}\big)^{2}. \label{eq:82}
\end{gather}
As before, \eqref{eq:82} implies that if
\begin{gather}
 \big(\tau_{k, \ell}^{ (\alpha, \beta)}\big)^{2}=\tau_{k, \ell}^{ (\alpha, \beta + 1)}
\tau_{k, \ell}^{ (\alpha, \beta -1)}- \tau_{k, \ell + 1}^{(\alpha, \beta - 1)}\tau_{k, \ell - 1}^{ (\alpha, \beta +1)}-\tau_{k-1, \ell}^{(\alpha,\beta-1)}\tau_{k + 1, \ell}^{(\alpha, \beta + 1)} ,\label{eq:83}
\end{gather}
the same equation holds after a shift $\ell\mapsto \ell+1$. So we reduce to the case of $\ell=0$ of \eqref{eq:83}:
\begin{gather*}
\big(\tau_{k, 0}^{ (\alpha, \beta)}\big)^{2}=\tau_{k, 0}^{ (\alpha, \beta + 1)}\tau_{k, 0}^{ (\alpha, \beta -1)}-\tau_{k-1, 0}^{(\alpha,\beta-1)}\tau_{k + 1, 0}^{(\alpha, \beta + 1)},
\end{gather*}
using Item \ref{item:8} of the examples in Section~\ref{sec:examples-3times-3}. By Item~\ref{item:9} of these examples $\tau_{k,0}^{(\alpha,\beta)}=\tau_{k}^{(\alpha-\beta)}$. So the case $\ell=0$ is just the $Q$-system of Theorem~\ref{Thm:Qsystem}. This proves the final part \texttt{[1,1]} of the theorem.
\end{proof}

\appendix

\section{Multi-component fermions and semi-infinite wedge space}\label{sec:ferm-semi-infin}

\subsection{Introduction}\label{sec:introduction-1}

In the main text we work with $n\times n$ matrices (depending on a spectral parameter $z$) for $n=2$ or $3$. In this appendix we will not specify $n$, as the theory of $n$-component fermions and the associated semi-infinite wedge space does not significantly depend on~$n$. A convenient reference for background and more details is ten Kroode and van der Leur~\cite{tKvdL:BosFer}.

\subsection{Semi-infinite Wedge space}\label{sec:semi-infinite-wedge}

Let
\begin{gather*}
\left\{e_{0}= \begin{bmatrix}
 1\\0\\\vdots\\0
\end{bmatrix}
,e_{1}=
\begin{bmatrix}
 0\\1\\\vdots\\0
\end{bmatrix}
,\dots,e_{n-1}=
\begin{bmatrix}
 0\\\vdots\\0\\1
\end{bmatrix}\right\}
\end{gather*} denote the standard basis of $\mathbb{C}^{n}$. Denote the corresponding elementary matrices by~$E_{ab}$ (such that $E_{ab}e_{c}=\delta_{bc}e_{a}$); they are also indexed by integers $0,1,\dots, n-1$. We also need the loop space of~$\mathbb{C}^{n}$, denoted by
\begin{gather*}
H^{(n)}=\mathbb{C}^{n}\otimes \mathbb{C}\big[z,z\inv\big],
\end{gather*}
with basis $e_{a}^{k}=e_{a}z^{k}$, for $a=0,\dots,n-1$ and $k\in\mathbb{Z}$. Let $F^{(n)}$ be the $n$-component fermionic Fock space, the semi-infinite wedge space based on~$H^{(n)}$. It is spanned by semi-infinite wedges
\begin{gather*}
\omega=w_{0}\wedge w_{1}\wedge w_{2}\wedge\cdots,\qquad w_{i}\in H^{(n)},
\end{gather*}
where the $w_{i}$ satisfy some restrictions that we will presently discuss. Semi-infinite wedges obey the usual rules of exterior algebra, like multilinearity in each factor and antisymmetry under exchange of two factors.

To formulate the restrictions on the $w_{i}$ that can appear in the wedge $\omega$ above we introduce the Clifford algebra ${\rm Cl}^{(n)}$ acting on~$F^{(n)}$: it is generated by exterior and interior products, denoted by
$e\big(e_{a}^{k}\big)$ and $i\big(e_{a}^{k}\big)$, defined as wedging and contracting operators, respectively:
\begin{gather*}
e\big(e_{a}^{k}\big)\alpha=e_{a}^{k}\wedge\alpha,\qquad i\big(e_{a}^{k}\big)\alpha=\beta, \qquad \text{if}\quad \alpha=e_{a}^{k}\wedge\beta.
\end{gather*}
It is useful to collect the generators of the Clifford algebra in generating series. Therefore, define \emph{fermion fields}
\begin{gather*}
\psi_{a}^{\pm}(w)=\sum_{k\in\mathbb{Z}} {}_{a}\psi_{(k)}^{\pm}w^{-k-1},\qquad a=0,1,\dots, n-1,
\end{gather*}
where
\begin{gather}
\fermion{a}{k}{+}=e\big(e_{a}^{k}\big)=e_{a}z^{k}\wedge,\qquad \fermion{a}{k}{-}=i\big(e_{a}z^{-k-1}\big).\label{eq:27}
\end{gather}
The fermionic fields satisfy anti-commutation relations\footnote{$[a,b]_{+}=ab+ba$.}
\begin{gather}
\big[\psi_{a}^{\pm}(z),\psi^{\pm}_{b}(w)\big]_{+}=0,\qquad \big[\psi_{a}^{+}(w_{1}),\psi_{b}^{-}(w_{2})\big]_{+}=\delta_{ab}\delta(w_{1},w_{2}),\label{eq:20}
\end{gather}
where the formal delta distribution is defined by
\begin{gather*}
\delta(z,w)=\sum_{k\in\mathbb{Z}}z^{k}w^{-k-1}.
\end{gather*}
Let $v_{0}$ be the vacuum vector
\begin{gather}
v_{0}=
\begin{bmatrix}
 1\\0\\0\\\vdots
\end{bmatrix}\wedge
\begin{bmatrix}
 0\\1\\0\\\vdots
\end{bmatrix}\wedge \dots\wedge
\begin{bmatrix}
 0\\0\\\vdots\\1
\end{bmatrix}
\wedge
\begin{bmatrix}
 z\\0\\0\\\vdots
\end{bmatrix}\wedge
\begin{bmatrix}
 0\\z\\0\\\vdots
\end{bmatrix}\wedge\dots\wedge
\begin{bmatrix}
 0\\0\\\vdots\\z
\end{bmatrix}
\wedge
\begin{bmatrix}
 z^{2}\\0\\0\\\vdots
\end{bmatrix}\wedge
\begin{bmatrix}
 0\\z^{2}\\0\\\vdots
\end{bmatrix}
\wedge\cdots.\label{eq:4}
\end{gather}
Then we define $F^{(n)}$ to be the span of the wedges obtained by acting on the vacuum $v_{0}$ by monomials in the wedging and contracting operators. To get a~basis for $F^{(n)}$ we specify an ordering on the wedging/contracting operators acting on $F^{(n)}$.

\begin{defn}\label{Def:ElemntaryWedge} An \emph{elementary wedge} in $F^{(n)}$ is an element $\omega=Mv_{0}$, where
\begin{gather*}
M=M_{n-1}\cdots M_{1}M_{0},\qquad M_{a}=M_{a}^{+}M_{a}^{-}, \qquad a=n-1,n-2,\dots,2,1,
\end{gather*}
where
\begin{gather*}
M_{a}^{\pm}=\fermion{a}{k_{1}}{\pm}\fermion{a}{k_{2}}{\pm}\cdots \fermion{a}{k_{s}}{\pm},\qquad k_{1}<k_{2}<\dots<k_{s}\le-1,
\end{gather*}
is a monomial in $\fermion{a}{k}{\pm}$ for $k\le -1$, ordered in increasing order from left to right.
\end{defn}

The statement that the elementary wedges form a basis for $F^{(n)}$ follows from the Poincar\'{e}--Birkhoff--Witt theorem for the Lie superalgebra underlying the Clifford algebra.

We define a bilinear form, denoted $\langle\,,\,\rangle$, on $F^{(n)}$ by declaring the elementary wedges to be orthonormal. We then have
\begin{gather} \label{eq:14}
\big\langle \fermion{a}{k}{+}v,w\big\rangle=\big\langle v, \fermion{a}{-k-1}{-}w\big\rangle,
\end{gather}
and
\begin{gather*}
 \big\langle\psi_{a}^{+}(z)v,w\big\rangle=\big\langle v,\psi_{a}^{-}\big(z\inv\big)z\inv w\big\rangle. 
\end{gather*}

The $n$-component fermionic Fock space $F^{(n)}$ has a grading by the Abelian group $\mathbb{Z}^{n}$, i.e., we have a decomposition $F^{(n)}=\oplus_{\alpha\in\mathbb{Z}^{n}}F^{(n)}_{\alpha}$. The vacuum has degree $\alpha=0$. To describe the grading, introduce a~basis in~$\mathbb{Z}^{n}$ by
\begin{gather*}
\delta_{a}=
\begin{bmatrix}
 0&0&\dots&1&\dots&0
\end{bmatrix},\qquad \text{where the $1$ is in position $a=0,1,\dots, n-1$.}
\end{gather*}
The grading on $F^{(n)}$ induces a grading on linear maps on $F^{(n)}$: if $L\colon F^{(n)}\to F^{(n)}$ has the property that there exists a $\delta\in \mathbb{Z}^{n}$ so that, for all $\omega\in \mathbb{Z}^{n}$, $L$ restricts to a map $ F^{(n)}_{\omega}\to F^{(n)}_{\omega+\delta}$, then we say that $L$ has degree $\delta$. Then the grading is uniquely determined by declaring wedging operators $e_{a}z^{k}\wedge$ to have degree $\delta_{a}$, and the contracting operators $i\big(e_{a}z^{k}\big)$ to have degree $-\delta_{a}$. The fields $\psi^{\pm}_{a}(z)$ have degree~$\pm\delta_{a}$.

The \emph{total degree} of an element $\omega$ of degree $\alpha$ is just the sum of the entries in the degree row vector~$\alpha$.

\subsection{Fermionic translation operators and translation group}\label{sec:ferm-transl-oper}

Besides the action of fermion operators, $\fermion{a}{k}{\pm}$, on $F^{(n)}$, we also have the action of fermionic translation operators $Q_{a}\colon F^{(n)}\to F^{(n)}$, $a=0,1,\dots, n-1$, given by
\begin{gather}
Q_{a}v_{0}=\psi_{a}^{+}(z)v_{0}\big|_{z=0}=\fermion{a}{-1}{+}v_{0},\label{eq:16}
\end{gather}
and
\begin{gather}
\psi_{a}^{\pm}(z)Q_{a}=z^{\pm}Q_{a}\psi_{a}^{\pm}(z), \label{eq:23}\\
\psi_{a}^{\pm}(z)Q_{b}=-Q_{b}\psi_{a}^{\pm}(z),\qquad a\ne b, \label{eq:24}\\
Q_{a}Q_{b}=-Q_{b}Q_{a},\qquad a\ne b.\nonumber
\end{gather}
The $Q_{a}$ are invertible. $Q_{a}^{\pm 1}$ has degree $\pm\delta_{a}$.

The $Q_{a}$ are unitary for the standard bilinear form of $F^{(n)}$:
\begin{gather}
\langle Q_{a}v,w\rangle=\big\langle v,Q_{a}\inv w\big\rangle, \qquad a=0,1,\dots, n-1, \quad \text{for all} \ w\in F^{(n)}.\label{eq:n1}
\end{gather}
The fermionic translation operators belong to the central extension of the loop group $\widetilde{\rm GL}_n$ (acting on~$F^{(n)}$). They are lifts of commuting elements of the non-centrally extended loop group, see for instance
\cite[Proposition~5.3.4]{MR97c:58061}. We have
\begin{gather}
\pi(Q_{a})=-\sum_{b\ne a}E_{bb}+z\inv E_{aa}.\label{eq:3}
\end{gather}

The group generated by $Q_{a}$, $a=0,1,\dots, n-1$, contains a subgroup of elements of total degree zero, generated by the translation operators $T_{s}=Q_{s}Q_{s-1}\inv$, $s=1,2,\dots,{n-1}$, of degree $\delta_{s}-\delta_{s-1}$. Another set of generators for this subgroup is also useful: define
\begin{gather*}
T_{ab}=Q_{a}Q_{b}^{-1},
\end{gather*}
of degree $\delta_{a}-\delta_{b}$.
\begin{lem}\label{Lem:GeneratorsTranslationGroupIdentities}\quad
 \begin{enumerate}\itemsep=0pt
 \item $T_{ab}Q_{c}=Q_{c}T_{ab}$ if $c\ne a,b$.\label{item:1}
 \item \label{item:6} For $a\ne b$ and for all $m\in \mathbb{Z}$
\begin{gather*}
\big(Q_{a}Q_{b}\inv\big)^{m}=(-1)^{\frac{m(m-1)}{2}}Q_{a}^{m}Q_{b}^{-m}.
\end{gather*}
 \item \label{item:2} For all $k,\ell\in \mathbb{Z}$ we have
 \begin{gather*}
T_{2}^{k}T_{1}^{\ell}=(-1)^{\frac{k(k-1)}{2}+\frac{\ell(\ell-1)}{2}}Q_{2}^{k}Q_{1}^{\ell-k}Q_{0}^{-\ell}.
\end{gather*}
\item \label{item:3} For all $\alpha,\beta,\gamma\in \mathbb{Z}$
\begin{gather*}
T_{10}^{\alpha}T_{20}^{\beta}T_{21}^{\gamma}=
(-1)^{\frac{\alpha(\alpha-1)}{2}+\frac{\beta(\beta-1)}{2}+\frac{\gamma(\gamma-1)}{2}+\alpha\gamma}
Q_{2}^{\alpha+\beta}Q_{1}^{\alpha-\gamma}Q_{0}^{-\alpha-\beta}.
\end{gather*}
\item \label{item:4}
For all $\alpha,\beta,\gamma\in \mathbb{Z}$
\begin{gather*}
T_{10}^{\alpha}T_{20}^{\beta}T_{21}^{\gamma}= (-1)^{\frac{\beta(\beta-1)}{2}+ \alpha\beta+\alpha\gamma+\beta\gamma}T_{2}^{\beta+\gamma}T_{1}^{\alpha+\beta}.
\end{gather*}
 \end{enumerate}
\end{lem}
\begin{proof} Part (\ref{item:1}) is clear. Part (\ref{item:6}) is a simple
 induction.
For part (\ref{item:2}) we get
 \begin{gather*}
 T_{2}^{k}T_{1}^{\ell}=\big(Q_{2}Q_{1}\inv\big)^{k}\big(Q_{1}Q_{0}\inv\big)^{\ell} =(-1)^{\frac{k(k-1)}{2}}(-1)^{\frac{\ell(\ell-1)}{2}}Q_{2}^{k}Q_{1}^{-k}Q_{1}^{\ell}Q_{0}^{-\ell}.
 \end{gather*}
Similarly for part (\ref{item:3}) we have
\begin{gather*}
 T_{10}^{\alpha}T_{20}^{\beta}T_{21}^{\gamma}=(-1)^{\frac{\alpha(\alpha-1)}{2}+\frac{\beta(\beta-1)}{2}+\frac{\gamma(\gamma-1)}{2}}Q_{1}^{\alpha}Q_{0}^{-\alpha}Q_{2}^{\beta}Q_{0}^{-\beta}Q^{\gamma}_{2}Q_{1}^{-\gamma}\\
\hphantom{T_{10}^{\alpha}T_{20}^{\beta}T_{21}^{\gamma}}{}
=(-1)^{\frac{\alpha(\alpha-1)}{2}+\frac{\beta(\beta-1)}{2}+\frac{\gamma(\gamma-1)}{2}}Q_{2}^{\beta}Q_{1}^{\alpha}Q_{2}^{\gamma}Q_{1}^{-\gamma}Q_{0}^{-\alpha}Q_{0}^{-\beta}\\
\hphantom{T_{10}^{\alpha}T_{20}^{\beta}T_{21}^{\gamma}}{}
=(-1)^{\frac{\alpha(\alpha-1)}{2}+\frac{\beta(\beta-1)}{2}+\frac{\gamma(\gamma-1)}{2}+\alpha\gamma}Q_{2}^{\beta+\gamma}Q_{1}^{\alpha-\gamma}Q_{0}^{-\alpha-\beta}.
\end{gather*}
Finally, for part (\ref{item:4}) we substitute
\begin{gather*}
Q_{2}^{\beta+\gamma}Q_{1}^{\alpha-\gamma}Q_{0}^{-\alpha-\beta}=(-1)^{\frac{\alpha(\alpha-1)}{2}+\frac{\gamma(\gamma-1)}{2}+\alpha\beta+\beta\gamma}T_{2}^{\beta+\gamma}T_{1}^{\alpha+\beta},
\end{gather*}
(from part (\ref{item:2})), in the right hand side of part
(\ref{item:3}). The result then follows from
\begin{gather*}
(-1)^{\frac{\alpha(\alpha-1)}{2}+\frac{\beta(\beta-1)}{2}+\frac{\gamma(\gamma-1)}{2}+\alpha\gamma+\frac{\alpha(\alpha-1)}{2}+\frac{\gamma(\gamma-1)}{2}+\alpha\beta+\beta\gamma}=
(-1)^{\frac{\beta(\beta-1)}{2}+\alpha\beta+\alpha\gamma+\beta\gamma}.\tag*{\qed}
\end{gather*}\renewcommand{\qed}{}
\end{proof}

Define the ordered product of $k\ge 0$ fermions by
\begin{gather*}
\mathop{\vec{\prod}}\limits_{l=1}^{k}{}_{a}\psi_{(-l)}^{\pm}={}_{a}\psi^{\pm}_{(-k)}\cdots{}_{a}\psi_{(-1)}^{\pm}.
\end{gather*}
The empty product is as usual the identity.

\begin{lem}\label{Lem:FermionQvacuum}\quad
 \begin{enumerate}
 \item \label{item:5} For $k\in\mathbb{Z}$,
\begin{gather*}
Q_{a}^{k}v_{0}=
\begin{cases}
v_{0},&k=0,\\
\displaystyle \mathop{\vec{\prod}}\limits_{l=1}^{k}\big(\fermion{a}{-l}{+}\big)v_{0},&k>0,\\
\displaystyle \mathop{\vec{\prod}}\limits_{l=1}^{-k}\big(\fermion{a}{-l}{-}\big)v_{0},&k<0.
\end{cases}
\end{gather*}
 \item For all $\alpha,\beta\ge0$
 \begin{gather*} Q_{1}^{\beta}Q_{0}^{\alpha}v_{0}=(-1)^{\beta\alpha}\mathop{\vec{\prod}}\limits_{l=1}^{\alpha}\fermion{0}{-l}{+}\mathop{\vec{\prod}}\limits_{m=1}^{\beta}\fermion{1}{-m}{+}v_{0}
=\mathop{\vec{\prod}}\limits_{m=1}^{\beta}\fermion{1}{-m}{+}\mathop{\vec{\prod}}\limits_{l=1}^{\gamma}\fermion{0}{-l}{+} v_{0}\\
\hphantom{Q_{1}^{\beta}Q_{0}^{\alpha}v_{0}}{}=e_{1}z^{-\beta}\wedge e_{1}z^{1-\beta}\wedge\cdots\wedge e_{1}z\inv\wedge e_{0}z^{-\alpha}\wedge e_{0}z^{1-\alpha}\wedge \cdots\wedge e_{0}z\inv \wedge v_{0}.
 \end{gather*}
 \end{enumerate}
\end{lem}

Just as the fermionic translation operators are unitary, so are the translation operators $T_{s}$ and $T_{ab}$: from~\eqref{eq:n1} it follows that
\begin{gather} \label{eq:22}
 \langle Tv, w\rangle=\big\langle v,T\inv w\big\rangle \qquad \text{for all} \quad w\in F^{(n)}.
\end{gather}

\subsection[The Lie algebra $\widetilde{\mathfrak{gl}}_n$ and fermions]{The Lie algebra $\boldsymbol{\widetilde{\mathfrak{gl}}_n}$ and fermions}\label{sec:lie-algebra-sln}

The loop algebra $\widetilde{\mathfrak{gl}}_n$ is defined as the Lie subalgebra of $\mathfrak{gl}(H^{(n)})$ generated by $E_{ab}z^{k}$, $a,b=0,1,\dots, n-1,$ $ k\in\mathbb{Z}$, where
\begin{gather*}
\big(E_{ab}z^{k}\big)\cdot \big(e_{c}z^{m}\big)=\delta_{bc}e_{a}z^{k+m}.
\end{gather*}
The loop algebra $\widetilde{\mathfrak{gl}}_n$ does not quite act on $F^{(n)}$. One would like to define the action on $F^{(n)}$ by
\begin{gather*}
E_{ab}z^{k}\mapsto\sum_{l\in\mathbb{Z}}\big(e_{a}z^{k+l}\wedge\big)\big(i\big(e_{b}z^{l}\big)\big)=\sum_{l\in\mathbb{Z}}\fermion{a}{k+l}{+}\fermion{b}{-l-1}{-}.
\end{gather*}
However, considering the action of $E_{aa}z^{0}$ on the vacuum $v_{0}$ we would get
\begin{gather*}
E_{aa}v_{0}=\sum_{l\ge0}v_{0},
\end{gather*}
(since $\big(e_{a}z^{l}\wedge\big)\big(i(e_{a}z^{l})\big)v_{0}=v_{0}$ for $l\ge0$), so that these diagonal elements would have a divergent action. Therefore we introduce a normal ordering on fermion fields~\cite{MR1651389} by
\begin{gather*}
\colon \psi_{a}^{+}(z)\psi_{b}^{-}(w)\colon=\psi_{a,\text{cr}}^{+}(z)\psi_{b}^{-}(w)-\psi_{b}^{-}(w)\psi_{a,\text{ann}}^{+}(z),
\end{gather*}
where the creation and annihilation parts of a fermion field are given by
\begin{gather*}
\psi_{\text{cr}}(z)=\sum_{k\ge0}\psi_{(-k-1)}z^{k},\qquad \psi_{\text{ann}}(z)=\sum_{k\ge0}\psi_{(k)}z^{-k-1}.
\end{gather*}
Of course, when $a\ne b$
\begin{gather*}
\colon \psi_{a}^{+}(z)\psi_{b}^{-}(w)\colon= \psi_{a}^{+}(z)\psi_{b}^{-}(w).
\end{gather*}
We define the normal ordering on the components of the fermion fields by
\begin{gather*}
\colon\psi_{a}^{+}(z)\psi_{b}^{-}(w)\colon=\sum_{k,l\in\mathbb{Z}}\colon \fermion{a}{k}{+}\fermion{b}{l}{-}\colon z^{-k-1}w^{-l-1}.
\end{gather*}
Then we have
\begin{gather*}
\colon \fermion{a}{k}{+}\fermion{b}{l}{-}\colon= \begin{cases}
 \fermion{a}{k}{+}\fermion{b}{l}{-},&k<0, \\
 -\fermion{b}{l}{-}\fermion{a}{k}{+},&k\ge 0.
\end{cases}
\end{gather*}
Note that
\begin{gather} \label{eq:77}
 \colon \fermion{a}{k}{+}\fermion{b}{l}{-}\colon= \fermion{a}{k}{+}\fermion{b}{l}{-},
\end{gather}
unless $a=b$ and $k=-l-1$.

Next we can define an action on $F^{(n)}$ by
\begin{gather*}
E_{ab}z^{k}\mapsto \colon\sum_{l\in\mathbb{Z}}\big(e_{a}z^{k+l}\wedge\big)\big(i\big(e_{b}z^{l}\big)\big)\colon=\sum_{l\in\mathbb{Z}} \colon{}_{a}\psi^{+}_{(k+l)}{}_{b}\psi^{-}_{(-l-1)}\colon.
\end{gather*}
This gives a central extension (cf.~\cite{tKvdL:BosFer})
\begin{gather*}
 0\to \mathbb{C}\to \widehat {\mathfrak{gl}_{n}} \to \widetilde{\mathfrak{gl}}_n\to0.
\end{gather*}
Introduce a generating series of loop algebra elements by
\begin{gather}
E_{ab}(z_{1})=\sum_{l\in\mathbb{Z}}E_{ab}z^{l}z_{1}^{-l-1}.
\label{eq:2}
\end{gather}
The action of this generating series on $F^{(n)}$ can be represented by a normal ordered product of fermion fields
\begin{gather}
E_{ab}(z_{1})=\colon\psi_{a}^{+}(z_{1})\psi_{b}^{-}(z_{1})\colon. \label{eq:26}
\end{gather}
Indeed,
\begin{gather*}
 E_{ab}(z_{1})=\sum_{k,l\in\mathbb{Z}}\colon\fermion{a}{k+l}{+}\fermion{b}{-l-1}{-}\colon z_{1}^{-k-1}\\
\hphantom{E_{ab}(z_{1})}{} =\sum_{k,l\in\mathbb{Z}}\colon(\fermion{a}{k+l}{+}z_{1}^{-k-l-1})(\fermion{b}{-l-1}{-}z_{1}^{l})\colon= \colon\psi_{a}^{+}(z_{1})\psi_{b}^{-}(z_{1})\colon.
\end{gather*}
The series $E_{ab}(z_1)$, acting on $F^{(n)}$, has degree $\delta_a-\delta_b$.

Equation \eqref{eq:26} is the reason we chose to encode the wedging and contracting operators as coefficients of fermion fields according to~\eqref{eq:27}.

We also need the commutator of the generating series of Lie algebra elements with fermionic translation operators.

\begin{lem} \label{Lem-Comm-E10-shifts} For all $\alpha,\beta\in\mathbb{Z}$ we have
\begin{gather*}
Q_{0}^{\beta}Q_{1}^{-\alpha}E_{10}(z)Q_{1}^{\alpha}Q_{0}^{-\beta}=(-1)^{\alpha+\beta}z^{\alpha+\beta}E_{10}(z).
\end{gather*}
\end{lem}

\subsection{Root lattice}\label{sec:root-lattice}

Recall, see Appendix \ref{sec:introduction-1}, the group $\mathbb{Z}^{n}$ that gives a grading for fermionic Fock space $F^{(n)}$. The root lattice $A_{n-1}$ is a subgroup of $\mathbb{Z}^{n}$. It is
generated by
\begin{gather*}
\alpha_{i}=\delta_{i}-\delta_{i-1},\qquad i=1,2,\dots,n-1,
\end{gather*}
so
\begin{gather*}
A_{n-1}=\mathop{\oplus}\limits_{i=1}^{n-1}\mathbb{Z}\alpha_{i}\subset\mathbb{Z}^{n}.
\end{gather*}
We will call elements in $A_{n-1}$ of the form $\alpha=\sum\limits_{i=1}^{n-1}n_{i}\alpha_i$ positive roots if $n_{i}\ge0$.

The translation group is also graded by $A_{n-1}$: the generator $T_{s}=Q_{s}Q_{s-1}\inv$ has degree $\alpha_{s}$. Similarly the Lie algebra generating fields $E_{ab}(z)$ have
\begin{gather*}
\deg(E_{aa-1}(z))=\alpha_a.
\end{gather*}

\section[Expressions for the $\tau$-functions]{Expressions for the $\boldsymbol{\tau}$-functions}\label{sec:expr-tau-funct}

In this appendix we prove Theorems \ref{Thm:tauHankel} and~\ref{thm:n=3taufunctions}. This gives expressions for the $\tau$-functions in terms of coordinates on the lower triangular subgroup~$\mathcal{N}$ of~$\widetilde{\rm GL}_n$, for $n=2,3$.

\subsection[The case of $n=2$, Theorem \ref{Thm:tauHankel}]{The case of $\boldsymbol{n=2}$, Theorem \ref{Thm:tauHankel}}\label{sec:expr-tau-funct-1}

Recall that $\mathcal{N}\subset \widetilde{GL}_{2}$ was defined as the subgroup of elements of the form \eqref{eq:7}. The inverse image of $\mathcal{N}$ under the projection $\pi\colon \widehat{\rm GL}_2 \to \widetilde {\rm GL}_{2}$ can be shown to be isomorphic to the group $\mathcal{N}\times \mathbb{C}^{\times}$ of pairs $(n,\mathsf{z})$, with multiplication $(n_{1},\mathsf{z}_{1})\cdot (n_{2},\mathsf{z}_{2})=(n_{1}n_{2},\mathsf{z}_{1}\mathsf{z}_{2})$. In other words, the central extension defining $\widehat{\rm GL}_2$ is trivial when restricted to $\pi\inv(\mathcal{N})$. Denote by $\widehat{\mathcal{N}}$ the subgroup of $\pi\inv(\mathcal{N})$ corresponding to pairs $(n,\mathsf{1})$ in $\mathcal{N}\times \mathbb{C}^{\times}$. Then $\widehat{\mathcal{N}}$ and $\mathcal{N}$ are isomorphic.

The $\tau$-functions for $\widehat{\rm GL}_2$ are given as matrix elements on $F^{(2)}$:
\begin{gather}
\tau_{k}^{(\alpha)}=\big\langle T^{k}v_{0},g_{C}^{(\alpha)}v_{0}\big\rangle,\label{eq:38}
\end{gather}
where the element $g_{C}^{(\alpha)}$ of the lower triangular subgroup $\widehat {\mathcal{N}}$ of $\widehat{\rm GL}_2$ has projection given by~\eqref{eq:28}.

We write the group element $g_{C}^{(\alpha)}$ as
\begin{gather*}
g^{(\alpha)}_{C}=\exp\big(\Gamma^{(\alpha)}_{C}\big)=\sum_{\ell=0}^{\infty}\frac1{\ell!}\big(\Gamma^{(\alpha)}_{C}\big)^{\ell},
\end{gather*}
where
\begin{gather*}
 \Gamma^{(\alpha)}_{C}=\Res_{z_{1}}\big(C^{(\alpha)}(z_{1})E_{10}(z_{1})\big), 
\end{gather*}
and $C^{(\alpha)}(z)$ is given by~\eqref{eq:31}, and the generating series of loop algebra elements $E_{10}(z_{1})$ by~\eqref{eq:2}.

The reader might object that $\Gamma^{(\alpha)}_{C}$ is an infinite sum of fermion operators each acting on $F^{(2)}$, and it is not so clear what this sum means. By imposing conditions on the coordinates $c_{k}\in \mathbb{C}$ we can ensure that $\Gamma^{(\alpha)}_{C}$ is indeed a~map $F^{(2)}\to F^{(2)}$. For our purposes it is easier to think of the $c_{i}$ as formal variables, and interpret $\Gamma^{(\alpha)}_{C}$ and $g_{C}^{(\alpha)}$ as maps $F^{(2)}\to F^{(2)}[[c_{i}]]$. (See also the comments about the interpretation of the loop group element following~\eqref{eq:7}.) This has as a~consequence that the $\tau$-function~\eqref{eq:38} is not a complex number, but a formal series in the~$c_{i}$.

$T^{k}v_{0}$ has degree $k(\delta_{1}-\delta_{0})$ in $F^{(2)}$, and $\Gamma_{C}^{{(\alpha)}}$ has degree $\delta_{1}-\delta_{0}$, since $E_{10}(z_{1})$ does. In $F^{(2)}$ homogeneous elements of different degree are orthogonal for $\langle\,,\rangle$. Hence only the $\ell=k$ term contributes to \eqref{eq:38} and
\begin{gather*}
\tau_{k}^{(\alpha)}=\frac1{k!}\big\langle T^{k}v_{0},\big(\Gamma_{C}^{(\alpha)}\big)^{k}v_{0}\big\rangle.
\end{gather*}
Recall that $C^{(\alpha)}(z)=(-1)^\alpha\sum\limits_{n\in\mathbb{Z}}c_{n+\alpha}z^{-n-1}$. Associated to the series $C^{(\alpha)}(z)$ is a $\mathbb{C}$-linear map
\begin{gather*}
c^{(\alpha)}\colon \ \mathbb{C}\big[\big[z,z\inv\big]\big]\to B=\mathbb{C}[[c_{i}]]_{i\in\mathbb{Z}}, \qquad f(z)\mapsto \Res_{z}\big(C^{(\alpha)}(z)f(z)\big).
\end{gather*}
We need multiple copies of the map $c^{(\alpha)}$ acting on series in variables $z_{1},z_{2},\dots$. We define $c^{(\alpha)}_{i}(f(z_{i}))=c^{(\alpha)}(f(z))$, for $i=1,2,\dots$, and impose linearity in the variables $z_{j}$, $j\ne i$, i.e., the condition that (for instance) $c^{(\alpha)}_{i}\big(f(z_{i})z_{j}^{m}\big)=z^{m}_{j}c^{(\alpha)}_{i}(f(z_{i}))$.

Then we can write
\begin{gather*}
 \big(\Gamma^{(\alpha)}_{C}\big)^{\ell}=\prod_{i=1}^{\ell}c^{(\alpha)}_{i}\prod_{j=1}^{\ell}E_{10}(z_{j}),
\end{gather*}
and so
\begin{gather*}
\tau_{k}^{(\alpha)}=\frac1{k!}\prod_{i=1}^{k}c^{(\alpha)}_{i}\left(\left\langle T^{k} v_{0},\prod_{j=1}^{k}E_{10}(z_{j})v_{0}\right\rangle\right)\notag\\
\hphantom{\tau_{k}^{(\alpha)}}{} =
\frac1{k!}\prod_{i=1}^{k}c^{(\alpha)}_{i} \left(\left\langle T^{k}v_{0}, \prod_{j=1}^{k}\left(\psi_{1}^{+}(z_{j})\psi_{0}^{-}(z_{j}) \right) v_{0}\right\rangle\right)\notag\\
\hphantom{\tau_{k}^{(\alpha)}}{}=\frac1{k!} \prod_{i=1}^{k}c^{(\alpha)}_{i} \left(\left\langle Q_{1}^{k}Q_{0}^{-k}v_{0}, \prod_{s=1}^{k}\psi_{1}^{+}(z_{s}) \prod_{t=1}^{k}\psi_{0}^{-}(z_{t})v_{0}\right\rangle
 \right).
\end{gather*}
Here we use the fact that, in the expression \eqref{eq:26} for $E_{10}(z_{1})$ in fermions, the normal ordering is just the ordinary product of fermion fields. We also use Lemma~\ref{Lem:GeneratorsTranslationGroupIdentities}, Part~\ref{item:6}{ and the anti-commutation relations~\eqref{eq:20}}.

Next we use the factorization Lemmas~\ref{Lem:fact} and~\ref{lem:3} to calculate the factors involving~$\psi_{0}(z)$ and $\psi_{1}(z)$ separately; we find
\begin{gather*}
\tau_{k}^{(\alpha)}= \frac1{k!}\prod_{i=1}^{k}c^{(\alpha)}_{i}\big(\det\big(V^{(k)}_{\{z_{i}\}}\big)^{2}\big).
\end{gather*}
Here $V^{(k)}_{\{z_{i}\}}$ is the Vandermonde matrix~\eqref{eq:21}.

This proves the first part of Theorem \ref{Thm:tauHankel}, since $\det\big(V^{(k)}_{\{z_{i}\}}\big)=\prod\limits_{k\ge j>i\ge 1}(z_{i}-z_{j})$.

For the second part we need a formula for the square of a Vandermonde determinant. Let the permutation group $\mathfrak{S}_{k}$ act on $\mathbb{C}[z_{1},z_{2},\dots, z_{k}]$ by permuting the subscripts.

\begin{lem}\label{Lem:squareVandermondeSumPerms}
 \begin{gather}
\det\big(V^{(k)}_{\{z_{i}\}}\big)^{2}= \sum_{\sigma\in \mathfrak{S}_{k}}\det\big(z_{\sigma(i)}^{i+j-2}\big)_{i,j=1}^{k}.\label{eq:40}
\end{gather}
\end{lem}
\begin{proof} The right hand side of \eqref{eq:40} can be written as
\begin{gather*}
\sum_{\sigma\in \mathfrak{S}_{k}}{\sigma}\det
\begin{bmatrix}
 1&z_{1}&z_{1}^{2}&\dots&z_{1}^{k-1}\\
 z_{2}&z_{2}^{2}&z_{2}^{3}&\dots&z_{2}^{k}\\
 z_{3}^{2}&z_{3}^{3}&z_{3}^{4}&\dots&z_{3}^{k+1}\\
\vdots&\vdots&\vdots&\ddots&\vdots\\
 z_{k}^{k-1}&z_{k}^{k}&z_{k}^{k+1}&\dots&z_{k}^{2k-2}
\end{bmatrix}.
\end{gather*}
From this, we see that the $\sigma=e$ term on the r.h.s.~of~(\ref{eq:40}) is
\begin{gather*}
A=\det\big(z_{i}^{i+j-2}\big)=\prod_{s=1}^{k}z_{s}^{s-1}\det\big(z_{i}^{j-1}\big) =\prod_{s=1}^{k}z_{s}^{s-1}\det\big(V^{(k)}_{\{z_{i}\}}\big).
\end{gather*}
Now for every permutation $\sigma$ in $\mathfrak{S}_{k}$ we have
\begin{gather*}
\sigma (A)=(-1)^{\abs{\sigma}}\prod_{s=1}^{k}z_{\sigma(s)}^{s-1}\det\big(V^{(k)}_{\{z_{i}\}}\big).
\end{gather*}
Summing over all permutations, we obtain
\begin{gather*}
\sum_{\sigma}\sigma(A)=\sum_{\sigma\in \mathfrak{S}_{k}}(-1)^{\abs \sigma}\prod_{s=1}^{k}z_{\sigma(s)}^{s-1}\det\big(V^{(k)}_{\{z_{i}\}}\big)=\det\big(V^{(k)}_{\{z_{i}\}}\big)^{2}.\tag*{\qed}
\end{gather*}\renewcommand{\qed}{}
\end{proof}

We then introduce a monomial column vector $v(z) =\begin{bmatrix} 1\\z\\z^{2}\\\vdots\\z^{k-1} \end{bmatrix}$. Define the Hankel matrix (with coefficients in~$B$)
\begin{gather*}
T^{(\alpha)}_{k}=c^{(\alpha)}\big(v(z)v(z)^{T}\big)=c^{(\alpha)}
\begin{bmatrix}
 1&z&z^{2}&\dots&z^{k-1}\\
z&z^{2}&z^{3}&\dots&z^{k}\\
z^{2}&z^{3}&z^{4}&\dots&z^{k+1}\\
\vdots&\vdots&\vdots&\ddots&\vdots\\
z^{k-1}&z^{k}&z^{k+1}&\dots&z^{2k-1}
\end{bmatrix}\\
\hphantom{T^{(\alpha)}_{k}}{} = (-1)^{k\alpha}\begin{bmatrix}
 c_{\alpha}&c_{\alpha+1}&c_{\alpha+2}&\dots&c_{\alpha+k-1}\\
 c_{\alpha+1}&c_{\alpha+2}&c_{\alpha+3}&\dots&c_{\alpha+k}\\
 c_{\alpha+2}&c_{\alpha+3}&c_{\alpha+4}&\dots&c_{\alpha+k+1}\\
\vdots&\vdots&\vdots&\ddots&\vdots\\
 c_{\alpha+k-1}&c_{\alpha+k}&c_{\alpha+k+1}&\dots&c_{\alpha+2k-2}
 \end{bmatrix},
\end{gather*}
where we apply $c^{(\alpha)}$ componentwise. Then the value of $\det\big(T^{(\alpha)}_{k}\big)\in B$ can be calculated in terms of Vandermonde determinants.

\begin{lem} \label{Lem-B-Heine} If $V^{(k)}_{\{z_{i}\}}$ is the Vandermonde matrix~\eqref{eq:21} then
 \begin{gather*}
\det\big(T^{(\alpha)}_{k}\big)= \frac{1}{k!}\prod_{i=1}^{k}c^{(\alpha)}_{i}\big(\det\big(V^{(k)}_{\{z_{i}\}}\big)^{2}\big).
 \end{gather*}
\end{lem}

\begin{proof} By the trivial observation that
\begin{gather*}
c^{(\alpha)}(f(z))\cdot c^{(\alpha)}(g(z))=c^{(\alpha)}_{1}(f(z_{1}))\cdot c^{(\alpha)}_{2}(g(z_{2}))
\end{gather*}
we have
\begin{gather*}
\det\big(T^{(\alpha)}_{k}\big)=\prod_{i=1}^{k}c^{(\alpha)}_{i}\det
\begin{bmatrix}
 1&z_{1}&z_{1}^{2}&\dots&z_{1}^{k-1}\\
 z_{2}&z_{2}^{2}&z_{2}^{3}&\dots&z_{2}^{k}\\
 z_{3}^{2}&z_{3}^{3}&z_{3}^{4}&\dots&z_{3}^{k+1}\\
\vdots&\vdots&\vdots&\ddots&\vdots\\
 z_{k}^{k-1}&z_{k}^{k}&z_{k}^{k+1}&\dots&z_{k}^{2k-2}
\end{bmatrix} =\prod_{i=1}^{k}c^{(\alpha)}_{i}\det \big(z_{i}^{i+j-2}\big).
 \end{gather*}
Now for any polynomial $f(z_{1},z_{2},\dots,z_{k})$ and any permutation $\sigma\in \mathfrak{S}_{k}$ we have
\begin{gather*} \prod _{i=1}^{k}c^{(\alpha)}_{i}(\sigma f(z_{1},z_{2},\dots,z_{k}))= \prod _{i=1}^{k}c^{(\alpha)}_{i}(f(z_{1},z_{2},\dots,z_{k})).\end{gather*} Hence
\begin{gather*}
\det\big(T^{(\alpha)}_{k}\big)=\frac1{k!}\prod_{i=1}^{k}c^{(\alpha)}_{i}\left(\sum_{\sigma}\sigma\det \big(z_{i}^{i+j-2}\big)\right),
\end{gather*}
and the lemma follows from the previous Lemma~\ref{Lem:squareVandermondeSumPerms}.
\end{proof}

This finishes the proof of the second part of Theorem~\ref{Thm:tauHankel}.

\subsection[Proof of $n=3$, Theorem \ref{thm:n=3taufunctions}]{Proof of $\boldsymbol{n=3}$, Theorem \ref{thm:n=3taufunctions}}\label{sec:proof-n=3}

\begin{proof}In this proof, for typographical simplicity, we will suppress the shift superscripts ${}^{(\alpha,\beta)}$.

We write $ g=\exp(\Gamma_{c})\exp(\Gamma_{d})\exp(\Gamma_{e})$, where
\begin{gather*} \Gamma_{c}=\Res_{z}(C(z_{1})E_{10}(z_{1})) , \qquad \Gamma_{d}=\Res_{z_{1}}(D(z_{1})E_{20}(z_{1})), \qquad \Gamma_{e}=\Res_{z_{1}}(E(z_{1})E_{21}(z_{1})).
\end{gather*} (Recall~(\ref{eq:61}).) This implies that $\tau_{k,\ell}$ is the sum of
\begin{gather*}
c_{n_c,n_d,n_e}=\Res_{\mathbf{x},\mathbf{y},\mathbf{z}}\left( \prod_{i=1}^{n_c}C(x_{i}) \prod_{i=1}^{n_d}D(y_{i}) \prod_{i=1}^{n_e}E(z_{i}) p_{n_c,n_d,n_e}\right),
\end{gather*}
where
\begin{gather*}
p_{n_c,n_d,n_e}=\left\langle{T_{1}^{k}T_{2}^{\ell}v_{0}, \prod_{i=1}^{n_c}\psi_{1}^{+}(x_{i})\psi_{0}^{-}(x_{i}) \prod_{i=1}^{n_d}\psi_{2}^{+}(y_{i})\psi_{0}^{-}(y_{i}) \prod_{i=1}^{n_e}\psi_{2}^{+}(z_{i})\psi_{1}^{-}(z_{i})v_{0}}\right\rangle,
\end{gather*}
where
\begin{gather*}
n_d+n_e=\ell,\qquad n_c+n_d=k.\label{eq:86}
\end{gather*}
We can factorize this using Appendix \ref{sec:reduct-one-comp} as
\begin{gather*}
p_{n_c,n_d,n_e}=(-1)^{\big( \frac{k(k-1)+\ell(\ell-1)}{2}+k\ell+ \frac{n_c(n_c-1)+n_d(n_d-1)+n_e(n_e-1)}{2}+n_cn_e\big)}\\
\hphantom{p_{n_c,n_d,n_e}=}{} \times
\left\langle{Q_{2}^{\ell}v_{0}, \prod_{i=1}^{n_d}\psi_{2}^{+}(y_{i}) \prod_{i=1}^{n_e}\psi_{2}^{+}(z_{i})v_{0}}\right\rangle\\
\hphantom{p_{n_c,n_d,n_e}=}{} \times
\left\langle{Q_{1}^{k-\ell}v_{0}, \prod_{i=1}^{n_c}\psi_{1}^{+}(x_{i}) \prod_{i=1}^{n_e}\psi_{1}^{-}(z_{i})v_{0}}\right\rangle
\left\langle{Q_{0}^{-k}v_{0}, \prod_{i=1}^{n_c}\psi_{0}^{-}(x_{i}) \prod_{i=1}^{n_d}\psi_{0}^{-}(y_{i})}\right\rangle.
\end{gather*}
Since, using \eqref{eq:86},
\begin{gather*}
(-1)^{ \frac{k(k-1)+\ell(\ell-1)}{2}+k\ell+ \frac{n_c(n_c-1)+n_d(n_d-1)+n_e(n_e-1)}{2}+n_cn_e}=(-1)^{ \frac{n_d(n_d+1)}{2}},
\end{gather*}
the theorem follows from the calculation of correlation functions in Appendix~\ref{sec:one-comp-ferm}.
\end{proof}

\section[Birkhoff factorization and matrix elements of semi-infinite wedge space]{Birkhoff factorization and matrix elements\\ of semi-infinite wedge space}\label{sec:birkh-fact-matr}

In this appendix we sketch proofs of Theorems~\ref{Thm:birkh-fact-tau} and~\ref{Thm:birkh-3x3fact-tau}. First we will discuss a more general statement about the Birkhoff factorization in~$\widetilde{\rm GL}_n$.

\subsection[Birkhoff factorization and $n$-component fermions]{Birkhoff factorization and $\boldsymbol{n}$-component fermions}\label{sec:birkh-fact-n}
Recall that most elements $\gamma\in {\rm GL}_{n}$ have a Gauss factorization:
\begin{gather*}
 \gamma=\gamma_{-}\gamma_{0+},
\end{gather*}
with $\gamma_{-}=\mathds{1}_{n\times n}+ $ strictly lower triangular, $\gamma_{0+}$ upper triangular (and invertible). Only the $\gamma$ for which the principal minors vanish don't have a Gauss factorization.

A similar story works for the loop group $\widetilde{\rm GL}_n$ of ${\rm GL}_{n}$. Let $\widetilde{\rm GL}_{n{-}}$ be the subgroup of $\widetilde{\rm GL}_n$ of elements $g_{-}=1+O\big(z\inv\big)$, and let $\widetilde{\rm GL}_{n0{+}}$ be the subgroup of elements $g_{0+}=A+O(z)$, where~$A$ is invertible (and independent of~$z$). Then most elements in~$\widetilde{\rm GL}_n$ have a \emph{Birkhoff factorization}
\begin{gather}
g=g_{-}g_{0+},\qquad g\in \widetilde{\rm GL}_n.\label{eq:9}
\end{gather}
The existence of such a factorization is controlled by the non-vanishing of a fermion matrix element in the semi-infinite wedge space~$F^{(n)}$. We will express~$g_{-}$ in terms of such fermion matrix elements..

Recall that, just as in the case of the loop algebra $\widetilde{\mathfrak{gl}}_n$, the loop group $\widetilde{\rm GL}_n$ does not actually act on $F^{(n)}$. We instead have a central extension (cf.~\cite{PrSe:LpGrps})
\begin{gather*}
1\to\mathbb{C}^{\times}\to \widehat{\rm GL}_n \overset{\pi}{\to} \widetilde{\rm GL}_n\to1,
\end{gather*}
and an action of $\widehat{\rm GL}_n$ on $F^{(n)}$. The inverse images $\pi\inv\big(\widetilde{\rm GL}_{n{-}}\big)$ and $\pi\inv\big(\widetilde{\rm GL}_{n0{+}}\big)$ can be shown to be isomorphic to product groups $\widetilde{\rm GL}_{n{-}}\times \mathbb{C}^{\times}$ and $\widetilde{\rm GL}_{n0{+}}\times \mathbb{C}^{\times}$, respectively, i.e., the central extension defining $\widehat{\rm GL}_n$ is trivial over the two inverse images. Denote by $\widehat{\rm GL}_{n{-}}$ the subgroup of $\pi\inv\big(\widetilde{\rm GL}_{n{-}}\big)$ corresponding to pairs $(g,1)$, and let $\widehat{\rm GL}_{n0{+}}$ denote the full preimage of $\widetilde{\rm GL}_{n0{+}}$. Then the intersection of $\widehat{\rm GL}_{n{-}}$ and $\widehat{\rm GL}_{n0{+}}$ will be the element $1\in\widehat{\rm GL}_n$; the image of $\mathbb{C}^{\times}$ belongs to $\widehat{\rm GL}_{n0{+}}$. Most elements $\hat{g}\in \widehat{\rm GL}_n $ will have a (unique) Birkhoff factorization,
\begin{gather*}
\hat{g}=\hat{g}_{-}\hat{g}_{0+}.
\end{gather*}
If $v_{0}$ is the vacuum \eqref{eq:4} of $F^{(n)}$, the $\tau$-function is defined as the matrix element
\begin{gather}
\tau(\hat{g})=\langle v_{0},\hat g v_{0}\rangle.
\label{eq:37}
\end{gather}
The element $\hat g$ (and also $g=\pi(\hat g)$) has a Birkhoff factorization as long as $\tau(\hat g)$ is not zero.

To calculate the negative component of $g$ in the factorization~\eqref{eq:9}, we choose a lift $\hat{g}$ of $g$, i.e., $\pi(\hat{g})=g$, and study the action of $\hat{g}$ on~$F^{(n)}$.

We have
\begin{gather*}
\hat{g}_{0+}v_{0}=\tau(\hat g)v_{0}.
\end{gather*}
This is explained in the case $n=2$ in~\cite{MR90b:5810} (cf.~\cite{MR699439,MR771354}). Hence (assuming $\hat g$ has a Birkhoff factorization, or $\tau(\hat g)\ne0$)
\begin{gather*}
 g_{-}v_{0}=g_{-}\hat {g}_{0+}v_{0}/\tau(\hat g)=\hat{g}v_{0}/\tau(\hat g).
\end{gather*}
Now write $g_{-}$ in terms of matrix elements
\begin{gather*}
g_{-}=\sum_{a,b=0}^{n-1}g_{ab}(z)E_{ab},
\end{gather*}
where
\begin{gather*}
g_{ab}(z)=\sum_{k\in\mathbb{Z}}g_{ab}^{(k)}z^{-k-1}, \qquad g_{ab}^{(-1)}=\delta_{ab},\qquad g_{ab}^{(l)}=0, \quad \text{if} \ \ l<-1. 
\end{gather*}
On $F^{(n)}$ (see Section~\ref{sec:lie-algebra-sln})
\begin{gather*}
E_{ab}z^{-k-1}=\colon\sum_{l}\big(e_{a}z^{l-k-1}\wedge\big)\big(i\big(e_{b}z^{l}\big)\big)\colon=\colon\sum_{l\in\mathbb{Z}}\fermion{a}{l-k-1}{+}\fermion{b}{-l-1}{-}\colon.\label{eq:17}
\end{gather*}
For the $E_{ab}z^{-k-1}$ appearing in $g_{-}$ (i.e., those with $k\ge0$) the normal ordering can be omitted, see~\eqref{eq:77}.

Now to find $g_{ab}^{(k)}$ we calculate
\begin{gather*}
g_{-}v_{0}=v_{0}+\sum_{k\ge0}g_{ab}^{(k)}\sum_{l\in\mathbb{Z}}\fermion{a}{l-k-1}{+}\fermion{b}{-l-1}{-}v_{0}+ \cdots,
\end{gather*}
where the omitted terms are quadratic and higher in the $E_{ab}$. We see that $g_{ab}^{(k)}$ appears as the coefficient of many elementary wedges $\fermion{a}{l-k-1}{+}\fermion{b}{-l-1}{-}v_{0}$. To calculate $g_{ab}^{(k)}$ we just pick one of these elementary wedges, say the $l=0$ term, and use orthogonality of elementary wedges to find ($k\ge0$)
\begin{gather}
g_{ab}^{(k)} =\big\langle \fermion{a}{-k-1}{+} \fermion{b}{-1}{-}v_{0},g_{-}v_{0}\big\rangle =\big\langle \fermion{a}{-k-1}{+}\fermion{b}{-1}{-}v_{0},\hat gv_{0}\big\rangle/\tau(\hat{g})\nonumber\\
\hphantom{g_{ab}^{(k)}}{} =\big\langle Q_{b}^{-1}v_{0},\fermion{a}{k}{-}\hat gv_{0}\big\rangle/\tau(\hat{g}),\label{eq:36}
\end{gather}
using \eqref{eq:14} and \eqref{eq:16}.

Now observe that
\begin{gather*}
\colon \fermion{a}{0}{+}\fermion{b}{-1}{-}\colon v_{0}=\delta_{ab}v_{0},
\end{gather*}
and
\begin{gather*}
\colon \fermion{a}{l}{+}\fermion{b}{-1}{-}\colon v_{0}=0,\qquad l>0.
\end{gather*}
This allows us to calculate $g_{ab}^{(k)}$ for $k<0$ in the same way as for $k\ge0$, see \eqref{eq:36}.

These remarks prove the following theorem.

\begin{thm}\label{Thm:BirkhoffFactFermionMatrixElement} Let $g\in\widetilde{\rm GL}_n$ admit a Birkhoff factorization $g=g_{-}g_{0+}$. Then
\begin{gather*}
g_{-}=\sum_{a,b=0}^{n-1}g_{ab}(z)E_{ab},
\end{gather*}
where
\begin{gather*}
g_{ab}(z)=\big\langle Q_{b}^{-1}v_{0},\psi_{a}^{-}(z)\hat gv_{0}\big\rangle/\tau(\hat{g}), 
\end{gather*}
and $\hat g\in \widehat{\rm GL}_n$ is any lift of $g$, so that $\pi(\hat g)=g$. The $\tau$-function is given by~\eqref{eq:37}.
\end{thm}

\subsection[The $2\times2$-case; Proof of Theorem~\ref{Thm:birkh-fact-tau}]{The $\boldsymbol{2\times2}$-case; proof of Theorem~\ref{Thm:birkh-fact-tau}}\label{sec:2times2-case-proof}

We now specialize $n$ and $g$ in Theorem~\ref{Thm:BirkhoffFactFermionMatrixElement}: in this subsection we set $n=2$ and
\begin{gather*}
g=\pi\big(g^{[k](\alpha)}\big)=\pi\big(T^{-k}\big)\pi\big(g^{(\alpha)}_{C}\big),
\end{gather*}
where $\pi(g^{(\alpha)}_{C})$ is given by \eqref{eq:28}. We have two interpretations of $g^{[k](\alpha)}$: if we choose the coefficients in $C$ suitably, $g^{(\alpha)}_{C}$ gives a well defined operator on~$F^{(2)}$, and the corresponding $\tau$-function will be a complex number. If the $\tau$-function is not zero, $g$ will have a~Birkhoff factorization. However, we prefer to think of $g^{[k](\alpha)}$ as a map $F^{(2)}\to F^{(2)}[[c_{i}]]$, so that the $\tau$-function is also a~formal series in $\mathbb{C}[[c_{i}]]$, which is not zero, and so the ``formal group element'' $g$ in this case will always have a Birkhoff factorization.

In this case our calculation of the Birkhoff factorization in Theorem~\ref{Thm:BirkhoffFactFermionMatrixElement} gives us
\begin{gather*}
g^{[k](\alpha)}_{-}=\sum_{a,b=0}^{1}g^{[k](\alpha)}_{ab}(z)E_{ab},\qquad g^{[k](\alpha)}_{ab}(z)=\big\langle Q_{b}^{-1}v_{0},\psi_{a}^{-}(z)T^{-k}g_{C}^{(\alpha)}v_{0}\big\rangle/\tau_{k}^{(\alpha)},
\end{gather*}
where the $\tau$-function is defined in \eqref{eq:32}. We will proceed to rewrite $g_{ab}^{[k](\alpha)}(z)$.

First of all, we will expand $g_{C}^{(\alpha)}$ in fermion operators. Note that
\begin{gather*}
g^{(\alpha)}_{C}=\exp\big(\Gamma^{(\alpha)}_{C}\big),
\end{gather*}
where
\begin{gather}
 \Gamma^{(\alpha)}_{C}=\Res_{z_{1}}\big(C^{(\alpha)}(z_{1})E_{10}(z_{1})\big),\label{eq:30}
\end{gather}
and $C^{(\alpha)}(z)$ is given by \eqref{eq:31}, and the generating series $E_{10}(z_{1})$ by~\eqref{eq:2}. (See the beginning of Appendix~\ref{sec:expr-tau-funct-1} for an interpretation of $g_{C}^{(\alpha)}$ and~$\Gamma^{(\alpha)}_{C}$ as operators on~$F^{(2)}$.)

This means that $g^{(\alpha)}_{C}=\sum\limits_{\ell\ge0}\big(\Gamma^{(\alpha)}_{C}\big)^{\ell}/\ell!$, both acting on $H^{(2)}$ and on $F^{(2)}$. Hence
\begin{gather}
 g^{[k](\alpha)}_{ab}(z)=\sum_{\ell\ge0} \frac1{\ell!}\big\langle Q_{b}\inv v_{0}, \psi_{a}^{-}(z)T^{-k}\big(\Gamma^{(\alpha)}_{C}\big)^{\ell}v_{0}\big\rangle /\tau_{k}^{(\alpha)}. \label{eq:33}
\end{gather}
Next, we use the standard grading on $F^{(2)}$. Note that $\Gamma_{C}^{(\alpha)}$ has degree $\delta_{1}-\delta_{0}$ (since $E_{10}(z_{1})$ does). So $Q_{b}\inv$ has degree $-\delta_{b}$ and $\psi_{a}^{-}(z)T^{-k}\big(\Gamma^{(\alpha)}_{C}\big)^{\ell}$ has degree $k(\delta_{0}-\delta_{1})+\ell(\delta_{1}-\delta_{0})-\delta_{a}$. Hence, by orthogonality of terms of different degree in $F^{(2)}$, we find that the only non-zero contribution to the sum~\eqref{eq:33} arises when $\ell=k+a-b$ and
\begin{gather*}
 g^{[k](\alpha)}_{ab}(z)=\frac1{\ell!}\big\langle Q_{b}\inv v_{0}, \psi_{a}^{-}(z)T^{-k}\big(\Gamma^{(\alpha)}_{C}\big)^{\ell}v_{0}\big\rangle/\tau_{k}^{(\alpha)}.
\end{gather*}
From now on, we will often use the abbreviation $\ell=k+a-b$ in formulas for $g^{[k](\alpha)}_{ab}(z)$.

Next, we need to commute $T^{-k}$ through $\psi_{a}^{-}(z)$.
\begin{lem}
 If $T=Q_{1}Q_{0}\inv$ then
\begin{gather*}
\psi_{a}^{-}(z)T^{-k}=(-1)^{k}z^{k(2a-1)}T^{-k}\psi_{a}^{-}(z).
\end{gather*}
\end{lem}
\begin{proof}
By \eqref{eq:23}, \eqref{eq:24} we have
\begin{gather*}
 \psi_{a}^{-}(z)T\inv= \psi_{a}^{-}(z)Q_{0}Q_{1}\inv = \begin{cases}
 z\inv Q_{0}\psi_{a}^{-}(z)Q_{1}\inv&\text{ if }a=0,\\
 -Q_{0}\psi_{a}^{-}(z)Q_{1}\inv&\text{ if }a=1
 \end{cases}\\
 \hphantom{\psi_{a}^{-}(z)T\inv= \psi_{a}^{-}(z)Q_{0}Q_{1}\inv}{} =
 \begin{cases}
 -z\inv T\inv\psi_{a}^{-}(z)&\text{ if }a=0,\\
 -zT\inv\psi_{a}^{-}(z)&\text{ if }a=1
 \end{cases} =-z^{2a-1}T\inv \psi_{a}^{-}(z).\tag*{\qed}
\end{gather*}\renewcommand{\qed}{}
\end{proof}

By unitarity of the translation operators, \eqref{eq:22}, this implies that
\begin{gather}
g_{ab}^{[k](\alpha)}(z) =\frac{(-1)^{k}z^{k(2a-1)}}{\ell!}\big\langle Q_{b}^{-1}v_{0},T^{-k}\psi_{a}^{-}(z) \big(\Gamma_{C}^{(\alpha)}\big)^{\ell}v_{0}\big\rangle/\tau_{k}^{(\alpha)}\notag\\
\hphantom{g_{ab}^{[k](\alpha)}(z)}{} = \frac{(-1)^{k}z^{k(2a-1)}}{\ell!}\big\langle T^{k}Q_{b}\inv v_{0},\psi_{a}^{-}(z)\big(\Gamma_{C}^{(\alpha)}\big)^{\ell}v_{0}\big\rangle/\tau_{k}^{(\alpha)}.\label{eq:34}
\end{gather}

Next we want to apply the factorization Lemma~\ref{Lem:fact}. We need to write $T^{k}Q_{b}\inv$ in standard form $Q_{1}^{\alpha}Q_{0}^{\beta}$.

\begin{lem}\label{lem:1}
 \begin{gather*}
T^{k}Q_{b}\inv=(-1)^{\frac{k(k-1)}2}(-1)^{bk}Q_{1}^{k-b}Q_{0}^{-k-1+b}.
\end{gather*}
\end{lem}
\begin{proof}
By Lemma \ref{Lem:GeneratorsTranslationGroupIdentities}, Part~\ref{item:6}
\begin{gather*}
 T^{k}Q_{b}\inv =(-1)^{\frac{k(k-1)}2}Q_{1}^{k}Q_{0}^{-k}Q_{b}\inv= (-1)^{\frac{k(k-1)}2}
\begin{cases}
 Q_{1}^{k}Q_{0}^{-k-1}&\text{if }b=0,\\
(-1)^{k}Q_{1}^{k-1}Q_{0}^{-k}&\text{if }b=1.
\end{cases}\tag*{\qed}
\end{gather*}\renewcommand{\qed}{}
\end{proof}

Now we are going to express $\big(\Gamma_{C}^{(\alpha)}\big)^{\ell}$ in terms of fermion fields, see~\eqref{eq:30} and~\eqref{eq:26}. Recall that $C^{(\alpha)}=(-1)^\alpha\sum\limits_{n\in\mathbb{Z}}c_{n+\alpha}z^{-n-1}$. Associated to the series $C^{(\alpha)}$ we have a map
\begin{gather*}
c^{(\alpha)}\colon \ \mathbb{C}\big[\big[z,z\inv\big]\big]\to B=\mathbb{C}[[c_{i}]]_{i\in\mathbb{Z}}, \qquad f(z)\mapsto \Res_{z}\big(C^{(\alpha)}(z)f(z)\big).
\end{gather*}
We will need multiple copies of the map $c^{(\alpha)}$ acting on series in variables $z_{1},z_{2},\dots$. We define $c^{(\alpha)}_{i}(f(z_{i}))=c^{(\alpha)}(f(z))$, for $i=1,2,\dots$, and impose linearity in the variables $z_{j}$, $j\ne i$, i.e., the condition that (for instance) $c^{(\alpha)}_{i}(f(z_{i})z_{j}^{m})=z^{m}_{j}c^{(\alpha)}_{i}(f(z_{i}))$.

Then we can write
\begin{gather}
 \big(\Gamma^{(\alpha)}_{C}\big)^{\ell} = \prod_{i=1}^{\ell}c^{(\alpha)}_{i}\prod_{j=1}^{\ell}E_{10}(z_{j})=\prod_{i=1}^{\ell}c^{(\alpha)}_{i} \prod_{j=1}^{\ell}\psi_{1}^{+}(z_{j})\psi_{0}^{-}(z_{j})\notag\\
\hphantom{\big(\Gamma^{(\alpha)}_{C}\big)^{\ell}}{} =(-1)^{\frac{\ell(\ell-1)}2} \prod_{i=1}^{\ell}c^{(\alpha)}_{i} \prod_{s=1}^{\ell}\psi_{1}^{+}(z_{s}) \prod_{t=1}^{\ell}\psi_{0}^{-}(z_{t}).\label{eq:41}
\end{gather}
Here we have used the fact that in the expression of $E_{10}(z_{1})$ in fermions, the normal ordering is just the ordinary product of fermion fields. We also use Lemma~\ref{Lem:GeneratorsTranslationGroupIdentities}, Part~\ref{item:6}.

Using Lemma~\ref{lem:1} and \eqref{eq:41} in \eqref{eq:34}, we obtain
\begin{gather*}
 g_{ab}^{[k](\alpha)}(z)=\epsilon_{ab}^{[k](\alpha)}(z) \prod_{i=1}^{\ell}c^{(\alpha)}_{i} \left\langle Q_{1}^{k-b}Q_{0}^{-k-1+b}v_{0},\psi_{a}^{-}(z) \prod_{s=1}^{\ell}\psi_{1}^{+}(z_{s})
 \prod_{t=1}^{\ell}\psi_{0}^{-}(z_{t})v_{0}\right\rangle/\tau_{k}^{(\alpha)},
\end{gather*}
where
\begin{gather*}
\epsilon_{ab}^{[k](\alpha)}(z)=\frac{(-1)^{k}z^{k(2a-1)}}{\ell!}(-1)^{\frac{\ell(\ell-1)}2}(-1)^{\frac{k(k-1)}{2}}(-1)^{kb}.
\end{gather*}
Here, we still have $\ell=k+a-b$. Now using the factorization Lemma~\ref{Lem:fact}, we get
\begin{gather*}
g_{ab}^{[k](\alpha)}(z) =\frac{(-1)^{(a+1)\ell}\epsilon_{ab}^{[k](\alpha)}(z)}{\tau_{k}^{(\alpha)}} \prod_{i=1}^{\ell}c_{i}^{(\alpha)}
\bigg \langle Q_{1}^{k-b}v_{0},
 \left.
 \begin{cases}
 \displaystyle \prod_{j=1}^{\ell}\psi_{1}^{+}(z_{j})v_{0}\bigg\rangle,&a=0,\\
 \displaystyle \psi_{1}^{-}(z)\prod_{j=1}^{\ell}\psi_{1}^{+}(z_{j})v_{0}\bigg\rangle,&a=1
 \end{cases}\right\}\\
\hphantom{g_{ab}^{[k](\alpha)}(z) =}{}
\times\bigg\langle Q_{0}^{-k-1+b}v_{0},
\left.
 \begin{cases}
 \displaystyle \psi_{0}^{-}(z)\prod_{j=1}^{\ell}\psi_{0}^{-}(z_{j})v_{0}\bigg\rangle,&a=0,\\
\displaystyle \prod_{j=1}^{\ell}\psi_{0}^{-}(z_{j})v_{0}\bigg\rangle,&a=1
 \end{cases}\right\}.
\end{gather*}
Using Lemma \ref{lem:6} to calculate the factors involving $\psi_{0}(z)$ and $\psi_{1}(z)$ separately we find
\begin{gather*}
g^{[k](\alpha)}_{ab}(z)=\frac{(-1)^{(a+1)\ell} \epsilon_{ab}^{[k](\alpha)}(z)}{\tau_{k}^{(\alpha)}}\prod_{i=1}^{\ell}c_{i}^{(\alpha)}\left( \det\big(V^{(\ell)}_{\{z_{i}\}}\big)^{2}\prod_{j=1}^\ell(z-{z}_{j})^{1-2a}\right).
\end{gather*}
Comparing this with the expression for the $\tau$-function in Theorem~\ref{Thm:tauHankel} and the definition of the shift fields~\eqref{eq:10} gives Theorem~\ref{Thm:birkh-fact-tau}.

\subsection[Birkhoff factorization in the $3\times 3$ case, proof of Theorem \ref{Thm:birkh-3x3fact-tau}]{Birkhoff factorization in the $\boldsymbol{3\times 3}$ case, proof of Theorem \ref{Thm:birkh-3x3fact-tau}}\label{sec:birkh-fact-3tim}

The proof of Theorem~\ref{Thm:birkh-3x3fact-tau} is similar to that of Theorem~\ref{Thm:birkh-fact-tau} sketched in the previous subsection. We leave the details to the reader.

\section{Factorization and reduction to one-component fermions}\label{sec:reduct-one-comp}

Often we want to calculate a matrix element in $F^{(n)}$ of fermion fields of the form
\begin{gather*}
\big\langle Q_{n-1}^{\alpha_{n-1}}\cdots Q_{1}^{\alpha_{1}}Q_{0}^{\alpha_{0}}v_{0},P\big(\psi_{a}^{\pm}(z_{a})\big)v_{0}\big\rangle,
\end{gather*}
where $P$ is some polynomial in the fermion fields $\psi_{a}^{\pm}(z_{a})$, $a=0,1,\dots,n-1$. By linearity, we can reduce to the case where $P=M$ is a monomial, and then we can rearrange the factors in the monomial as in Definition~\ref{Def:ElemntaryWedge}: $M=M_{n-1}\cdots M_{1}M_{0}$, $M_{a}=M_{a}^{+}M_{a}^{-}$, where $M_{a}^{\pm}$ is a monomial in a~single type of fermions, ordered according to the subscript of the arguments of the fields:
\begin{gather*}
M_{a}^{\pm}=\mathop{\overleftarrow\prod}\limits_{i=1}^{t}\psi_{a}^{\pm}(z_{i})= \psi_{a}^{\pm}(z_{t})\psi_{a}^{\pm}(z_{t-1})\cdots\psi^{\pm}_{a}(z_{2})\psi_{a}^{\pm}(z_{1}). 
\end{gather*}
This defines the ordered product of fermion fields.

We calculate such matrix elements using the following factorization lemma.

Let us first introduce some notation. Let $F=F^{(1)}$ be 1-component fermionic Fock space, with vacuum $v_{0}^{(1)}$, and with bilinear form $\langle\,,\,\rangle_{F}$. The fermionic translation operators $Q$,~$Q\inv$, defined as in Section~\ref{sec:ferm-transl-oper}, act on~$F$. Let $N_{a}=N_{a}\big(\psi_{a}^{\pm}(z^{a}_{i})\big)$, $a=0,\dots, n-1$, be monomials in fermion fields $\psi_{a}^{+}(z^{a}_{i})$, $\psi_{a}^{-}(z^{a}_{j})$ of just type $a$ (acting on $F^{(n)}$), and let $\overline N_{a}$ be the corresponding monomial in 1-component fermion fields $\psi^{\pm}(z^{a}_{i})$ (acting on $F$). So for example if
\begin{gather*}
N_{a}=\psi_{a}^{+}(z^{a}_{1})\psi_{a}^{+}(z^{a}_{2})\psi^{-}_{a}(z^{a}_{3}),
\end{gather*}
then
\begin{gather*}
\overline N_{a}=\psi^{+}(z^{a}_{1})\psi^{+}(z^{a}_{2})\psi^{-}(z^{a}_{3}).
\end{gather*}

\begin{lem}\label{Lem:fact}Then
 \begin{gather*}
 \big\langle Q_{n-1}^{\alpha_{n-1}}\cdots Q_{1}^{\alpha_{1}}Q_{0}^{\alpha_{0}}v_{0}, N_{n-1}\cdots N_{1}N_{0}v_{_{0}}\big\rangle_{F^{(n)}}\\
 \qquad{} = \big\langle Q^{\alpha_{n-1}}v_{0}^{(1)},\overline N_{n-1}v_{0}^{(1)}\big\rangle_{F}\cdots \big\langle Q^{\alpha_{1}}v_{0}^{(1)},\overline N_{1}v_{0}^{(1)} \big\rangle_{F}
\big\langle Q^{\alpha_{0}}v_{0}^{(1)},\overline N_{0}v_{0}^{(1)} \big\rangle_{F}.
\end{gather*}
\end{lem}

\begin{proof} A basis for $F$ is given by elementary wedges (see Definition~\ref{Def:ElemntaryWedge}) $\omega_{\underline k}$, labeled by pairs of sequences $\underline k=(\underline k^{+},\underline k^{-})$, where each sequence $\underline k^{\pm}$ is of length $l^{\pm}$ (in general $l^{+}\ne l^{-}$) and is of the form $\underline k^{\pm}=\big(k_{1}^{\pm}<k_{2}^{\pm}<\cdots<k_{l^{\pm}}^{\pm}\le-1\big)$. Roughly speaking $\omega_{\underline k}$ is obtained from the vacuum~$v_{0}^{(1)}$ by deleting factors $z^{-k_{i}^{-}-1}$ from the vacuum{ (using the contraction operators $i\big(z^{-k_{i}^{-}-1}\big)$)} and then adding factors $z^{k_{j}^{+}}$ (using the wedging operators $\big(z^{k_{j}^{+}}\wedge\big)$). Since the order in which we apply these operations matters, we must be more precise. So define
\begin{gather*}
\omega_{\underline k}=M(\underline k)v_{0}^{(1)}=M^{+}(\underline k^{+})M^{-}(\underline k^{-})v_{0}^{(1)},
\end{gather*}
where
\begin{gather*}
M^{\pm}(\underline k^{\pm})=\fermion{}{k^{\pm}_{1}}{\pm}\fermion{}{k_{2}^{\pm}}{\pm}\cdots \fermion{}{k^{\pm}_{s}}{\pm}.
\end{gather*}
Here, we still have $ k^{\pm}_{1}<k^{\pm}_{2}<\dots<k^{\pm}_{s}\le-1$.

Similarly, elementary wedges in $F^{(n)}$ are labelled by $n$-tuples $(\underline k_{n-1},\underline k_{n-2},\dots, \underline k_{0})$, and are defined by
\begin{gather*}
\omega_{(\underline k_{n-1},\underline k_{n-2},\dots, \underline k_{0})}= M_{n-1}(\underline k_{n-1})M_{n-2}(\underline k_{n-2})\cdots M_{0}(\underline k_{0})v_{0}^{(n)}.
\end{gather*}
For the duration of the proof, we will write $v_0^{(n)}$ for the vacuum in $F^{(n)}$.

Now define a multilinear map from the $n$-fold product of $F$ with itself to $n$-component fermionic Fock space $F^{(n)}$:
\begin{gather*}
\phi \colon \ F\times F\times \cdots \times F\to F^{(n)},\qquad (\omega_{\underline k_{n-1}}, \omega_{\underline k_{n-2}},\dots, \omega_{\underline k_{0}}) \mapsto\omega_{(\underline
k_{n-1},\underline k_{n-2},\dots, \underline k_{0})}.
\end{gather*}
By the universal property of the tensor product, this induces a unique map
\begin{gather*}
\tilde \phi\colon\ F\otimes F\otimes \cdots \otimes F\to F^{(n)},\qquad \omega_{\underline k_{n-1}}\otimes \omega_{\underline k_{n-2}}\otimes\dots\otimes \omega_{ \underline k_{0}}\mapsto
\omega_{(\underline k_{n-1},\underline k_{n-2},\dots, \underline k_{0})}.
\end{gather*}
This map is an isomorphism of vector spaces, and is in fact an isometry, if we define a bilinear form on $F\otimes F\otimes\dots\otimes F$ by declaring the basis $\{ \omega_{\underline k_{n-1}}\otimes \omega_{\underline
 k_{n-2}}\otimes\cdots\otimes \omega_{ \underline k_{0}}\}$ to be orthonormal. For this bilinear form on $F^{\otimes n}$ we have (given $\omega_{a},\tilde\omega_{a}\in F$)
\begin{gather} \label{eq:87}
 \langle \omega_{n-1}\otimes\omega_{n-2}\otimes\dots\otimes\omega_{0},\tilde\omega_{n-1}\otimes\tilde\omega_{n-2}\otimes\dots\otimes\tilde\omega_{0}\rangle_{F^{\otimes n}}=\prod \langle \omega_{a},\tilde \omega_{a}\rangle_{F}.
\end{gather}
Now one checks that
\begin{gather*}
\tilde \phi\big(Q^{\alpha_{n-1}}v_{0}^{(1)}\otimes Q^{\alpha_{n-2}}v_{0}^{(1)}\otimes\cdots\otimes Q^{\alpha_{0}}v_{0}^{(1)}\big)=Q_{n-1}^{\alpha_{n-1}}Q_{n-2}^{\alpha_{n-2}}\cdots Q_{0}^{\alpha_{0}}v_{0}^{(n)},
\end{gather*}
and
\begin{gather*}
\tilde \phi\big(\overline N_{n-1}v_{0}^{(1)}\otimes \overline N_{n-2} v_{0}^{(1)}\otimes\overline N_{0}v_{0}^{(1)} \big)=N_{n-1}N_{n-2}\cdots N_{0}v_{0}^{(n)}.
\end{gather*}
The lemma then follows from \eqref{eq:87} and the fact that $\tilde \phi$ is an isometry.
\end{proof}

So the correlation functions we want to calculate reduce to products of correlation functions on one-component semi-infinite wedge space~$F$. In Appendix~\ref{sec:one-comp-ferm} we review some formulas for one component fermions.

\section{One-component fermion correlation functions} \label{sec:one-comp-ferm}
In this appendix we collect some results on one-component fermions. In other words, we are dealing with the fermionic Fock space $F=F^{(1)}$, based on $H=H^{(1)}$. The results in this Appendix should be known, for instance Lemma~\ref{LEMMA E.3} can be found (without proof) in~\cite{MR3027554}, but we could not find references with complete proofs of the facts we need.

The whole discussion of Appendix~\ref{sec:ferm-semi-infin} transfers to the present one-component context. For typographical convenience we will write $\psi^{\pm}(z)$ for $\psi_{0}^{\pm}(z)$ and similarly we write $Q^{\pm1}$ for~$Q_{0}^{\pm1}$.

The correlation functions are matrix elements in $F=F^{(1)}$ of the form $\langle Q^{k}v_{0},M(z_{i},w_{j})v_{0}\rangle$, where $M$ is some monomial in $\psi(z_{i})$ and $\psi(w_{j})$. For reasons of orthogonality of distinct charges we need to insert a power of the fermionic translation operator~$Q$.

The simplest case of a correlation function is
\begin{gather*}
\big\langle Q^{\pm1}v_{0},\psi^{\pm}(z)v_{0}\big\rangle= \left\langle \psi^{\pm}_{(-1)}v_{0},\sum_{k\in\mathbb{Z}}z^{k}\psi^{\pm}_{(-k-1)}v_{0}\right\rangle =1,
\end{gather*}
since only the $k=0$ term contributes.
\begin{lem}\label{lem:3}
 For all $k\ge1$ we have
\begin{gather*}
\left\langle Q^{\pm k}v_{0}, \mathop{\overleftarrow\prod}\limits_{i=1}^{k}\psi^{\pm}(z_{i})v_{0}\right\rangle= \det\big(V^{(k)}_{\{z_{i}\}}\big).
\end{gather*}
\end{lem}

Here the Vandermonde matrix in $\{z_{i}\}=\{z_{1},z_{2},\dots,z_{k}\}$ is given by
\begin{gather}
V^{(k)}_{\{z_{i}\}}=\begin{bmatrix}
 1&1&\dots&1\\
 z_{1}&z_{2}&\dots &z_{k}\\
 z_{1}^{2}&z_{2}^{2}&\dots &z_{k}^{2}\\
\vdots&\vdots&\ddots&\vdots\\
 z_{1}^{k-1}&z_{2}^{k-1}&\dots &z_{k}^{k-1}\\
\end{bmatrix}.\label{eq:21}
\end{gather}
Then we have
\begin{gather*}
 \det\big(V^{k}_{\{z_{i}\}}\big)=\prod_{k\ge\alpha>\beta\ge1}(z_{\alpha}-z_{\beta}).
\end{gather*}

We need
\begin{lem}\label{LEMMA E.3}
 \begin{gather*}
\big\langle Q^{m-n}v_0,\psi^{+}(w_1)\cdots \psi^{+}(w_m)\psi^{-}(y_1)\cdots \psi^{-}(y_n)v_0\big\rangle
= \frac{ \prod\limits_{1\le i<j\le m}(w_i-w_j) \prod\limits_{1\le i<j\le n}(y_i-y_j)} {\prod\limits_{i=1}^{m}\prod\limits_{j=1}^{n}(w_i-y_j)}.
 \end{gather*}
\end{lem}

\begin{proof} We first consider the case that $m<n$. So
 \begin{gather*}
 \big\langle Q^{m-n}v_0,\psi^{+}(w_1)\cdots \psi^{+}(w_m)\psi^{-}(y_1)\cdots \psi^{-}(y_n)v_0\big\rangle \\
 \qquad{} = \big\langle\psi_{(m-n)}^{-}\psi_{(m-n+1)}^{-}\cdots \psi_{(-1)}^{-}v_0,\psi^{+}(w_1)\cdots \psi^{+}(w_m)\psi^{-}(y_1)\cdots \psi^{-}(y_n)v_0\big\rangle,
 \end{gather*}
which is equal to the sum of the coefficients corresponding to all ways of pulling out
\begin{gather*} \psi_{(m-n)}^{-}\psi_{(m-n+1)}^{-}\cdots \psi_{(-1)}^{-}v_0,
\end{gather*} from the product of fermionic fields acting on the vacuum,
\begin{gather*}
\psi^{+}(w_1)\cdots \psi^{+}(w_m)\psi^{-}(y_1)\cdots \psi^{-}(y_n)v_0.
\end{gather*}
Given a $\sigma\in \mathfrak{S}_n$, we reorder the fermionic fields and record the sign obtained from doing so
\begin{gather*}
(-1)^{\abs{\sigma}}\psi^{+}(w_1)\cdots\psi^{+}(w_m)\psi^{-}(y_{\sigma(1)})\cdots \psi^{-}(y_{\sigma(n)})v_0.
\end{gather*}
We then pull out terms in such a way that no additional sign changes occur from permu\-ting the operators: Pull out $\psi_{(-1)}^{-}$ from $\psi^{-}(y_{\sigma(n)})$, $\psi_{(-2)}^{-}$ from $\psi^{-}(y_{\sigma(n-1)})$, $\dots$, and $\psi_{(m-n)}^{-}$ from~$\psi^{-}(y_{\sigma(m-n)})$. The product of the coefficients corresponding to these choices is
\begin{gather*} y_{\sigma(n)}^{0}y_{\sigma(n-1)}^{1}\cdots y_{\sigma(m+1)}^{n-m-1}.\end{gather*}
The remaining contributions for this choice of $\sigma$ come from pulling out coefficients of products of wedging and contracting operators from
\begin{gather*}
\psi^{+}(w_1)\cdots\psi^{+}(w_m)\psi^{-}(y_{\sigma(1)})\cdots \psi^{-}(y_{\sigma(m)})
\end{gather*}
whose actions cancel with each other. We again pull out these terms such a way that no additional sign changes occur: We count only contributions coming from terms in $\psi^{+}(w_m)\psi^{-}(y_{\sigma(1)})$ that cancel with each other, terms in $\psi^{+}(w_{m-1})\psi^{-}(y_{\sigma(2)})$ that cancel with each other, $\dots$ and terms in $\psi^{+}(w_1)\psi^{-}(y_{\sigma(m)})$ that cancel with each other.

We claim that we can count each pair, $\psi^{+}(w_{m-i})\psi^{-}(y_{\sigma(i+1)})$, $0\le i\le m-1$, as contributing $\frac{1}{w_{m-i}-y_{\sigma(i+1)}}= \sum\limits_{\ell=0}^{\infty} \frac{y_{\sigma(i+1)}^{\ell}}{w_{m-i}^{\ell+1}}$. We know we are not omitting any nontrivial terms in doing this, since any $\frac{y_{\sigma(i+1)}^{\ell}}{w_{m-i}^{\ell+1}}$ with $\ell<0$ corresponds to $\psi^{+}_{(\ell)}\psi^{-}_{(-\ell-1)}$ and $\psi^{-}_{(-\ell-1)}$ kills the vacuum or any vector obtained by acting by contracting operators on the vacuum. We must therefore only prove that we are not including any extra nontrivial terms. Towards this end, consider some monomial,
\begin{gather*}
(-1)^{\abs{\sigma}}y_{\sigma(n)}^{0}y_{\sigma(n-1)}\cdots
y_{\sigma(m+1)}^{n-m-1}\frac{y_{\sigma(1)}^{\ell_1}}{w_m^{\ell_{1}+1}}
\frac{y_{\sigma(2)}^{\ell_2}}{w_{m-1}^{\ell_{2}+1}}\cdots
\frac{y_{\sigma(m)}^{\ell_m}}{w_1^{\ell_{m}+1}},
\end{gather*}
corresponding to a product of wedging operators acting on the vacuum vector which give 0. Since all of the wedging operators, $\psi^{-}_{(\ell)}$, are such that $\ell<0$, the only way this is possible is if two of the wedging operators are the same. But this means that two of the $y_{i}$s in the above expression are being raised to the same power. Define a new element, $\gamma \in \mathfrak{S}_n$ by composing $\sigma$ with the transposition that interchanges these two $y_{i}$s. The sign of this new element is $-(-1)^{\abs{\sigma}}$. So there is a monomial in the expansion of
\begin{gather*}(-1)^{\abs{\gamma}}
\frac{y_{\gamma(n)}^{0}y_{\gamma(n-1)}\cdots y_{\gamma(m+1)}^{n-m-1}}{\prod\limits_{i=0}^{m-1}(w_{m-i}-y_{\gamma(i+1)})},\end{gather*}
which cancels with the above monomial.

Summing over all $\sigma \in \mathfrak{S}_n$, we have that
\begin{gather*}
 \big\langle Q^{m-n}v_0, \psi^{+}(w_1)\cdots \psi^{+}(w_m)\psi^{-}(y_1)\cdots \psi^{-}(y_n)v_0\big\rangle =\sum_{\sigma \in \mathfrak{S}_{n}}(-1)^{\abs{\sigma}}
 \frac{y_{\sigma(n)}^{0}y_{\sigma(n-1)}\cdots y_{\sigma(m+1)}^{n-m-1}} {\prod\limits_{i=0}^{m-1}(w_{m-i}-y_{\sigma(i+1)})}.
\end{gather*}
Using Leibniz's formula to expand this as a determinant and then computing the determinant, we find that this is exactly equal to
\begin{gather*}
 \frac{ \prod\limits_{1\le i<j\le m}(w_i-w_j) \prod\limits_{1\le i<j\le n}(y_i-y_j)} { \prod\limits_{i=1}^{m}\prod\limits_{j=1}^{n}(w_i-y_j)}.
\end{gather*}
The proof in the case that $m\ge n$ is similar. Here, we need to argue that
\begin{gather*}
\big\langle Q^{m-n}v_0, \psi^{+}(w_1)\cdots \psi^{+}(w_m) \psi^{-}(y_1)\cdots \psi^{-}(y_n)v_0\big\rangle\\
\qquad{} = \sum_{\sigma \in \mathfrak{S}_{m}} (-1)^{\abs{\sigma}}
 \frac{w_{\sigma(1)}^{m-n-1}w_{\sigma(2)}^{m-n-2}\cdots w_{\sigma(m-n)}^{0}} { \prod\limits_{i=0}^{n-1}(w_{\sigma({m-i})}-y_{i+1})}.\tag*{\qed}
\end{gather*}\renewcommand{\qed}{}
\end{proof}

In Lemma \ref{LEMMA E.3} we see that this particular matrix element of fermion fields is the expansion of a rational function in the variables appearing in the fermion fields. This is not an accident, but is a basic property of vertex algebras, see \cite{MR996026, MR1142494,MR1651389}, referred as rationality of vertex operators. Indeed, the one-component fermionic Fock space $F^{(1)}$ is an example of a vertex algebra, and the fermionic fields are vertex operators for this vertex algebra structure. Another basic property of vertex algebras is called commutativity; roughly speaking it says that if we permute the vertex operators in a matrix element of a product of vertex operators the answer is again an expansion of the same rational function, but in a different region, up to a sign. See also~\cite{MR1653021}.

For instance we will also need the following matrix elements.
\begin{lem} \label{lem:6}
\begin{gather*}\left\langle Q^{m-n-1}v_0,\psi^{-}(z) \prod_{i=1}^{m}\psi^{+}(w_i) \prod_{i=1}^{n}\psi^{-}(y_i)v_0\right\rangle\\
\qquad {} = \frac{\prod\limits_{i=1}^{m} (z-y_{i}) \prod\limits_{1\le i<j\le m}(w_i-w_j) \prod\limits_{1\le i<j\le n}(y_i-y_j)}{\prod\limits_{i=1}^{n}(z-w_i) \prod\limits_{i=1}^{m} \prod\limits_{j=1}^{n}(w_i-y_j)}.
\end{gather*}
\end{lem}

One could derive this lemma from Lemma~\ref{LEMMA E.3} by commutativity of vertex operators, using the general theory of vertex algebras. For the convenience of the reader we give an elementary proof of this lemma, just using the commutation relations of fermion fields.

\begin{proof}Let $z, w_{1},w_{2},\dots,w_{m}$ be variables and define
\begin{gather*}
 W=\prod_{1\le j<k\le m}(w_{j}-w_{k}).
\end{gather*}
Consider the following rational function
\begin{gather*}
R=\frac{W}{\prod\limits_{1\le j\le m}(w_{j}-z)}.
\end{gather*}
The partial fraction expansion\footnote{Recall that if $f(z)$ is a rational function with a simple pole at $w$, then $f(z)=\frac{A}{w-z}+g(z)$, with $A$ equal to the value of $(w-z)f(z)$ at $z=w$.} of $R$ is
\begin{gather*}
R=\sum_{i=1}^{m}\frac{W_{i}}{w_{i}-z},
\end{gather*}
where
\begin{gather*}
W_{i}=\frac{W}{\prod\limits_{1\le j\le m \atop j\ne i}(w_{j}-w_{i})}=(-1)^{i+m}\prod\limits_{1\le j<k\le m \atop j,k\ne i}(w_{j}-w_{k}).
\end{gather*}
Instead of $R$ we can also consider the rational function
\begin{gather*}
 R^{\prime}=\frac{W}{\prod\limits_{1\le j\le m}(z-w_{j})}= (-1)^{m+1}\sum_{i=1}^{m}\frac{W_{i}}{z-w_{i}}.
\end{gather*}
Recall the formal series
\begin{gather*}
\delta(z,w)=\frac1{z-w}+\frac1{w-z}.
\end{gather*}
Here (and from now on) we use the convention that we expand in the second variable, so that for example $\frac1{z-w}=\sum\limits_{k=0}^{\infty}\frac{w^{k}}{z^{k+1}}$. In particular we will think of $R$ as a series in positive powers of $z$ and $R^{\prime}$ as a series in positive powers of the $w_{i}$s. Then we have the following identity:
\begin{gather*}
 R+(-1)^{m+1}R^{\prime}=\sum_{i=1}^{m}\delta(z,w_{i})W_{i},
\end{gather*}
or, writing out the definitions, multiplying by $(-1)^{m+1}$ and rearranging terms:
\begin{gather} \label{eq:1001}
\sum_{i=1}^{m}\delta(z,w_{i}) (-1)^{m+1}W_{i}+(-1)^{m} \frac{W }{\prod\limits_{j=1}^{m}(w_{j}-z)}= \frac{W}{\prod\limits_{i=1}^{m}(z-w_{i})}.
\end{gather}
Now, after this preparation, we turn to the matrix element, call it $A$, that we actually want to compute. By the fermion field commutation relation~\eqref{eq:20} and the previous Lemma~\ref{LEMMA E.3} we have for the matrix element~$A$
\begin{gather*}
A= \left\langle Q^{m-n-1}v_{0},\psi^{-}(z)\prod\limits_{i=1}^{m}\psi^{+}(w_{i}) \prod\limits_{s=1}^{n}\psi^{-}(y_{s})v_{0}\right\rangle\\
\hphantom{A}{} = \sum_{i=1}^{m}(-1)^{i+1}\delta(z,w_{i}) \left\langle Q^{m-n-1}v_{0}, \prod\limits_{j=1\atop j\ne i}^{m}\psi^{+}(w_{j}) \prod\limits_{s=1}^{n}\psi^{-}(y_{s})v_{0}\right\rangle\\
\hphantom{A=}{} + (-1)^{m}\left \langle Q^{m-n-1}v_{0}, \prod\limits_{j=1}^{m}\psi^{+}(w_{i}) \psi^{-}(z) \prod\limits_{s=1}^{n}\psi^{-}(y_{s})v_{0}\right\rangle\\
\hphantom{A}{} = \sum_{i=1}^{m}\delta(z,w_{i}) \frac{(-1)^{m+1}W_{i}} {\prod\limits_{l=1\atop l\ne i}^{m} \prod\limits_{s=1}^{n}(w_{l}-y_{s})} \cdot \prod\limits_{1\le s<t\le n}(y_{s}-y_{t})\\
\hphantom{A=}{} + (-1)^{m}\frac{W}{\prod\limits_{i=1}^{m}(w_{i}-z)}\cdot \frac{\prod\limits_{s=1}^{n}(z-y_{s}) \prod\limits_{1\le s<t\le n}(y_{s}-y_{t})} {\prod\limits_{i=1}^{m}\prod\limits_{s=1}^{n}(w_{i}-y_{s})}.
\end{gather*}
Now
\begin{gather*}
\delta(z,w_{i}) \frac{W_{i}}{\prod\limits_{l=1\atop l\ne i}^{m}\prod\limits_{s=1}^{n}(w_{l}-y_{s})}= \delta(z,w_{i}) \prod\limits_{s=1}^{n}(w_{i}-y_{s}) \frac{W_{i}} {\prod\limits_{l=1}^{m}\prod\limits_{s=1}^{n}(w_{l}-y_{s})}\\
\hphantom{\delta(z,w_{i}) \frac{W_{i}}{\prod\limits_{l=1\atop l\ne i}^{m}\prod\limits_{s=1}^{n}(w_{l}-y_{s})}}{} =
 \delta(z,w_{i}) \prod\limits_{s=1}^{n}(z-y_{s}) \frac{W_{i}} {\prod\limits_{l=1}^{m}\prod\limits_{s=1}^{n}(w_{l}-y_{s})}.
 \end{gather*}
 Hence, using \eqref{eq:1001},
 \begin{gather*}
 A =\left[\sum_{i=1}^{m}\delta(z,w_{i})(-1)^{m+1}W_{i}+(-1)^{m}\frac{W}{\prod\limits_{i=1}^{m}(w_{i}-z)} \right] \frac{\prod\limits_{s=1}^{n}(z-y_{s}) \prod\limits_{1\le s<t\le n} (y_{s}-y_{t})}
 {\prod\limits_{i=1}^{m} \prod\limits_{s=1}^{n} (w_{i}-y_{s})}\\
\hphantom{A}{}= \frac{W}{\prod\limits_{i=1}^{m}(z-w_{i})} \cdot \frac{\prod\limits_{s=1}^{n}(z-y_{s}) \prod\limits_{1\le s<t\le n} (y_{s}-y_{t})} {\prod\limits_{i=1}^{m} \prod\limits_{s=1}^{n} (w_{i}-y_{s})},
 \end{gather*} which is what we wanted to show.
\end{proof}

\subsection*{Acknowledgements}
The authors gratefully acknowledge travel support from the Simons Foundation, Collaboration Grant 245048. Addabbo expresses thanks for support from Dr.~Lois M.~Lackner Mathematics Fellowships, the University of Illinois Research Board, and the Associate Alumnae of Douglass College. The authors also thank Philippe Di Francesco and Rinat Kedem for helpful conversations, and anonymous referees for many helpful comments.

\pdfbookmark[1]{References}{ref}
\LastPageEnding

\end{document}